\documentclass[DIV=12]{scrartcl}

\usepackage{amsmath}
\usepackage{amssymb}
\usepackage{amsthm}
\usepackage{cite}
\usepackage[T1]{fontenc}
\usepackage{hyperref}
\usepackage{lmodern}
\usepackage{mathscinet}
\usepackage{mathtools}
\usepackage{microtype}
\usepackage{pmat}
\usepackage[automark]{scrpage2}%
\usepackage{url}

\usepackage{174arxiv}

 \newtheorem{thm}{Theorem}[section]
 \newtheorem{cor}[thm]{Corollary}
 \newtheorem{lem}[thm]{Lemma}
 \newtheorem{prop}[thm]{Proposition}
\theoremstyle{definition}
 \newtheorem{defn}[thm]{Definition}
 \newtheorem{nota}[thm]{Notation}
 \newtheorem*{prob}{Problem}
\theoremstyle{remark}
 \newtheorem{rem}[thm]{Remark}
 \newtheorem{exa}[thm]{Example}

\mathtoolsset{mathic=true}
\numberwithin{equation}{section}
\title{On the truncated matricial Stieltjes moment problem \mproblem{\rhl}{m}{\leq}}
\author{Bernd Fritzsche \and Bernd Kirstein \and Conrad M\"adler \and Torsten Schr\"oder}

\begin{document}
\maketitle

\begin{abstract}
 This paper gives via Stieltjes transform a complete description of the solution set of a matricial truncated Stieltjes-type power moment problem in the non-degenerate and degenerate cases. The approach is based on the Schur type algorithm which was worked out in the papers~\zitas{MR3611479,MR3611471}.
 Furthermore, the subset of parameters is determined which corresponds to another truncated matricial Stieltjes-type moment problem.
\end{abstract}

\begin{description}
 \item[Keywords:] Stieltjes moment problem, Schur type algorithm, Stieltjes pairs.
\end{description}

\section{Introduction}\label{S1417}
 This paper is closely related to~\zitas{MR3611479,MR3611471}.
 The main goal is to achieve a simultaneous treatment of the even and odd cases of a further truncated matricial Stieltjes moment problem, which is related but different from that truncated matricial Stieltjes moment problem which was studied in~\zitas{MR3611479,MR3611471} on the basis of Schur analysis methods.
 We will demonstrate that appropriate modifications of our Schur analysis conceptions lead to a complete solution to the problem under consideration in the non-degenerate and degenerate cases.
 Our conception in~\zitas{MR3611479,MR3611471} was based on working out two interrelated versions of Schur type algorithms, namely an algebraic one and a function-theoretic one, and then using the interplay between both of them.
 This strategy stands again in the center of our investigations.
  
 In order to describe more concretely the central topics studied in this paper, we introduce some notation. Throughout this paper, let \(p\) and \(q\) be positive integers. Let \symba{\N}{n}, \symba{\NO}{n}, \symba{\Z}{z}, \symba{\R}{r}, and  \symba{\C}{c} be the set of all positive integers, the set of all \tnn{} integers, the set of all integers, the set of all real numbers, and the set of all complex numbers, respectively. For every choice of \(\rho, \kappa \in \R \cup\{-\infty,\infp\}\), let \symba{\mn{\rho}{\kappa}\defeq \setaa{k\in\Z}{\rho\leq k \leq \kappa}}{z}. We will write \symba{\Cpq}{c}, \symba{\CHq}{c}, \symba{\Cggq}{c}, and \symba{\Cgq}{c} for the set of all complex \tpqa{matrices}, the set of all \tH{} complex \tqqa{matrices}, the set of all \tnnH{} complex \tqqa{matrices}, and the set of all \tpH{} complex \tqqa{matrices}, respectively. 

 We will use \symba{\BorR}{b} to denote the \(\sigma\)\nobreakdash-algebra of all Borel subsets of \(\R\).
 For each \(\Omega \in \BorR\setminus\set{\emptyset}\), let \symba{\BorO  \defeq \BorR\cap \Omega}{b}. Furthermore, for each \(\Omega \in \BorR\setminus\set{\emptyset}\), we will write \(\MggqO\) to designate the set of all \tnnH{} \tqqa{measures} defined on \(\BorO\), \ie{}, the set of all \(\sigma\)\nobreakdash-additive mappings \(\mu\colon\BorO\to \Cggq\). We will use the integration theory with respect to \tnnH{} \tqqa{measures}, which was worked out independently by I.~S.~Kats~\zita{MR0080280} and M.~Rosenberg~\zita{MR0163346}. Some features of this theory are sketched in \rapp{Appendix_Int}. For every choice of \(\Omega \in \BorR\setminus\set{\emptyset}\) and \(\kappa\in \NOinf\), we will use \symba{\MggquO{\kappa}}{m} to denote the set of all \(\sigma \in \MggqO\) such that the integral\index{s@\(\suo{j}{\sigma}\)}
 \begin{equation}\label{F1*1}
  \suo{j}{\sigma}
  \defeq \int_\Omega x^j \sigma(\dif x)
 \end{equation}
 exists for all \(j\in\mn{0}{\kappa}\).

\begin{rem}\label{K1B} 
 Let \(\Omega \in \BorR  \setminus \set{\emptyset}\), let \(\kappa \in \NOinf \), and let \(\sigma \in \MggquO{\kappa}\).
 In view of \eqref{F1*1}, then one can easily check that \(\rk{ \suo{j}{\sigma} }^\ad = \suo{j}{\sigma}\) holds true for all \(k \in \mn{0}{\kappa}\).
\end{rem}
 
\breml{R1715}
 If \(k,\ell\in\NO\) with \(k<\ell\), then \(\MggquO{\ell}\subseteq\MggquO{k}\) holds true.
\erem

 The central problem studied in this paper is formulated as follows:\index{m@\mproblem{\rhl}{m}{\leq}}
\begin{prob}[\mproblem{\rhl}{m}{\leq}]
 Let \(\ug\in\R\), let \(m\in\NO\) and let \(\seqs{m}\) be a sequence of complex \tqqa{matrices}.
 Parametrize the set \symba{\MggqKskg{m}}{m} of all \(\sigma\in\MggquK{m}\) for which the matrix \(\su{m}-\suo{m}{\sigma}\) is \tnnH{} and, in the case \(m\geq1\), moreover \(\suo{j}{\sigma}=\su{j}\) is satisfied for each \(j\in\mn{0}{m-1}\).
\end{prob}

 The papers~\zitas{MR3611479,MR3611471} were concerned with the study of the following question:\index{m@\mproblem{\rhl}{m}{=}}
\begin{prob}[\mproblem{\rhl}{m}{=}]
 Let \(\ug\in\R\), let \(m\in \NO\), and let \(\seqs{m}\) be a sequence of complex \tqqa{matrices}.
 Parametrize the set \symba{\MggqKsg{m}}{m} of all \(\sigma \in \MggquK{m}\) for which  \(\suo{j}{\sigma}=\su{j}\) is fulfilled for all \(j \in \mn{0}{m}\).
\end{prob}

 A closer look at the just formulated two problems leads to the following more or less obvious observations on interrelations between the two problems:
 
\bremal{R1445}
 Let \(\ug\in\R\), let \(m\in\NO\) and let  \(\seqs{m}\) be a sequence of complex \tqqa{matrices}.
 Then
 \(
  \MggqKsg{m}
  \subseteq\MggqKskg{m}
 \).
\erema

\bremal{R1721}
 Let \(\ug\in\R\), let \(m\in\N\), and let  \(\seqs{m}\) be a sequence of complex \tqqa{matrices}.
 For all \(\ell\in\mn{0}{m-1}\), then
 \(
  \MggqKskg{m}
  \subseteq\MggqKsg{\ell}
 \)
\erema

 In the case that a sequence \(\seqs{m}\) of complex \tqqa{matrices} is given for which the set \(\MggqKsg{m}\) is non-empty, we obtained in~\zitaa{MR3611471}{\cthm{13.1}} a complete parametrization of this set via  a linear fractional transformation of matrices the generating function of which is a \taaa{2q}{2q}{matrix} polynomial built from the sequence \(\seqs{m}\) of the given original data.
 The set of parameters is chosen as a particular class of \tqqa{matrix-valued} functions which are holomorphic in the domain \(\Cs\).
 Thus, combining this with \rrem{R1445}, we were encouraged to adopt the approach from~\zita{MR2038751} in order to parametrize the set \(\MggqKskg{m}\) via a linear fractional transformation with the same \taaa{2q}{2q}{matrix} polynomial as generating function.
 However, the class of parameter functions has to be appropriately extended.
 
 The realization of this idea determines the basic strategy of this paper.
 In~\zitas{MR3611479,MR3611471}, we presented a Schur analysis approach to Problem~\mproblem{\rhl}{m}{=}.
 In this paper, we will indicate that our method can be appropriately modified to produce a complete description of the set \(\MggqKskg{m}\) in the non-degenerate and degenerate cases:

 In order to state a necessary and sufficient condition for the solvability of each of the above formulated moment problems, we have to recall the notion of two types of sequences of matrices. If \(n\in \NO\) and if \(\seqs{2n}\) is a sequence of complex \tqqa{matrices}, then \(\seqs{2n}\) is called \notii{\tHnnd{}} (respectively,  \notii{\tHpd{}}) if the block Hankel matrix\index{h@\(\Hu{n}\)}
\[
 \Hu{n}
 \defeq \matauuo{\su{j+k}}{j,k}{0}{n}
\]
 is \tnnH{} (respectively, \tpH{}).
 A sequence \(\seqsinf \) of complex \tqqa{matrices} is called \notii{\tHnnd{}} (respectively,  \notii{\tHpd{}}) if \(\seqs{2n}\) is \tHnnd{} (respectively, \tHpd{}) for all \(n\in\NO\).
 For all \(\kappa\in\NOinf \), we will write \symba{\Hggqu{2\kappa}}{h} (respectively, \symba{\Hgqu{2\kappa}}{h}) for the set of all \tHnnd{} (respectively, \tHpd{}) sequences \(\seqs{2\kappa}\) of complex \tqqa{matrices}.
 Furthermore, for all \(n \in \NO\), let \symba{\Hggequ{2n}}{h} be the set of all sequences \(\seqs{2n}\) of complex \tqqa{matrices} for which there exist complex \tqqa{matrices} \(\su{2n+1}\) and \(\su{2n+2}\) such that \(\seqs{2(n+1)} \in\Hggqu{2(n+1)}\), whereas \symba{\Hggequ{2n+1}}{h} stands for the set of all sequences \(\seqs{2n+1}\) of complex \tqqa{matrices} for which there exist some \(\su{2n+2}\in\Cqq\) such that \(\seqs{2(n+1)} \in\Hggqu{2(n+1)}\).
 For each \(m\in\NO\), the elements of the set \(\Hggequ{m}\) are called \emph{\tHnnde{} sequences}. For technical reason, we set \(\Hggeqinf\defeq\Hggqinf\)\index{h@\(\Hggeqinf\)}.
 
 Besides the just introduced classes of sequences of complex \tqqa{matrices}, we need analogous classes of sequences \(\seqs{\kappa}\) of complex \tqqa{matrices}, which take into account the influence of the prescribed number \(\ug\in\R\). 
 We will introduce several classes of finite or infinite sequences of complex \tqqa{matrices} which are characterized by properties of the sequences \(\seqs{\kappa}\) and \(\seq{-\ug\su{j}+\su{j+1}}{j}{0}{\kappa-1}\).
 
 Let \(\seqs{\kappa}\) be a sequence of complex \tpqa{matrices}. Then, for all \(n\in\NO\) with \(2n+1\leq\kappa\), we introduce the block Hankel matrix\index{k@\(\Ku{n}\)}
\[
 \Ku{n}
 \defeq\matauuo{\su{j+k+1}}{j,k}{0}{n}.
\]
 Let \(\ug\in\R\).
 Let \(\Kggqua{0}\defeq\Hggqu{0}\)\index{k@\(\Kggqua{0}\)}, and, for all \(n\in\N\), let \symba{\Kggqua{2n}}{k} be the set of all sequences \(\seqs{2n}\) of complex \tqqa{matrices} for which the block Hankel  matrices \(\Hu{n}\) and \(-\ug\Hu{n-1}+\Ku{n-1}\) are both \tnnH{}, \ie{},
\beql{Kgg2n}
 \Kggqua{2n}
 =\setaa*{\seqs{2n}\in\Hggqu{2n}}{\seq{-\ug\su{j}+\su{j+1}}{j}{0}{2(n-1)}\in\Hggqu{2(n-1)}}.
\eeq
 Furthermore, for all \(n\in\NO\), let \symb{\Kggqua{2n+1}} be the set of all sequences \(\seqs{2n+1}\) of complex \tqqa{matrices} for which the block Hankel  matrices \(\Hu{n}\) and \(-\ug\Hu{n}+\Ku{n}\) are both \tnnH{}, \ie{}, if \symba{\setseqq{2n+1}}{f} is the set of all sequences \(\seqs{2n+1}\) of complex \tqqa{matrices}, then
\beql{Kgg2n+1}
 \Kggqua{2n+1}
 \defeq\setaa*{\seqs{2n+1}\in\setseqq{2n+1}}{\set*{\seqs{2n},\seq{-\ug\su{j}+\su{j+1}}{j}{0}{2n}}\subseteq\Hggqu{2n}}.
\eeq

\bremal{R1459}
 Let \(\ug\in\R\), let \(m\in\NO\), and let \(\seqs{m}\in\Kggq{m}\).
 Then it is easily checked that \(\seqs{\ell}\in\Kggq{\ell}\) for all \(\ell\in\mn{0}{m}\).
\erema

 Let \(\ug\in\R\).
 In view of \rrem{R1459}, let \symba{\Kggqinfa}{k} be the set of all sequences \(\seqsinf\) of complex \tqqa{matrices} such that \(\seqs{m}\in\Kggqua{m}\) for all \(m\in\NO\). Formulas \eqref{Kgg2n} and \eqref{Kgg2n+1} show that the sets \(\Kggqua{2n}\) and \(\Kggqua{2n+1}\) are determined by two conditions. The condition \(\seqs{2n}\in\Hggqu{2n}\) ensures that a particular Hamburger  moment problem associated with the sequence \(\seqs{2n}\) is solvable (see, \eg{}~\zitaa{MR2570113}{\cthm{4.16}}). The second condition \(\seq{-\ug\su{j}+\su{j+1}}{j}{0}{2(n-1)}\in\Hggqu{2(n-1)}\) (respectively \(\seq{-\ug\su{j}+\su{j+1}}{j}{0}{2n}\in\Hggqu{2n}\)) controls that the original sequences \(\seqs{2n}\) and \(\seqs{2n+1}\) are well adapted to the interval \(\rhl \).
 
 For each \(m\in\NO\), let \symba{\Kggequa{m}}{k} be the set of all sequences \(\seqs{m}\) of complex \tqqa{matrices} for which there exists an \(\su{m+1}\in\Cqq\) such that \(\seqs{m+1}\) belongs to \(\Kggqua{m+1}\).

\bremal{R1502}
 Let \(\ug\in\R\), let \(m\in\NO\), and let \(\seqs{m}\in\Kggeq{m}\).
 Then \(\seqs{\ell}\in\Kggq{\ell}\) for all \(\ell\in\mn{0}{m}\).
\erema

\bremal{R1506}
 Let \(\ug\in\R\), let \(\kappa\in\Ninf\), and let \(\seqs{\kappa}\in\Kggq{\kappa}\).
 Then \(\seqs{\ell}\in\Kggeq{\ell}\) for all \(\ell\in\mn{0}{\kappa-1}\).
\erema

 Let \(m\in\NO\).
 Then we call a sequence \(\seqs{m}\) of complex \tqqa{matrices} \notii{\traSnnd{\ug}} if it belongs to \(\Kggq{m}\) and \notii{\traSnnde{\ug}} if it belongs to \(\Kggeq{m}\). 
 
 Now we can characterize the situations in which the problems formulated above have a solution:
 
\bthmnl{\zitaa{MR2735313}{\cthmss{1.3}{1.4}}}{T1*3+2}
 Let \(\ug\in\R\), let \(m\in\NO\), and let \(\seqs{m}\) be a sequence of complex \tqqa{matrices}. Then:
 \benui
  \il{T1*3+2.a} \(\MggqKsg{m}\neq\emptyset\) if and only if \(\seqs{m}\in\Kggequa{m}\).
  \il{T1*3+2.b} \(\MggqKskg{m}\neq\emptyset\) if and only if \(\seqs{m}\in\Kggqua{m}\).
 \eenui
\ethm

\bcorol{C1022}
 Let \(\ug\in\R\), let \(m\in\NO\), and let \(\sigma\in\MggquK{m}\). Then \(\seq{\suo{j}{\sigma}}{j}{0}{m}\) belongs to \(\Kggeq{m}\).
\ecoro
\bproof
 Apply \rthmp{T1*3+2}{T1*3+2.a}.
\eproof

 \rthm{T1*3+2} indicates the importance of the sets \(\Kggqua{m}\) and \(\Kggequa{m}\) for the above formulated truncated matricial Stieltjes-type moment problems.
 The following result reflects an interrelation between these two sets, which strongly influences our following considerations:
 
\bthmnl{\zitaa{MR2735313}{\cthm{5.2}}}{T1605}
 Let \(\ug\in\R\), let \(m\in\NO\), and let \(\seqs{m}\in\Kggqua{m}\).
 Then there is a unique sequence \(\eqseq{s}{m}\in\Kggequa{m}\) such that
\beql{T1605.A1}
  \MggqAAkg{m}{\rhl}{\seqs{m}}
  =\MggqAAkg{m}{\rhl}{\eqseq{s}{m}}.
\eeq
\ethm

 \rthm{T1605} leads us to the following notion:

\bdefnl{D1241}
 If \(m \in \NO \) and if \(\seqs{m} \in \Kggqua{m}\), then the unique sequence \symba{\eqseq{s}{m}}{.} belonging to \(\Kggequa{m}\) for which \eqref{T1605.A1} holds true is said to be the \notii{\teqseq{\seqs{m}}}.
\edefn

\bleml{L0709}
 Let \(\ug\in\R\), let \(m\in\N\), and let \(\seqs{m}\in\Kggq{m}\).
 Denote by \(\eqseq{s}{m}\) the \teqseq{\seqs{m}}.
 Then \(\su{m}-\eqseqmat{s}{m}\in\Cggq\).
 If \(m\geq1\), then \(\su{j}=\eqseqmat{s}{j}\) for all \(j\in\mn{0}{m-1}\).
 Moreover, \(\seqs{m}=\eqseq{s}{m}\) if and only if \(\seqs{m}\in\Kggeq{m}\).
\elem
\bproof
 By construction, we have
\(
 \eqseq{s}{m}
 \in\Kggeq{m}
\)
 and \eqref{T1605.A1}.
 Then we infer from \rthmp{T1*3+2}{T1*3+2.a} that \(\Mggqaag{m}{\rhl}{\eqseq{s}{m}}\neq\emptyset\).
 Let
\(
 \sigma
 \in\Mggqaag{m}{\rhl}{\eqseq{s}{m}}
\).
 This implies \(\suo{m}{\sigma}=\eqseqmat{s}{m}\).
 Consequently, from \rrem{R1445} it follows
\(
 \sigma
 \in\Mggqaakg{m}{\rhl}{\eqseq{s}{m}}
\).
 Hence, \eqref{T1605.A1} implies
\(
 \sigma\in\MggqKskg{m}
\).
 Thus, \(\su{m}-\suo{m}{\sigma}\in\Cggq\).
 Because of \(\suo{m}{\sigma}=\eqseqmat{s}{m}\),  we get \(\su{m}-\eqseqmat{s}{m}\in\Cggq\).
 The remaining assertions are immediate consequences of \rthm{T1605}.
\eproof

 \rthm{T1605} is essential for the realization of the above formulated basic strategy of our approach, because it is namely possible to restrict our considerations to the case that the given sequence \(\seqs{m}\) belongs to the subclass \(\Kggeq{m}\) of \(\Kggq{m}\).
 This provides the opportunity to use immediately the basics of the machinery developed in~\zitas{MR3611479,MR3611471}.
 
 Similar as in~\zitas{MR3611479,MR3611471}, we reformulate the original truncated matricial moment problem via Stieltjes transform into an equivalent problem of prescribed asymptotic expansions for particular classes of matrix-valued functions, which are holomorphic in \(\Cs\).
 The key for the success of our approach is caused by the fact that the Schur-Stieltjes transform for \tqqa{matrix-valued} holomorphic functions in \(\Cs\), which we worked out in~\zita{MR3611471}, is also compatible with the problem under consideration in this paper.
 
 This paper is organized as follows.
 In \rsec{S1331}, we recall some aspects of the class \(\Kggeq{m}\) of \traSnnde{\ug} sequences and some of its subclasses.

 In \rsubsec{S0850}, we view the first \tascht{\ug} of a sequence \(\seqs{\kappa}\) from \(\Cpq\) introduced in~\zitaa{MR3611479}{\cdefn{7.1}} under the aspect of this paper.
 \rlem{T149_4} indicates that the \tascht{\ug} has the desired behavior. 
 \rsubsec{S1549} is from its character similar to \rsec{S0850}.
 It will be demonstrated that the inverse \tascht{\ug}ation of sequences of matrices introduced in~\zitaa{MR3611479}{\cdefn{10.1}} stands in full harmony with the aims of this paper.
 For our considerations, it is very useful that \rlem{T2110_1} yields an affirmative answer concerning the realization of our aims.
 We also summarize some essential features of the Schur type algorithm for finite or infinite sequences of complex \tpqa{matrices}, which was constructed in~\zitaa{MR3611479}{\csecss{8}{9}}.
 
 In \rsec{S1104}, we recall some basic facts on some classes of matrix-valued holomorphic functions.
 The scalar versions of these classes had been proved to be essential tools for studying classical moment problems (see, \eg{}~\zita{MR0458081}).

 In \rsec{S*5}, we summarize some basic facts on the classes \(\SFqas{\kappa}\) of \taSt{s} of the measures belonging to \(\MggqKsg{\kappa}\).
 
 In \rsec{S0509}, we translate the original moment problem~\mproblem{\rhl}{m}{\leq} via \taSt{ation} into the language of a particular class of holomorphic matrix-valued functions, namely the class \(\SFqa\).
 
 \rsec{S1253} is written against to a special background.
 Indeed, there arises a new phenomenon in comparison with our approach to the moment problem \mproblem{\rhl}{m}{=}, which was undertaken in~\zitas{MR3611479,MR3611471}.
 As we will see later, in contrast to~\zitaa{MR3611471}{\cthm{13.1}}, the set of all \taSt{s} of the solution set of the moment problem \mproblem{\rhl}{m}{\leq} can not be parametrized by a linear fractional transformation, where the set of parameters is given by an appropriate subclass of \(\SFqa\).
 Now we will be confronted with a particular class of ordered pairs of \tqqa{matrix-valued} functions which are meromorphic in \(\Cs\).
 For this class of ordered pairs, essential features are given.
 
 \rsecsss{S*8}{S0815}{S0559} lie in the heart of this paper.
 In the center of these sections stands the function-theoretic version of the Schur algorithm which was the basic tool of our approach to \rprob{\mproblem{\rhl}{m}{=}} chosen in~\zitas{MR3611479,MR3611471}.
 Now we demonstrate that this algorithm is also suitable to handle \rprob{\mproblem{\rhl}{m}{\leq}}.
 The main reason for this is that we already know how this Schur algorithm acts in the framework of \rprob{\mproblem{\rhl}{m}{=}}.
 Then we can apply the corresponding results (see \rthmss{T0837}{T10*9}), which provided key instruments for the successful realization of our strategy chosen in~\zitas{MR3611479,MR3611471}.
 It should be mentioned that the proofs of \rthmss{T0837}{T10*9}, presented in~\zita{MR3611471}, are based on matricial versions of the classical Hamburger-Nevanlinna theorem (see~\zitaa{MR3611471}{\cthm{6.1}}).
 The central result of \rsec{S0815} is \rthm{T149_5}, which indicates that the \taSSt{\su{0}} \(\STao{F}{\su{0}}\) introduced in \eqref{F8*1} well behaves for functions \(F\in\SFqaskg{m}\).
 The central result of \rsec{S0559} is \rthm{T0904}, which describes the behavior of the \tiaSSt{\su{0}} \(\STiao{F}{\su{0}}\) for functions \(F\) belonging to a special class.
 The generic application of \rthm{T0904} is \rthm{T1140}, where \(\ug\in\R\), \(m\in\N\), and a sequence \(\seqs{m}\in\Kggeq{m}\) with \tseinsalpha{} \(\seq{\su{j}^\seinsalpha}{j}{0}{m-1}\) are given.
 Then \rthm{T1140} tells us that, for each function \(F\in\SFuqaakg{m-1}{\seq{\su{j}^\seinsalpha}{j}{0}{m-1}}\), the \tiaSSt{\su{0}} \(\STiao{F}{\su{0}}\) belongs to \(\SFqaskg{m}\).
 From \rthm{T1140} it can be seen why our approach to handle \rprob{\mproblem{\rhl}{m}{\leq}} is based on the assumption that the sequence \(\seqs{m}\) is \traSnnde{\ug}.
 
 \rsec{S1351} occupies an exceptional position in this paper.
 More precisely, we determine the set \(\SFqaskg{0}\).
 This leads us to the solution of a matrix inequality for a class of holomorphic \tqqa{matrix-valued} functions \(F\) in the open upper half plane \(\ohe\) which has \tnnH{} imaginary part in each point \(w\in\ohe\).
 Another phenomenon, which is completely new in comparison with the papers~\zitas{MR3611479,MR3611471} where \rprob{\mproblem{\rhl}{m}{=}} was treated, is the fact, that we are led to (equivalence classes of) ordered pairs of meromorphic \tqqa{matrix-valued} functions in \(\Cs\) as parameters in the linear fractional transformation.
 
 \rsec{S1302} contains a complete description of the set \(\SFqaskg{m}\) (see \rthm{T0945}).
 The method to prove \rthm{T0945} is to apply \rsec{S1351}, where the case \(m=0\) was already handled and then proceed by induction.
 The success of the induction step is then realized by the application of \rthmss{T149_5}{T1140}.
 Later the two ``extremal'' cases for the set \(\SFqaskg{m}\) are treated separately in more detail (see \rthmss{T1206}{T1227}).
 
 In \rsec{S1314}, we prove a parametrization of the matricial truncated Stieltjes power moment problem \mproblem{\rhl}{m}{\leq} in the general case.
 We treat the non-degenerate case, the completely degenerate case, and the degenerate but not completely degenerate case separately.
 In the completely degenerate  case, we see that the problem has a unique solution.

 If \(\ug\in\R\), \(m\in\NO\), and a sequence \(\seqs{m}\in\Kggeq{m}\) are given then the question arises:
 Which subset of the parameter set of the solution of \rprob{\mproblem{\rhl}{m}{\leq}} generates all solutions of \rprob{\mproblem{\rhl}{m}{=}}?
 An answer to this question is given in \rsec{S1144}.

 In several appendices, we summarize facts from matrix theory, integration theory with respect to \tnnH{} measures, meromorphic matrix-valued functions, and linear fractional transformations of matrices.
 Of particular importance is \rapp{A1557} which contains a detailed investigation of that pair of \taaa{2q}{2q}{matrix} polynomials which provides the elementary factors for the description of the function-theoretic version of our Schur algorithm.

\section{On further classes of sequences of complex \tqqa{matrices}}\label{S1331}
 In this section, we introduce two classes of finite sequences of complex \tqqa{matrices}, which prove to be \traSnnde{\ug}.
 We start with some notation.
 By \symba{\Ip}{\iu} and \symba{\Opq}{0} we designate the unit matrix in \(\Cpp\) and the null matrix in \(\Cpq\), respectively.
 If the size of a unit matrix and a null matrix is obvious, then we will also omit the indexes.
 For each \(A\in\Cqq\), let \symba{\tr A}{t} be the trace of \(A\) and let \symba{\det A}{d} be the determinant of \(A\).
 If \(A\in\Cpq\), then we denote by \symba{\nul{A}}{n} and \symba{\ran{A}}{r} the null space of \(A\) and the column space of \(A\), respectively, and we will use \symba{\rank A}{r} and \symba{\normS{A}}{.} to denote the rank of \(A\) and the operator norm of \(A\), respectively.
 For every choice of \(x,y\in\Cq\), the notation \symba{\inner{x}{y}}{.} stands for the (left) Euclidean inner product.
 For each \(A\in \Cpq\), let \(\normF{A}\defeq\sqrt{\tr\rk{A^\ad A}}\)\index{\(\normF{A}\)} be the Euclidean norm of \(A\).
 If \(\mathcal{M}\) is a non-empty subset of \(\Cq\), then \symba{\mathcal{M}^\bot}{.} stands for the (left) orthogonal complement of \(\mathcal{M}\).
 If \(\mathcal{U}\) is a linear subspace of \(\Cq\), then let \symba{\OPu{\mathcal{U}}}{p} be the orthogonal projection matrix onto \(\mathcal{U}\), \ie{}, \(\OPu{\mathcal{U}}\) is the unique complex \tqqa{matrix} \(P\) that fulfills the three conditions \(P^2=P\), \(P^\ad=P\), and \(\ran{P}=\mathcal{U}\).
 We will often use the Moore-Penrose inverse of a complex \tpqa{matrix} \(A\).
 This is the unique complex \tqpa{matrix} \(X\) such that the four equations \(AXA=A\), \(XAX=X\), \(\rk{AX}^\ad=AX\), and \(\rk{XA}^\ad=XA\) hold true (see, \eg{}~\zitaa{MR1152328}{\cprop{1.1.1}}).
 As usual, we will write \symba{A^\mpi}{.} for this matrix \(X\).
 If \(n\in\N\), if \(\seq{p_j}{j}{1}{n}\) is a sequence of positive integers, and if \(A_j\in\Coo{p_j}{q}\) for all \(j\in\mn{1}{n}\), then let\index{c@\(\col\seq{A_j}{j}{1}{n}\)}
\[
 \col\seq{A_j}{j}{1}{n}
 \defg
 \bCol
  A_1\\
  A_2\\
  \vdots\\
  A_n
 \eCol.
\]

  We use the L\"owner semi-ordering in \(\CHq\), \ie{}, we write \(A\geq B\)\index{\(A\geq B\)} or \symb{B\leq A} in order to indicate that \(A\) and \(B\) are \tH{} complex matrices such that the matrix \(A-B\) is \tnnH{}.

 Let \(\kappa\in\NOinf \) and let \(\seqska \) be a sequence of complex \tpqa{matrices}. 
 We will associate with \(\seqska \) several matrices, which we will often need in our subsequent considerations: 
 For all \(l,m\in\NO\) with \(l\leq m\leq\kappa\), let
\begin{align}\label{yz}
 \yuuo{l}{m}{s}&\defg\col\seq{\su{j}}{j}{l}{m}&
 &\text{and}&
 \zuuo{l}{m}{s}&\defg\mat{\su{l},\su{l+1},\dotsc,\su{m}}.
\end{align}
 \index{y@\(\yuuo{l}{m}{s}\), \(\yuu{l}{m}\)}\index{z@\(\zuuo{l}{m}{s}\), \(\zuu{l}{m}\)}Let\index{h@\(\Huo{n}{s}\), \(\Hu{n}\)}\index{k@\(\Kuo{n}{s}\), \(\Ku{n}\)}\index{g@\(\Guo{n}{s}\), \(\Gu{n}\)}
\begin{align}
 \Huo{n}{s}&\defg\matauuo{\su{j+k}}{j,k}{0}{n}&\text{for all }n&\in\NO\text{ with }2n\leq\kappa,\label{Hs}\\
 \Kuo{n}{s}&\defg\matauuo{\su{j+k+1}}{j,k}{0}{n}&\text{for all }n&\in\NO\text{ with }2n+1\leq\kappa.\label{Ks}
\end{align}
Let\index{l@\(\Luo{n}{s}\), \(\Lu{n}\)}\index{l@\(\LLuo{n}{s}\), \(\LLu{n}\)}
\begin{align}\label{L}
 \Luo{0}{s}&\defg\su{0}&
&\text{and let}&
 \Luo{n}{s}&\defg\su{2n}-\zuuo{n}{2n-1}{s}(\Huo{n-1}{s})^\mpi\yuuo{n}{2n-1}{s}
\end{align}
for all \(n\in\N\) with \(2n\leq\kappa\). Let\index{t@\(\Thetauo{n}{s}\), \(\Trip{n}\)}
\begin{align}\label{The}
 \Thetauo{0}{s}&\defg\Opq&
 &\text{and let}&
 \Thetauo{n}{s}&\defg\zuuo{n}{2n-1}{s}(\Huo{n-1}{s})^\mpi\yuuo{n}{2n-1}{s}
\end{align}
 for all \(n\in\N\) with \(2n-1\leq\kappa\).
 In situations in which it is obvious which sequence \(\seqska \) of complex matrices is meant, we will also write \(\yuu{l}{m}\), \(\zuu{l}{m}\), \(\Hu{n}\), \(\Ku{n}\), \(\Lu{n}\), and  \(\Trip{n}\) instead of \(\yuuo{j}{k}{s}\), \(\zuuo{j}{k}{s}\), \(\Huo{n}{s}\), \(\Kuo{n}{s}\), \(\Luo{n}{s}\), and \(\Thetauo{n}{s}\), respectively.

 Let \(\ug\in\C\) and let \(\kappa\in\Ninf\).
 Then the sequence \(\seq{v_j}{j}{0}{\kappa-1}\) given by
\begin{align}\label{vsa}
 v_j&\defg\sau{j}&
&\text{and}&
 \sau{j}&\defeq-\ug\su{j}+\su{j+1}
\end{align}
 for all \(j\in\mn{0}{\kappa-1}\) plays a key role in our following considerations.
 We define
 \index{t@\(\Tripa{n} \defg\Thetauo{n}{v}\)}
 \index{h@\(\Hau{n}\)}
 \index{l@\(\Lau{n}\)}
 \index{k@\(\Kau{n}\)}  
\begin{align}
 \Tripa{n}&\defg\Thetauo{n}{v}&&&&&\text{for all }n&\in\NO\text{ with }2n\leq\kappa,\label{Ta}\\
 \Hau{n}&\defg\Huo{n}{v},&
 &\text{and}&
 \Lau{n}&\defg\Luo{n}{v}&\text{for all }n&\in\NO\text{ with }2n+1\leq\kappa,\label{Ha}\\
 \Kau{n}&\defg\Kuo{n}{v}&&&&&\text{for all }n&\in\NO\text{ with }2n+2\leq\kappa,\label{Ka}
\end{align}
 and \(\yauu{l}{m}\defg\yuuo{l}{m}{v}\) and \(\zauu{l}{m}\defg\zuuo{l}{m}{v}\) for all \(l,m\in\NO\) with \(l\leq m\leq\kappa\). 
 In view of \eqref{Hs}, \eqref{Ks}, \eqref{vsa}, and \eqref{Ha}, then \(-\ug\Hu{n}+\Ku{n}= \Hau{n}\) for all \(n\in\NO\) with \(2n+1\leq\kappa\).

\begin{defn}[\zitaa{MR3014201}{\cdefn{4.2}}]\label{D1021}
 Let \(\alpha\in\C\), let \(\kappa\in\NOinf \), and let \(\seqska \) be a sequence of complex \tpqa{matrices}.
 Then the sequence \(\seq{\Spu{j}}{j}{0}{\kappa}\)\index{q@\(\seq{\Spu{j}}{j}{0}{\kappa}\)} given by \(\Spu{2k}\defg\Lu{k}\) for all \(k\in\NO\) with \(2k\leq\kappa\) and by \(\Spu{2k+1}\defg\Lau{k}\) for all \(k\in\NO\) with \(2k+1\leq\kappa\) is called the \noti{\trasp{\alpha} of \(\seqska \)}{\hrasp{\alpha} of \(\seqska \)}.
 In the case \(\alpha=0\), the sequence \(\seq{\Spu{j}}{j}{0}{\kappa}\) is simply said to be the \notii{\trSpa{\seqska }}.
\end{defn}

\begin{rem}[\zitaa{MR3014201}{\crem{4.3}}]\label{R0929}
 Let \(\alpha\in\C\), let \(\kappa\in\NOinf \), and let \(\seq{\Spu{j}}{j}{0}{\kappa}\) be a sequence of complex \tpqa{matrices}. Then it can be immediately checked by induction that there is a unique sequence \(\seqska \) of complex \tpqa{matrices} such that \(\seq{\Spu{j}}{j}{0}{\kappa}\) is the \traspa{\alpha}{\seqska }, namely the sequence \(\seqska \) recursively given by \(\su{2k}=\Trip{k}+\Spu{2k}\) for all \(k\in\NO\) with \(2k\leq\kappa\) and \(\su{2k+1}=\alpha\su{2k}+\Tripa{k}+\Spu{2k+1}\) for all \(k\in\NO\) with \(2k+1\leq\kappa\).
\end{rem}

 In~\zita{MR3014201} one can find characterizations of the membership of sequences of complex \tqqa{matrices} to the class \(\Kggq{\kappa}\) and to several of its subclasses, respectively.
 For our following considerations, we introduce some of these subclasses.
 
 Let \(\ug\in\R\).
 Let \symba{\Kgq{0}\defg\Hgqu{0}}{k}, and, for all \(n\in\N\), let \(\Kgq{2n}\) be the set of all sequences \(\seqs{2n}\) of complex \tqqa{matrices} for which the block \tHankel{} matrices \(\Hu{n}\) and \(-\ug\Hu{n-1}+\Ku{n-1}\) are \tpH{}, \ie{}, \symba{\Kgq{2n}\defg\setaa{\seqs{2n}\in\Hgqu{2n}}{\seq{\sau{j}}{j}{0}{2(n-1)}\in\Hgqu{2(n-1)}}}{k}.
 Furthermore, for all \(n\in\NO\), let \symba{\Kgq{2n+1}}{k} be the set of all sequences \(\seqs{2n+1}\) of complex \tqqa{matrices} for which the block \tHankel{} matrices \(\Hu{n}\) and \(-\ug\Hu{n}+\Ku{n}\) are \tpH{}, \ie{},
\(
 \Kgq{2n+1}
 \defg\setaa{\seqs{2n+1}\in\setseqauu{(\Cqq)}{0}{2n+1}}{\set{\seqs{2n},\seq{\sau{j}}{j}{0}{2n}}\subseteq\Hgq{2n}}
\)
 Let \symba{\Kgqinf}{k} be the set of all sequences \(\seqsinf\) of complex \tqqa{matrices} such that \(\seqs{m}\in\Kgq{m}\) for all \(m\in\NO\).

\bpropnl{\zitaa{MR3014201}{\cprop{2.20}}}{P1830}
 Let \(\ug\in\R\) and \(m\in\NO\).
 Then \(\Kgq{m}\subseteq\Kggeq{m}\).
\eprop

 For all \(n\in\NO\), let\index{h@\(\Hggdq{2n}\)}
\(
 \Hggdq{2n}
 \defg\setaa{\seqs{2n}\in\Hggq{2n}}{\Luo{n}{s}=\Oqq}
\),
 where \(\Luo{n}{s}\) is given via \eqref{L}.
 The elements of the set \(\Hggdq{2n}\) are called \notii{\tHd{}}.
 For every choice of \(\ug\in\R\) and \(n\in\NO\), let \(\Kggdq{2n}\defg\Kggq{2n}\cap\Hggdq{2n}\)\index{k@\(\Kggdq{2n}\)} and let \(\Kggdq{2n+1}\defg\setaa{\seqs{2n+1}\in\Kggq{2n+1}}{\seq{\sau{j}}{j}{0}{2n}\in\Hggdq{2n}}\)\index{k@\(\Kggdq{2n+1}\)}.

\bdefnl{D0933}
 Let \(\ug\in\R\) and let \(\seqsinf\in\Kggqinf\).
\benui
 \il{D0933.a} Let \(m\in\NO\).
 Then \(\seqsinf\) is called \notii{\haSdo{m}} if \(\seqs{m}\in\Kggdq{m}\).
 \il{D0933.b} The sequence \(\seqsinf\) is called \notii{\haSd{}} if there exists an \(m\in\NO\) such that \(\seqsinf\) is \taSdo{m}.
\eenui
\edefn

\bpropnl{\zitaa{MR3014201}{\cprop{5.9}}}{P1628}
 Let \(\ug\in\R\) and \(m\in\NO\).
 Then \(\Kggdq{m}\subseteq\Kggeq{m}\).
\eprop

 For the convince of the reader, we add a technical result:
 
\blemnl{\zitaa{MR3014201}{\clem{2.9}}}{R1738}
 Let \(\ug\in\R\), let \(\kappa\in\NOinf\), and let \(\seqs{\kappa}\in\Kggq{\kappa}\).
 \benui
  \il{R1738.a} \(\su{j}\in\CHq\) for all \(j\in\mn{0}{\kappa}\) and \(\sau{j}\in\CHq\) for all \(j\in\mn{0}{\kappa-1}\).
  \il{R1738.b} \(\su{2k}\in\Cggq\) for all \(k\in\NO\) with \(2k\leq\kappa\) and \(\sau{2k}\in\Cggq\) for all \(k\in\NO\) with \(2k+1\leq\kappa\).
 \eenui
\elem

\section{A Schur type algorithm for sequences of complex matrices}
\subsection{Some observations on the \hascht{\alpha} of sequences of complex matrices}\label{S0850}
 The basic object of this section was introduced in~\zita{MR3611479}.
 We want to recall its definition.
 To do this we start with the \tsrautea{a} given sequence of complex \tpqa{matrices}.

\begin{defn}[\zitaa{MR3014197}{\cdefn{4.13}}]\label{D1430}
 Let \(\kappa \in \NOinf \) and let \(\seqska \) be a sequence of complex \tpqa{matrices}. The sequence \(\seq{\su{j}^\rez}{j}{0}{\kappa}\)\index{\(\seq{\su{j}^\rez}{j}{0}{\kappa}\)} given by \(\su{0}^\rez\defg\su{0}^\mpi\) and \(s_j^\rez\defg-\su{0}^\mpi \sum_{l=0}^{j-1}\su{j-l} \su{l}^\rez\) for all \(j\in\mn{1}{\kappa}\) is said to be the \notii{\tsrautea{\(\seqska \)}}.
\end{defn}

\begin{rem}[\zitaa{MR3014197}{\crema{4.17}}]\label{R1022}
 Let \(\kappa\in\NOinf \) and let \(\seqska \) be a sequence of complex \tpqa{matrices} with \tsraute{} \(\seq{\su{j}^\rez}{j}{0}{\kappa}\). For all \(m\in\mn{0}{\kappa}\), then \(\seq{\su{j}^\rez}{j}{0}{m}\) is the \tsrautea{\(\seqs{m}$}.
\end{rem}

\begin{defn}[\zitaa{MR3611479}{\cdefn{4.1}}]\label{D1455}
 Let \(\alpha\in\C\), let \(\kappa\in\NOinf \), and let \(\seqska \) be a sequence of complex \tpqa{matrices}.
 Then we call the sequence \symb{\seq{\su{j}^\splusalpha}{j}{0}{\kappa}} given by \(\su{j}^\splusalpha\defg-\alpha\su{j-1}+\su{j}\) for all \(j\in\mn{0}{\kappa}\), where \(\su{-1}\defg\Opq\), the \noti{\tsplusalphata{\seqska }}{\hsplusalphata{\seqska }}.
\end{defn}

 Obviously, the \tsplusalphata{\seqska } is connected with the sequence \(\seq{\sau{j}}{j}{0}{\kappa-1}\) given in \eqref{vsa} via \(\su{j+1}^\splusalpha=\sau{j}\) for all \(j\in\mn{0}{\kappa-1}\). Furthermore, we have \(\su{0}^\splusalpha =\su{0}\).

 Let \(\alpha\in\C\).
 In order to prepare the basic construction in \rsect{S0815}, we study the \tsraute{} corresponding to the \tsplusalphat{} of a sequence.
 Let \(\kappa\in\NOinf \) and let \(\seqska \) be a sequence of complex \tpqa{matrices} with \tsplusalphat{} \(\seq{u_j}{j}{0}{\kappa}\).
 Then we define \(\seq{\su{j}^\reza{\alpha}}{j}{0}{\kappa}\)\index{\(\seq{\su{j}^\reza{\alpha}}{j}{0}{\kappa}\)} by \(\su{j}^\reza{\ug}
 \defg u_j^\rez\) for all \(j\in\mn{0}{\kappa}\), \ie{}, the sequence \(\seq{\su{j}^\reza{\alpha}}{j}{0}{\kappa}\) is defined to be the \tsrautea{the \tsplusalphata{\seqska }}.

\begin{defn}[\zitaa{MR3611479}{\cdefn{7.1}}]\label{D1059}
 Let \(\alpha\in\C\), let \(\kappa\in\Ninf \), and let \(\seqska \) be a sequence of complex \tpqa{matrices}.
 Then the sequence \symb{\seq{\su{j}^\sntaa{1}{\alpha}}{j}{0}{\kappa-1}} defined by \(\su{j}^\sntaa{1}{\alpha}\defg-\su{0}\su{j+1}^\reza{\ug}\su{0}\) for all \(j\in\mn{0}{\kappa-1}\) is called the \notii{first \tascht{\alpha}} (or short the \notii{first \tlasnt{\ug}}) \notii{of \(\seqska \)}.
\end{defn}

\bthmnl{\zitaa{MR3611479}{\cthm{7.21(a) and~(b)}}}{P1546}
 Let \(\ug\in\R\), let \(m\in\N\), and let \(\seqs{m}\) be a sequence of complex \tqqa{matrices} with \tseinsalpha{} \(\seq{\su{j}^\seinsalpha}{j}{0}{m-1}\).
 Then:
 \benui
  \il{P1546.a} If \(\seqs{m}\in\Kggq{ m}\), then \(\seq{\su{j}^\seinsalpha}{j}{0}{ m-1}\in\Kggq{ m-1}\).
  \il{P1546.b} If \(\seqs{m}\in\Kggeq{ m}\), then \(\seq{\su{j}^\seinsalpha}{j}{0}{ m-1}\in\Kggeq{ m-1}\).
 \eenui
\etheo

 The next result should be considered against to the background of \rprob{\mproblem{\rhl}{m}{\leq}} and indicates that the first \tlasnt{\ug} for finite sequences preserves a particular matrix inequality with respect to the L\"owner semi-ordering for \tH{} matrices.
 This observation has far-reaching consequences for our further considerations.

\bleml{T149_4}
 Let \(\ug\in\R\) and \(m \in \N\).
 Furthermore, let \(\seqs{m}\) and \((t_j)_{j=0}^m\) be sequences of \tH{} complex \tqqa{matrices} such that
\begin{align} \label{T149_4_V2}
 t_j&= s_j\text{ for all } j\in \Zzim&
&\text{and}& 
 t_m&\leq s_m.
\end{align}
 Denote by \(\seq{\su{j}^\seinsalpha}{j}{0}{m-1}\) and \(\seq{t_j^\seinsalpha}{j}{0}{m-1}\) the first \tlasnt{\ug}s of \(\seqs{m}\) and \((t_j)_{j=0}^m\), respectively.
 Then:
\benui
 \il{T149_4.a} For each \(j\in\mn{0}{m-1}\), the matrices \(\su{j}^\seinsalpha\) and \(t_j^\seinsalpha\) are both \tH{}.
 \il{T149_4.b} The inequality \(t_{m-1}^\seinsalpha  \leq\su{m-1}^\seinsalpha \) holds true.
 Furthermore, if \(m \geq 2\), then \(t_j^\seinsalpha  = s_j^\seinsalpha \) for all \(j \in \mn{0}{m-2}\).
\eenui
\elem
\begin{proof}
 \eqref{T149_4.a} Since \(\ug\in\R\) is supposed,~\eqref{T149_4.a} follows from~\zitaa{MR3611479}{\clem{7.5(f)}}.
 
 \eqref{T149_4.b} Since \(\seqs{m}\) is a sequence of \tH{} matrices, from \rrem{R1133} we have \((\sop\so)^\ad= \su{0}\sop\) and, in view of~\cite[\crem{8.7}]{MR3611479}, furthermore \(\set{s_{m-1}^\seinsalpha , t_{m-1}^\seinsalpha }\subseteq \CHq\).
 
 First we assume \(m = 1\). 
 Taking into account~\zitaa{MR3611479}{\clem{7.5(f)}} and \((\sop\so)^\ad = \su{0}\sop\), we obtain 
\beql{T149_4.100}\begin{split} %
s_{m-1}^\seinsalpha  - t_{m-1}^\seinsalpha  
&= \su{0}^\seinsalpha  - t_0^\seinsalpha  
= \su{0}\su{0}^\mpi s_1^\splusalpha \sop\su{0} - t_0 t_0^\mpi  t_1^\splusalpha t_0^\mpi t_0  \\
&= \su{0}\sop(-\alpha\su{0} + s_1)\su{0}\su{0}^\mpi - t_0t_0^\mpi (-\alpha t_0 + t_1)t_0^\mpi  t_0  \\
&= \su{0}\sop(-\alpha \su{0} + s_1)\sop\su{0} - \su{0}\sop(-\alpha\su{0} + t_1)\sop\su{0}  \\
&= (\sop\so)^\ad (s_1 - t_1)\sop\su{0} 
= (\sop\so)^\ad (s_m - t_m)\sop\so.
\end{split}\eeq

 Now we consider the case \(m \geq 2\).
 From \eqref{T149_4_V2} and~\cite[\crem{7.3}]{MR3611479} we get \(t_j^\seinsalpha  = s_j^\seinsalpha \) for all \(j \in \mn{0}{m-2}\). 
 Because of~\zitaa{MR3611479}{\clem{7.8}}, \eqref{T149_4_V2}, and \((\sop\so)^\ad  = \su{0}\sop\), we conclude
\beql{T149_4.101}\begin{split}
&s_{m-1}^\seinsalpha  - t_{m-1}^\seinsalpha   \\
&=\su{0}\sop\rk*{ s_m^\splusalpha\sop\su{0} - \sum_{l=0}^{m-2}\su{m-1-l}^\splusalpha\su{0}^\mpi s_l^\seinsalpha } - t_0 t_0^\mpi  \rk*{ t_m^\splusalpha t_0^\mpi t_0 - \sum_{l=0}^{m-2} t_{m-1-l}^\splusalpha t_0^\mpi  t_l^\seinsalpha }  \\
&= \su{0}\su{0}^\mpi  \ek*{\rk{-\alpha\su{m-1} +s_m} \su{0}^\mpi \su{0} - \sum_{l=0}^{m-2} \rk{-\alpha\su{m-2-l}+s_{m-1-l}} \su{0}^\mpi s_l^\seinsalpha  }  \\
&\qquad -t_0t_0^\mpi  \ek*{\rk{-\alpha t_{m-1} +t_m} t_0^\mpi t_0 - \sum_{l=0}^{m-2} \rk{ -\alpha t_{m-2-l}+t_{m-1-l}} t_0^\mpi t_l^\seinsalpha }  \\
&= \su{0}\su{0}^\mpi  \ek*{\rk{-\alpha\su{m-1} +s_m} \su{0}^\mpi \su{0} - \sum_{l=0}^{m-2} \rk{-\alpha\su{m-2-l}+s_{m-1-l}} \su{0}^\mpi s_l^\seinsalpha }  \\
&\qquad -\su{0}\su{0}^\mpi  \ek*{\rk{ -\alpha\su{m-1} +t_m } \su{0}^\mpi \su{0} - \sum_{l=0}^{m-2} \rk{-\alpha\su{m-2-l}+s_{m-1-l}} \su{0}^\mpi s_l^\seinsalpha }  \\
&=(\sop\so)^\ad \rk{ s_m - t_m}\sop\su{0} 
=(\sop\so)^\ad \rk{ s_m - t_m}\sop\so.
\end{split}\eeq
 In view of \(t_m\leq\su{m}\), we see from \eqref{T149_4.100} and \eqref{T149_4.101} that \(t_{m-1}^\seinsalpha\leq\su{m-1}^\seinsalpha\).
\end{proof}

\subsection{The algorithm}\label{S1549}
 The basic object of this section was introduced in~\zitaa{MR3611479}{\cSect{10}}.
 For the convenience of the reader, we want to recall our motivation:
 Let \(\ug\in\C\) and \(\kappa\in\Ninf\).
 Let \(\seqs{\kappa}\) be a sequence from \(\Cpq\) and let \(\seq{\su{j}^\seinsalpha}{j}{0}{\kappa-1}\) be its \tseinsalpha{} (see \rdefn{D1059}).
 Then we want to recover the sequence \(\seqs{\kappa}\) on the basis of the sequence \(\seq{\su{j}^\seinsalpha}{j}{0}{\kappa-1}\) and the matrix \(\su{0}\).
 Against to this background we recall the following notion.
 
\begin{defn}[\zitaa{MR3611479}{\cdefn{10.1}}]\label{D1712}
 Let \(\alpha\in\C\), let \(\kappa\in\NOinf \), let \(\seq{t_j}{j}{0}{\kappa}\) be a sequence of complex \tpqa{matrices}, and let \(A\) be a complex \tpqa{matrix}.
 The sequence \(\seq{t_j^\sminuseinsalphaa{A}}{j}{0}{\kappa+1}\)\index{\(\seq{t_j^\sminuseinsalphaa{A}}{j}{0}{\kappa+1}\)} recursively defined by
 \begin{align*}
  t_0^\sminuseinsalphaa{A}&\defg A&
 &\text{and}&
  t_j^\sminuseinsalphaa{A}&\defg\alpha^jA+\sum_{l=1}^j\alpha^{j-l}AA^\mpi\ek*{\sum_{k=0}^{l-1}t_{l-k-1}A^\mpi(t_k^\sminuseinsalphaa{A})^\splusalpha}
 \end{align*}
 for all \(j\in\mn{1}{\kappa+1}\) is called the \notii{\tsminuseinsalphaaa{\seq{t_j}{j}{0}{\kappa}}{A}}.
\end{defn}

\begin{rem}[\zitaa{MR3611479}{\crema{10.2}}]\label{T2110_114_102}
 Let \(\alpha\in\C\), let \(\kappa\in\NOinf \), let \(\seq{t_j}{j}{0}{\kappa}\) be a sequence from \(\Cpq\), and let \(A\) \(\in \Cpq\).
 Denote by \(\seqs{\kappa+1}\) the \tsminuseinsalphaaa{\seq{t_j}{j}{0}{\kappa}}{A}.
 In view of \rdefi{D1712}, one can easily see that, for all \(m\in\mn{0}{\kappa}\), the sequence \(\seqs{m+1}\) depends only on the matrices \(A\) and \(t_0,t_1,\dotsc,t_m\) and it is hence exactly the \tsminuseinsalphaaa{\seq{t_j}{j}{0}{m}}{A}.
\end{rem}

\begin{defn}[\zitaa{MR3014197}{\cdefn{4.3}}]\label{D1658}
 Let \(\kappa\in\NOinf\) and let \(\seqs{\kappa}\) be a sequence of complex \tpqa{matrices}.
 We then say that \(\seqs{\kappa}\) is \emph{\tftd{}} if \(\bigcup_{j=0}^\kappa\ran{\su{j}}\subseteq\ran{\su{0}}\) and \(\nul{\su{0}}\subseteq\bigcap_{j=0}^\kappa\nul{\su{j}}\).
 The set of all \tftd{} sequences \(\seqs{\kappa}\) of complex \tpqa{matrices} will be denoted by \symba{\Dpqu{\kappa}}{d}.
\end{defn}

\begin{rem}[\zitaa{MR3611479}{\crema{10.3}}]\label{T2110_114_103}
 Let \(\alpha\in\C\), let \(\kappa\in\NOinf \), let \(\seq{t_j}{j}{0}{\kappa}\) be a sequence from \(\Cpq\), and let \(A\) \(\in \Cpq\).
 Denote by \(\seqs{\kappa+1}\) the \tsminuseinsalphaaa{\seq{t_j}{j}{0}{\kappa}}{A}.
 From \rdefi{D1712} we easily see then that \(\seqs{\kappa+1}\in\Dpqu{\kappa+1}\).
\end{rem}

\begin{lem}[\zitaa{MR3611479}{\clemm{10.4}}] \label{T2110_114_104}
 Let \(\alpha\in\C\), let \(\kappa\in\NOinf \), let \(\seq{t_j}{j}{0}{\kappa}\) be a sequence from \(\Cpq\), and let \(A\) \(\in \Cpq\).
 Denote by \(\seqs{\kappa+1}\) the \tsminuseinsalphaaa{\seq{t_j}{j}{0}{\kappa}}{A} and by \(\seq{\su{j}^\splusalpha}{j}{0}{\kappa+1}\) the \tsplusalphata{\seqs{\kappa+1}}.
 Then \(\su{0}=A\) and \(s_j=\alpha\su{j-1}+AA^\mpi\sum_{k=0}^{j-1}t_{j-1-k}A^\mpi\su{k}^\splusalpha\) for all \(j\in\mn{1}{\kappa+1}\).
\end{lem}

\begin{lem}[\zitaa{MR3611479}{\clemm{10.13}}]\label{T2110_114_1013}
 Let \(\ug\in\R\), let \(\kappa\in\NOinf \), let \(\seq{t_j}{j}{0}{\kappa}\) be a sequence of \tH{} complex \tqqa{matrices}, and let \(A\) be a \tH{} complex \tqqa{matrix}.
 Then the \tsminuseinsalpha{} \(\seq{t_{j}^\sminuseinsalphaa{A}}{j}{0}{\kappa+1}\) corresponding to \([\seq{t_j}{j}{0}{\kappa},A]\) is a sequence of \tH{} complex \tqqa{matrices}.
\end{lem}

 The following result can be considered as an analogue of \rlem{T149_4} for the \tsminuseinsalpha{} in the case of given sequences of \tH{} complex \tqqa{matrices}.
\bleml{T2110_1}
 Let \(\ug\in\R\), let \(m \in \NO \), and let \(A \in \CHq\). 
 Further, let \(\seqs{m}\) and \(\tjm\) be sequences of \tH{} complex \tqqa{matrices} such that \(t_m\leq s_m\) and, in the case \(m\geq1\), furthermore, \(t_j= s_j\) for all \(j\in \Zzim\) hold true.
 Denote by \(\seq{\su{j}^\sminuseinsalphaa{A}}{j}{0}{m+1}\) and \(\seq{t_{j}^\sminuseinsalphaa{A}}{j}{0}{m+1}\) the \tsminuseinsalphaaa{\seqs{m}}{A} and \([\seqt{m},A]\), respectively.
 Then \(t_j^\sminuseinsalphaa{A} = s_j^\sminuseinsalphaa{A}\) for each \(j \in \Zzm\) and \(t_{m+1}^\sminuseinsalphaa{A} \leq\su{m+1}^\sminuseinsalphaa{A}\).
\elem
\begin{proof}
 For each \(j\in\mn{0}{m+1}\) let \(r_j\defeq\su{j}^\sminuseinsalphaa{A}\) and \(u_j\defeq t_{j}^\sminuseinsalphaa{A}\). 
 From \eqref{T149_4_V2} and \rrem{T2110_114_102} we obtain immediately \(u_j= r_j\) for each \(j\in\Zzm\).
 It remains to check \(u_{m+1} \leq r_{m+1}\).
 Because of \(A^\ad  = A\) and \rrem{L1631}, we have \((A^\mpi A)^\ad  = AA^\mpi \). 
 Since \(\seqs{m}\) and \(\tjm\) are sequences of \tH{} complex matrices, \rlem{T2110_114_1013} yields \(\set{ r_{m+1}, u_{m+1}  } \subseteq \CHq\).
 If \(m = 0\), then \rlem{T2110_114_104}, \((A^\mpi A)^\ad  = AA^\mpi \), and \eqref{T149_4_V2} provide us
\beql{T2110_1.4}\begin{split}
r_{m+1} - u_{m+1} 
&= r_1 - u_1 
= \alpha r_0 + A A^\mpi  \su{0} A^\mpi  r_0^\splusalpha - \rk{ \alpha u_0 + AA^\mpi  t_0 A^\mpi  u_0^\splusalpha }  \\
&= \alpha r_0 + A A^\mpi  \su{0} A^\mpi  r_0 - \rk{ \alpha u_0 + AA^\mpi  t_0 A^\mpi  u_0 }  \\
&= \alpha A + A A^\mpi  \su{0} A^\mpi  A - \rk{ \alpha A + AA^\mpi  t_0 A^\mpi  A }  \\
&= A A^\mpi ( \su{0} - t_0) A^\mpi  A  
= \rk{ A^\mpi A }^\ad ( s_m - t_m) A^\mpi  A.
\end{split}\eeq

 Now we consider the case \(m \geq 1\). 
 Using \rlem{T2110_114_104}, \eqref{T149_4_V2}, \(u_j= r_j\) for each \(j\in\Zzm\), \((A^\mpi A)^\ad  = AA^\mpi \), and again \eqref{T149_4_V2}, we get then
\beql{T2110_1.5}\begin{split}
r_{m+1} - u_{m+1} 
&= \alpha r_m + AA^\mpi  \sum^m_{k=0}\su{m-k} A^\mpi  r_k^\splusalpha - \rk*{ \alpha u_m + AA^\mpi  \sum^m_{k=0} t_{m-k} A^\mpi  u_k^\splusalpha}  \\
&= \alpha r_m + AA^\mpi  s_mA^\mpi  r_0^\splusalpha + AA^\mpi  \sum^m_{k=1}\su{m-k} A^\mpi  r_k^\splusalpha  \\
&\qquad- \rk*{ \alpha u_m + AA^\mpi  t_mA^\mpi  u_0^\splusalpha + AA^\mpi  \sum^m_{k=1} t_{m-k} A^\mpi  u_k^\splusalpha}  \\
&= \alpha r_m + AA^\mpi  s_mA^\mpi  r_0 + AA^\mpi  \sum^m_{k=1}\su{m-k} A^\mpi  (-\alpha r_{k-1} +r_k)   \\
&\qquad- \ek*{ \alpha u_m + AA^\mpi  t_mA^\mpi  u_0 + AA^\mpi  \sum^m_{k=1} t_{m-k} A^\mpi  (-\alpha u_{k-1} +u_k)}  \\
&= \alpha r_m + AA^\mpi  s_mA^\mpi A + AA^\mpi  \sum^m_{k=1}\su{m-k} A^\mpi  (-\alpha r_{k-1} +r_k)   \\
&\qquad- \ek*{ \alpha r_m + AA^\mpi  t_mA^\mpi A + AA^\mpi  \sum^m_{k=1}\su{m-k} A^\mpi  (-\alpha r_{k-1} +r_k)}  \\
&= AA^\mpi  \rk{ s_m - t_m } A^\mpi A 
= \rk{ A^\mpi A }^\ad  \rk{ s_m - t_m } A^\mpi A.
\end{split}\eeq

 In view of \(t_m\leq\su{m}\), we see from \eqref{T2110_1.4} and \eqref{T2110_1.5} that \(u_{m+1}\leq r_{m+1}\).
\end{proof}

 The \tascht{\ug} for sequences of complex \tpqa{matrices} introduced in \rsec{S0850} generates in a natural way a corresponding algorithm for (finite and infinite) sequences of complex \tqqa{matrices}.
 In generalization of \rdefn{D1059}, we introduced the following:
 
\bdefnnl{\zitaa{MR3611479}{\cdefn{8.1}}}{D1632}
 Let \(\ug\in\C\), let \(\kappa\in\NOinf \), and let \(\seqska \) be a sequence of complex \tpqa{matrices}.
 The sequence \(\seq{\su{j}^\sta{0}}{j}{0}{\kappa}\) given by \symb{\su{j}^\sta{0}\defg\su{j}} for all \(j\in\mn{0}{\kappa}\) is called the \notii{\tsaalphaa{0}{\seqska }}.
 In the case \(\kappa\geq1\), for all \(k\in\mn{1}{\kappa}\), the \notii{\tsaalpha{k} \(\seq{\su{j}^\sta{k}}{j}{0}{\kappa-k}\) of \(\seqska\)} is recursively defined by \symb{\su{j}^\sta{k}\defg t_j^\seinsalpha} for all \(j\in\mn{0}{\kappa-k}\), where \(\seq{t_j}{j}{0}{\kappa-(k-1)}\) denotes the \tsaalphaa{(k-1)}{\seqska}.
\edefn

 A comprehensive investigation of this algorithm was carried out in~\zita{MR3611479}.

\bremnl{\zitaa{MR3611479}{\crem{8.3}}}{R1510}
 Let \(\ug\in\C\), let \(\kappa\in\NOinf\), and let \(\seqs{\kappa}\) be a sequence of complex \tpqa{matrices}. From \rdefi{D1632} then it is immediately obvious that, for all \(k\in\mn{0}{\kappa}\) and all \(l\in\mn{0}{\kappa-k}\), the \((k+l)\)\nobreakdash-th \tlasnt{\alpha} \(\seq{\su{j}^\sntaa{k+l}{\alpha}}{j}{0}{\kappa-(k+l)}\) of \(\seqs{\kappa}\) is exactly the \(l\)\nobreakdash-th \tlasnt{\alpha} of the \(k\)\nobreakdash-th \tlasnt{\alpha} \(\seq{\su{j}^\sntaa{k}{\alpha}}{j}{0}{\kappa-k}\) of \(\seqs{\kappa}\).
\erem

 There is a remarkable interrelation between \saScht{ation} and \trasp{\ug}:
\bthmnl{\zitaa{MR3611479}{\cthm{9.15}}}{T1615}
 Let \(\ug\in\R\), let \(\kappa\in\NOinf \), and let \(\seqs{\kappa}\in\Kggeq{\kappa}\).
 Then \(\seq{\su{0}^\sta{j}}{j}{0}{\kappa}\) is exactly the \traspa{\ug}{\seqs{\kappa}}.
\ethm

\section{On some classes of holomorphic matrix-valued functions}\label{S1104}
\subsection{The class \(\SFqa\)}\label{S*2}
 The use of several classes of holomorphic matrix-valued functions is one of the special features of this paper. In this section, we summarize some basic facts about the class of \tlSF{s} of order \(q\), which are mostly taken from our former paper~\zita{FKM15a}. If \(A\in\Cqq\), then let \(\re A\defeq\frac{1}{2}(A+A^\ad)\)\index{r@\(\re A\)} and \(\im A\defeq\frac{1}{2\iu}(A-A^\ad)\)\index{i@\(\im A\)} be the real part and the imaginary part of \(A\), respectively. Let \(\ohe\defeq\setaa{z\in\C}{\im z>0}\)\index{p@\(\ohe\)} be the open upper half plane of \(\C\). The first class of functions, which plays an essential role in this paper, is the  following.

\bdefnl{D1242}
 Let \(\ug\in\R\) and let \(F\colon\Cs\to\Cqq\). Then \(F\) is called a \notii{\tlSFq{}} if \(F\) satisfies the following three conditions:
 \bAeqi{0}
  \il{D1242.I} \(F\) is holomorphic in \(\Cs\).
  \il{D1242.II} For all \(w\in\ohe\), the matrix \(\im\ek{F(w)}\) is \tnnH{}.
  \il{D1242.III} For all \(w\in\lhl\), the matrix \(F(w)\) is \tnnH{}.
 \eAeqi
 We denote by \(\SFqa\)\index{s@\(\SFqa\)} the set of all \tlSF{s} of order \(q\).
\edefn

\bexaml{E1027}
 Let \(\ug\in\R\) and let \(A\in\Cggq\).
 Then \(F\colon\Cs\to\Cqq\) defined by \(F(z)\defeq A\) belongs to \(F\in\SFqa\).
\eexam

 For a comprehensive survey on the class \(\SFqa\), we refer the reader to~\zita{FKM15a}.
 We give now a useful characterization of the membership of a function to the class \(\SFqa\).
 Let \(\uhe\defeq\setaa{z\in\C}{\im z<0}\)\index{Pi@\(\uhe\defeq\setaa{z\in\C}{\im z<0}\)} and let \symba{\lhea\defeq\setaa{z\in\C}{\re z\in(-\infty,\ug)}}{c}.

\bpropnl{\zitaa{FKM15a}{\cprop{4.4}}}{214-1}
 Let \(\ug\in\R\) and let \(F\colon\Cs \to\Cqq\) be a matrix-valued function.
 Then \(F\) belongs to \(\SFqa\) if and only if the following four conditions are fulfilled:
 \bAeqi{0}
  \il{214-1.I} \(F\) is holomorphic in \(\Cs \).
  \il{214-1.II} For each \(z\in\ohe\), the matrix \(\im F(z)\) is \tnnH{}.
  \il{214-1.III} For each \(z\in\uhe\), the matrix \(-\im F(z)\) is \tnnH{}.
  \il{214-1.IV} For each \(z\in\lhea\), the matrix \(\re F(z)\) is \tnnH{}.
 \eAeqi
\eprop 

\bcorol{C1031}
 Let \(\ug\in\R\) and let \(F\in\SFqa\).
 For each \(z\in\C\setminus\R\), then \(\rk{\im z}^\inv\im F(z)\in\Cggq\).
\ecoro
\bproof
 In view of \(\C\setminus\R=\ohe\cup\uhe\) and \(F\in\SFqa\), this follows from conditions~\ref{214-1.II} and~\ref{214-1.III} of \rprop{214-1}.
\eproof

 The functions belonging to the class \(\SFqa\) admit an important integral representation. To state this, we introduce some terminology: If \(\mu\) is a \tnnH{} \tqqa{measure} on a measurable space \((\Omega,\mathfrak{A})\) and if \(\mathbb{K}\in\set{\R,\C}\), then we will use \symba{\LaaaK{\Omega}{\mathfrak{A}}{\mu}}{l} to denote the space of all Borel-measurable functions \(f\colon\Omega\to\mathbb{K}\) for which the integral \(\int_\Omega f\dif\mu\) exists. In preparing the desired integral representation, we observe that, for all \(\mu\in\MggqK\) and each \(z\in\Cs\), the function \(\hu{z}\colon\rhl\to\C\)\index{h@\(\hu{z}\)} defined by \(\hua{z}{t}\defeq(1+t-\ug)/(t-z)\) belongs to \(\LaaaC{\rhl}{\BorK}{\mu}\).

\bthmnl{\cf{}~\zitaa{FKM15a}{\cthm{3.6}}}{T2*1}
 Let \(\ug \in\R\) and let \(F\colon\Cs  \to \Cqq \). Then:
 \benui
 \il{T2*1.a} If \(F \in \SFqa \), then there are a unique matrix \(\gamma \in \Cggq\) and a unique \tnnH{} measure \(\mu \in\Mggqrhl \) such that
 \beql{F2*2}
  F(z)
  = \gamma +\int_{\rhl} \frac{1+ t-\ug}{t-z} \mu (\dif t)
 \eeq
 holds true for each \(z\in\Cs \).
 \il{T2*1.b} If there are a matrix \(\gamma\in\Cggq\) and a \tnnH{} measure \(\mu \in\Mggqrhl \) such that \(F\) can be represented via \eqref{F2*2} for each \(z\in \Cs \), then \(F\) belongs to the class \(\SFqa \).
 \eenui
\ethm

 In the special case that \(q=1\) and \(\ug=0\) hold true, \rthm{T2*1} can be found in Krein/Nudelman~\zitaa{MR0458081}{Appendix}.

\bnotal{N1123}
 For all \(F\in\SFqa\), we will write \symba{(\gammau{F},\muu{F})}{g} for the unique pair \((\gamma,\mu)\in\Cggq\times\Mggqrhl \) for which the representation \eqref{F2*2} holds true for all \(z\in \Cs \).
\enota

 We are interested in the structure of the values of functions belonging to \(\SFqa\).

\bpropnl{\cf{}~\zitaa{FKM15a}{\cprop{3.15}}}{P2*7}
 Let \(\ug \in\R\) and let \(F\in\SFqa \).
 For all \(z\in \Cs \), then \(\ran{F(z)}=\ran{\gammau{F}}+\ran{\muua{F}{\rhl }}\), \(\nul{F(z)}=\nul{\gammau{F}}\cap\nul{\muua{F}{\rhl }}\), and \(\ran{\ek{F(z)}^\ad}=\ran{F(z)}\).
\eprop

 In the sequel, we will sometimes meet situations where interrelations between the null space (respectively, column space) of a function \(F\in\SFqa\) and the null space (respectively, column space) of a given matrix \(A\in\Cpq\) are of interest.

\bnotal{D1157}
 Let \(\ug\in\R\) and let \(A\in\Cqp\).
 We denote  by \symba{\SFqr{A}}{s} the set of all \(F\in\SFqa\) which satisfy \(\ran{F(z)}\subseteq\ran{A}\) for all \(z\in\Cs\).
\enota

 Observe that the fact that a matrix-valued function \(F\in\SFqa\) belongs to the subclass \(\SFqr{A}\) of the class \(\SFqa\) can be characterized by several conditions (see~\zitaa{FKM15a}{\clem{3.18}}).
 In particular, we are led to the following useful characterization of the elements of the class \(\SFqr{A}\).
 
\bleml{L1210}
 Let \(\ug\in\R\), let \(F\in\SFqa\), and let \(A\in\Cqp\).
 Then \(F\in\SFqr{A}\) if and only if \(\ran{\gammaF}+\ran{\muFa{\rhl}}\subseteq\ran{A}\).
\elem
\bproof
 Combine \rprop{P2*7} with \rremp{L1631}{L1631.b}.
\eproof

\subsection{On some subclasses of \(\SFqa\)}\label{S*3}
 An essential feature of our subsequent considerations is the use of several subclasses of \(\SFqa\).
 In this section, we summarize some basic facts about these subclasses, which are characterized by growth properties on the positive imaginary axis.
 It should be mentioned that scalar versions of the function classes were introduced and studied by Kats/Krein~\zita{KK74}. 
 Furthermore, we recognized in~\zita{MR3611471} that a detailed analysis of the behavior on the positive imaginary axis of the concrete functions of \(F\in\SFqa\) under study is very useful.
 For this reason, we turn now our attention to some subclasses of \(\SFqa\), which are described in terms of their growth on the positive imaginary axis.
 For each \(\ug\in\R\), first we consider the set\index{s@\(\SFdqa\)}
\beql{F3*4}
 \SFdqa
 \defeq\setaa*{F\in\SFqa}{\lim_{y\to\infp}\normS*{F(\iu y)}=0}
\eeq
 and we observe that \(\SFdqa=\setaa{F\in\SFqa}{\gammaF=\NM}\) (see~\zitaa{FKM15a}{\ccor{3.14}}).
 In the following considerations, we associate with a function \(F\in\SFqa\) often the unique ordered pair \((\gammaF,\muF)\) given via \rnota{N1123}.

 In~\zitaa{MR3611471}{\csec{4}}, we considered a particular subclass of the class \(\SFdqa\).
 We have seen in \rprop{P2*7} that, for an arbitrary function \(F\in\SFqa\), the null space of \(F(z)\) is independent from the concrete choice of \(z\in\Cs\).
 For the case that \(F\) belongs to \(\SFdqa\), a complete description of this constant null space was given in~\zitaa{MR3611471}{\cprop{3.7}}.
 Against to this background, we single out now a special subclass of \(\SFdqa\), which is characterized by the interrelation of this constant null space to the null space of a prescribed matrix \(A\in\Cpq\).
 More precisely, in view of \rnota{N1123}, for all \(A\in\Cpq\), let\index{s@\(\SFdqaa{A}\)}
\begin{equation}\label{F4*1}
 \SFdqaa{A}
 \defeq\setaa*{F\in\SFdqa}{\nul{A}\subseteq\Nul{\muuA{F}{\rhl}}}.
\end{equation}
 In our investigations in~\zita{MR3611471}, the role of the matrix \(A\) was taken by a matrix which is generated from the sequence of the given data of the moment problem via a Schur type algorithm. 
 In~\zitaa{MR3611471}{\crem{4.4}}, some characterizations of the set \(\SFdqaa{A}\) are proved.
 Furthermore, note that~\zitaa{MR3611471}{\cprop{4.7}} contains essential information on the structure of the set \(\SFdqaa{A}\), where \(A\in\Cpq\) is arbitrarily given.
 In our considerations, the case \(A\in\CHq\) is of particular interest.
 
\breml{R1433}
 Let \(\ug\in\R\) and let \(A\in\CHq\).
 Then~\zitaa{MR3611471}{\crem{4.4}} yields \(\SFdqaa{A}=\setaa{F\in\SFdqa}{\ran{\muFa{\rhl}}\subseteq\ran{A}}\).
 Furthermore, \(\SFdqaa{A}=\SFdqa\cap\SFqr{A}\).
\erem

 A further important subclass of the class \(\SFqa\) is the set
 \beql{F3*3}
  \SFuqa{0}
  \defeq\setaa*{F\in\SFqa}{\sup_{y\in[1,\infp)}y\normS*{F(\iu y)}<\infp}.
 \eeq
 Let \(\Omega\) be a \tne{} closed subset of \(\R\) and let \(\sigma\in\Mggqa{\Omega}\).
 Then, in view of~\zitaa{FKM15a}{\clem{A.4}}, for each \(z\in\C\setminus\Omega\), the function \(f_z\colon\Omega\to\C\) defined by \(f_z(t)\defeq\rk{t-z}^\inv\)\index{f@\(f_z\colon\Omega\to\C\); \(f_z(t)\defeq\rk{t-z}^\inv\)} belongs to \(\LaaaC{\Omega}{\BorO}{\sigma}\).
 In particular, for each \(\ug\in\R\) and each \(\sigma\in\MggqK\), the matrix-valued function \(\OSt{\sigma}\colon\Cs\to\Cqq\)\index{s@\(\OSt{\sigma}\)} given by
\beql{F3*2}
 \OSta{\sigma}{z}
 \defeq\int_\rhl\frac{1}{t-z}\sigma(\dif t)
\eeq
 is well defined and it is called \notii{\taSto{\sigma}}.
 It is readily checked that, for each \(z\in\Cs\), one has \(\ko z\in\Cs\) and \(\OSta{\sigma}{\ko z}=\ek{\OSta{\sigma}{z}}^\ad\).
 
There is an important characterization of the set of all \taSt{s} of measures belonging to \(\MggqK\):
 
\bthmnl{\zitaa{MR3611471}{\cthm{3.2}}}{T3*2}
 Let \(\ug\in\R\).
 The mapping \(\sigma\mapsto\OSt{\sigma}\) given by \eqref{F3*2} is a bijective correspondence between \(\MggqK\) and \(\SFuqa{0}\).
 In particular, \(\SFuqa{0}=\setaa{\OSt{\sigma}}{\sigma\in\MggqK}\).
\ethm

 For each \(F\in\SFuqa{0}\), the unique measure \(\sigma\in\MggqK\) satisfying \(\OSt{\sigma}=F\) is called the \notii{\taSmo{F}} and we will also write \symba{\OSm{F}}{s} for \(\sigma\).
 \rthm{T3*2} and \eqref{F3*3} indicate that the \taSt{} \(\OSt{\sigma}\) of a measure \(\sigma\in\MggqK\) is characterized by a particular mild growth on the positive imaginary axis.

 In view of \rthm{T3*2}, the \rprobss{\mproblem{\rhl}{m}{=}}{\mproblem{\rhl}{m}{\leq}} can be reformulated as an equivalent problem in the class \(\SFuqa{0}\) as follows:

\index{s@\iproblem{\rhl}{m}{=}}
\begin{prob}[\iproblem{\rhl}{m}{=}]
 Let \(\ug\in\R\), let \(m\in\NO\), and let \(\seqs{m}\) be a sequence of complex \tqqa{matrices}.
 Parametrize the set \symba{\SFqas{m}}{s} of all \(F\in\SFuqa{0}\) the \taSm{} of which belongs to \(\MggqKsg{m}\).
\end{prob}

\index{s@\iproblem{\rhl}{m}{\leq}}
\begin{prob}[\iproblem{\rhl}{m}{\leq}]
 Let \(\ug\in\R\), let \(m\in\NO\), and let \(\seqs{m}\) be a sequence of complex \tqqa{matrices}.
 Parametrize the set \symba{\SFqaskg{m}}{s} of all \(F\in\SFuqa{0}\) the \taSm{} of which belongs to \(\MggqKskg{m}\).
\end{prob}

 In~\zitaa{MR3611471}{\csec{6}}, we stated a reformulation of the original power moment problem \mproblem{\rhl}{\kappa}{=} as an equivalent problem of finding a prescribed asymptotic expansion in a sector of the open upper half plane \(\ohe\).

 For all \(\ug\in\R\) and all \(\kappa\in\Ninf\), we now consider the class\index{s@\(\SFkaqa\)}
\beql{F3*11}
 \SFkaqa
 \defeq\setaa*{F\in\SFuqa{0}}{\OSm{F}\in\MggquK{\kappa}}.
\eeq

\blemnl{\zitaa{MR3611471}{\clem{3.9}}}{L3*22}
 Let \(\ug\in\R\), let \(\kappa\in\NOinf\), let \(F\in\SFkaqa\), and let \(\OSm{F}\) be the \taSmo{F}.
 For each \(z\in\Cs\), then \(\ran{F(z)}=\ran{\OSma{F}{\rhl}}\) and \(\nul{F(z)}=\nul{\OSma{F}{\rhl}}\).
\elem

\breml{R3*23}
 Let \(\ug\in\R\),  let \(\kappa\in\NOinf\), and let \(F\in\SFkaqa\).
 From \eqref{F3*11} and \eqref{F3*3} we see then that \(\lim_{y\to\infp}F(\iu y)=\Oqq\).
\erem

\breml{R3*18}
 Let \(\ug\in\R\).
 In view of the \eqref{F3*11}, \rrem{R1715}, \eqref{F3*3}, \rrem{R3*23}, and \eqref{F3*4}, then \(\SFinfqa\subseteq\SFuqa{\ell}\subseteq\SFuqa{m}\subseteq\SFuqa{0}\subseteq\SFdqa\subseteq\SFqa\)
 for all \(\ell,m\in\N\) with \(\ell\geq m\).
\erem

\subsection{On the class \(\RFOq\)}\label{S1040}
 In this section, we consider a class of holomorphic \tqqa{matrix-valued} functions in the open upper half plane \symba{\ohe\defeq\setaa{z\in\C}{\im z\in(0,\infp)}}{p}.
 A function \(F\colon\ohe\to\Cqq\) is called a \notii{\tHNf{} in \(\ohe\)} if \(F\) is holomorphic in \(\ohe\) and satisfies \(\im\ek{F(z)}\in\Cggq\) for all \(z\in\ohe\).
 We denote by \symba{\RFq}{r} the set of all \tHNf{s} in \(\ohe\).
 Furthermore, let\index{r@\(\defeq\setaa{F\in\RFq}{\sup_{y\in[1,\infp)}y\normF{F\rk{\iu y}}<\infp}\)}
\beql{RFO}
 \RFOq
 \defeq\setaa*{F\in\RFq}{\sup_{y\in[1,\infp)}y\normS*{F\rk{\iu y}}<\infp}.
\eeq
 The functions belonging to \(\RFOq\) admit an important integral representation.
 
\bthmnl{see, \eg{}~\zitaa{MR2222521}{\cthm{8.7}}}{T1157}
 \benui
  \il{T1157.a} For each \(F\in\RFOq\), there is a unique \tnnH{} \tqqa{measure} \(\mu\in\Mggqa{\R}\) such that \(F\) admits the representation
  \beql{T1157.B1}
   F(z)
   =\int_\R\frac{1}{t-z}\mu\rk{\dif t}
  \eeq
  for all \(z\in\ohe\).
  \il{T1157.b} If \(F\colon\ohe\to\Cqq\) is a matrix-valued function for which there exists a \(\mu\in\Mggqa{\R}\) such that \eqref{T1157.B1} is satisfied for all \(z\in\ohe\), then \(F\in\RFOq\).
 \eenui
\ethm

 For each \(F\in\RFOq\), the unique \(\mu\in\Mggqa{\R}\) satisfying \eqref{T1157.B1} is called the \notii{\tSmo{F}}.

 The following result on the class \(\RFOq\), which is stated, \eg{}, in~\zitaa{MR2222521}{\clem{8.9}}, plays a key role in our subsequent considerations.
 
\bleml{P1208}
 Let \(M\in\Cqq\) and let \(F\colon\ohe\to\Cqq\) be a matrix-valued function which is holomorphic in \(\ohe\) and which fulfills
\begin{align*}
 \bMat M&F(w)\\ F^\ad(w)&\frac{F(w)-F^\ad(w)}{w-\ko w}\eMat
 &\in\Cggo{2q}&\text{for all }w&\in\ohe.
\end{align*}
 Then \(F\in\RFOq\) and \(\sup_{y\in(0,\infp)}y\normS{F\rk{\iu y}}\leq\normS{M}\).
 Furthermore, the \tsm{} \(\mu\) of \(F\) satisfies \(M-\mu\rk{\R}\in\Cggq\).
\elem

\bpropl{P1216}
 Let \(\ug\in\R\), let \(G\in\SOFqalpha\), and let \(F\defeq\rstr_\ohe G\).
 Then \(F\in\RFOq\).
 If \(\OSm{G}\) be the \taSmo{G} and if \(\sm{F}\) is the \tSmo{F}, then \(\sma{F}{\R\setminus\rhl}=\Oqq\) and \(\OSm{G}=\rstr_\BorK\sm{F}\).
\eprop
\bproof
 From \rdefn{D1242}, \eqref{F3*3}, and \eqref{RFO} we see that \(F\in\RFOq\).
 The rest was already shown in the proof of~\zitaa{FKM15a}{\cthm{5.1}}.
\eproof

\section{The class \(\SFqas{\kappa}\)}\label{S*5}
 In~\zitaa{MR3611471}{\csec{4}}, we considered particular subclasses of the class \(\SFkaqa\), which was introduced in \eqref{F3*3} for \(\kappa=0\) and in \eqref{F3*11} for all \(\kappa\in\Ninf \).
 In view of \rthm{T3*2}, for each function \(F\) belonging to one of the classes \(\SFkaqa\) with some \(\kappa\in\NOinf \), we can consider the \taSm{} \(\OSm{F}\) of \(F\).
 Now, taking \rrem{R1715} into account, we turn our attention to subclasses of functions \(F\in\SFkaqa\) with prescribed first \(\kappa+1\) power moments of the \taSm{} \(\OSm{F}\).

 For all \(\ug\in\R\), all \(\kappa \in \NOinf \), and each sequence \(\seqska\) of complex \tqqa{matrices}, we consider the class\index{s@\(\SFqaska\)}
 \begin{equation} \label{F5*3}
  \SFqaska
  \defeq\setaa*{F \in\SFkaqa}{\OSm{F}\in \MggqKSg{\kappa}}.
 \end{equation}

\blemnl{\zitaa{MR3611471}{\clem{5.6}}}{L5*8}
 Let \(\ug\in\R\), let \(\kappa\in\NOinf \) and let \(\seqska \) be a sequence of complex \tqqa{matrices}.
 Then \(\SFqas{\kappa}\subseteq\SFdqaa{\su{0}}\).
 If \(\iota\in\NOinf \), with \(\iota\leq\kappa\) then
\[
 \SFqaS{\kappa}
 =\bigcap_{m=0}^\kappa\SFqaS{m}
 \subseteq\SFqaS{\iota}.
\]
\elem

 The following result shows the relevance of the set \(\Kggeq{\kappa}\).

\bthmnl{\zitaa{MR3611471}{\cthm{5.3}}}{T-P5*4}
 Let \(\ug\in\R\), let \(\kappa\in\NOinf\), and let \(\seqska\) be a sequence of complex \tqqa{matrices}.
 Then \(\SFqas{\kappa}\neq \emptyset\) if and only if \(\seqska\in\Kggeqka\).
\ethm

\begin{rem}\label{R1530}
 Let \(\ug\in\R\) and let \(F\colon\Cs\to\Cqq\) be a function.
 Then:
 \benui
  \il{R1530.a} If \(F\in\SFuqa{0}\), then \(F\in\SFqas{0}\) with \(\su{0}\defeq\OSma{F}{\rhl}\).
  \il{R1530.b} Let \(\kappa\in\NOinf\) and let \(\seqska\in\Kggeqka\).
 If \(F\in\SFqas{\kappa}\), then \(F\in\SFuqa{0}\) and \(\OSma{F}{\rhl}=\su{0}\).
 \eenui
\end{rem}

Now we state a useful characterization of the set of functions given in \eqref{F5*3}.

\bthmnl{\zitaa{MR3611471}{\cthm{5.4}}}{T-C5*6}
 Let \(\ug\in\R\), let \(\kappa\in\NOinf\), and let \(\seqska\) be a sequence of complex \tqqa{matrices}.
 In view of \eqref{F3*2}, then
 \[
  \SFqas{\kappa}
  =\setaa*{\OSt{\sigma}}{\sigma\in\MggqKSg{\kappa}}.
 \]
\ethm

 \rthm{T-C5*6} shows that \(\SFqaska \) coincides with the solution set of Problem~\iproblem{\rhl}{\kappa}{=}, which is via \taSt{} equivalent to the original Problem~\mproblem{\rhl}{\kappa}{=}.
 Thus, the investigation of the set \(\SFqaska \) is a central theme of our further considerations.
 At the end of this section, we give a useful technical result.
 It is based on the following:
 
\bpropnl{\zitaa{MR3611471}{\cprop{5.5}}}{P5*7}
 Let \(\ug\in\R\), let \(\kappa\in\NOinf\), let \(\seqska\in\Kggeqka\), and let \(F\in\SFqas{\kappa}\).
 Then:
 \benui
  \il{P5*7.a} \(\ran{F(z)}=\ran{\su{0}}\) and \(\nul{F(z)}=\nul{\su{0}}\) for all \(z\in\Cs\).
  \il{P5*7.b} \([F(z)][F(z)]^\mpi=\su{0}\su{0}^\mpi\) and \([F(z)]^\mpi[F(z)]=\su{0}^\mpi\su{0}\) for all \(z\in\Cs\).
  \il{P5*7.c} The function \(F\) belongs to the class \(\SFqa\) with \(\ran{\su{0}}=\ran{\gammaF}+\ran{\muFa{\rhl}}\) and \(\nul{\su{0}}=\nul{\gammaF}\cap\nul{\muFa{\rhl}}\).
 \eenui
\eprop

\bpropl{P1228}
 Let \(\ug\in\R\), let \(\kappa\in\NOinf\), and let \(\seqs{\kappa}\in\Kggeq{\kappa}\).
 Then \(\SFqas{\kappa}\subseteq\SFqr{\su{0}}\).
\eprop
\bproof
 Let \(F\in\SFqas{\kappa}\).
 In view of \eqref{F5*3}, \eqref{F3*3}, and \eqref{F3*11}, then \(F\in\SFqa\).
 \rprop{P5*7} shows that \(FF^\mpi=\su{0}\su{0}^\mpi\).
 Thus, \(\su{0}\su{0}^\mpi F=F\).
 Hence, \(F\in\SFqr{\su{0}}\).
\eproof

\section{The class \(\SFqaskg{m}\)}\label{S0509}
 In this section, we consider special subclasses of the class \(\SFuqa{m}\), which was introduced in \eqref{F3*3} for \(m=0\) and in \eqref{F3*11} for each \(m\in\N\).
 For all \(\ug\in\R\), all \(m\in\NO\), and each sequence \(\seqs{m}\) of complex \tqqa{matrices}, we consider the class\index{s@\(\SFqaskg{m}\)}
\beql{SFkg}
 \SFqaSkg{m}
 \defeq\setaa*{F \in\SFuqa{m}}{\OSm{F}\in \MggqKSkg{m}}.
\eeq

\breml{R0517}
 Let \(\ug\in\R\), let \(m\in\NO\), and let \(\seqs{m}\) be a sequence from \(\Cqq\).
 In view of \rrem{R1445}, then
 \(
  \SFqas{m}
  \subseteq\SFqaskg{m}
 \).
\erem

\breml{R0520}
 Let \(\ug\in\R\), let \(m\in\N\), and let \(\seqs{m}\) be a sequence from \(\Cqq\).
 Then \rrem{R1721} yields
 \(
  \SFqaskg{m}
  \subseteq\SFqas{\ell}
 \)
 for all \(\ell\in\mn{0}{m-1}\).
\erem

\breml{R0522}
 Let \(\ug\in\R\), let \(m\in\N\), and let \(\seqs{m}\) be a sequence from \(\Cqq\).
 Let \(F\in\SFqaskg{m}\) with \taSm{} \(\OSm{F}\).
 Combining \rrem{R0520} and \rremp{R1530}{R1530.b}, we get then
 \(
  \su{0}
  =\suo{0}{\OSm{F}}
  =\OSma{F}{\rhl}
 \).
\erem

 Now we characterize those sequences for which the sets defined in \eqref{SFkg} are \tne{}.
 (It is a reformulation of \rthmp{T1*3+2}{T1*3+2.b}.)

\bthml{T0527}
 Let \(\ug\in\R\), let \(m\in\NO\), and let \(\seqs{m}\) be a sequence from \(\Cqq\).
 Then \(\SFqaskg{m}\neq \emptyset\) if and only if \(\seqs{m}\in\Kggq{m}\).
 Furthermore, in view of \eqref{F3*2}, if \(\seqs{m}\in\Kggq{m}\), then
 \(
  \SFqaskg{m}
  =\setaa{\OSt{\sigma}}{\sigma\in\MggqKskg{m}}
 \).
\ethm
\bproof
 Combine \eqref{F3*3} and \eqref{SFkg} with \rthmp{T1*3+2}{T1*3+2.b} and \rthm{T3*2}.
\eproof

 In the case \(m\in\N\), \rprop{P5*7} can be modified:
 
\bpropl{P0537}
 Let \(\ug\in\R\), let \(m\in\N\), let \(\seqs{m}\in\Kggq{m}\), and let \(F\in\SFqaskg{m}\).
 Then statements~\eqref{P5*7.a},~\eqref{P5*7.b}, and~\eqref{P5*7.c} of \rprop{P5*7} hold true.
\eprop
\bproof
 In view of \rrem{R1506}, we have \(\seqs{m-1}\in\Kggeq{m-1}\), whereas \rrem{R0520} yields \(F\in\SFqas{m-1}\).
 Thus, applying \rprop{P5*7} completes the proof.
\eproof

\bpropl{P1226}
 Let \(\ug\in\R\), let \(m\in\NO\), and let \(\seqs{m}\in\Kggq{m}\).
 Then \(\SFqaskg{m}\subseteq\SFqr{\su{0}}\).
\eprop
\bproof
 Let \(F\in\SFqaskg{m}\).
 First we consider the case \(m\in\N\).
 In view of \rrem{R1506}, we have \(\seqs{m-1}\in\Kggeq{m-1}\), whereas \rrem{R0520}, yields \(F\in\SFqas{m-1}\).
 Thus, \rprop{P1228} yields \(F\in\SFqr{\su{0}}\).
 
 Now let \(m=0\).
 Because of  the choice of \(F\) and \eqref{SFkg}, we obtain then
\(
 F
 \in\SOFqalpha
\)
 and
\(
 \su{0}-\OSmA{F}{\rhl}
 \in\Cggq
\).
 Hence, we get from \rlem{L3*22} that \(\ran{F(z)}=\ran{\OSma{F}{\rhl}}\) for all \(z\in\Cs\) and from \(\OSma{F}{\rhl}\in\CHq\) and \rrem{R1417} that
\(
 \ran{\OSma{F}{\rhl}}
 \subseteq\ran{\su{0}}
\).
 Consequently, \(\ran{F(z)}\subseteq\ran{\su{0}}\) for all \(z\in\Cs\).
 Hence, \rnota{D1157} shows that \(F\in\SFqr{\su{0}}\).
\eproof

 Now we modify \rlem{L5*8}.
\breml{R0544}
 Let \(\ug\in\R\), let \(m\in\N\), and let \(\seqs{m}\) be a sequence from \(\Cqq\).
 In view of \rrem{R0520}, then \(\SFqaskg{m}\subseteq\SFqas{m-1}\).
 Thus, \rlem{L5*8} shows that \(\SFqaskg{m}\subseteq\SFdqaa{\su{0}}\).
\erem

\section{Stieltjes pairs of meromorphic \tqqa{matrix-valued} functions in \(\Cs\)}\label{S1253}
 Let \(\ug\in\R\).
 Then the set \(\Cs\) is clearly a region in \(\C\).
 We consider a class of ordered pairs of \tqqa{matrix-valued} meromorphic functions in \(\Cs\), which turns out to be closely related to the class \(\SFqa\) introduced in \rdefn{D1242}.
 This set of ordered pairs of \tqqa{matrix-valued} meromorphic functions in \(\Cs\) plays an important role in our subsequent considerations.
 Indeed, this set acts as the set of parameters in our description of the set of Stieltjes transforms of the solutions of our original truncated matricial moment problem at the interval \(\rhl\).

 Now we introduce the central object of this section.
 In doing this, we use the notation \(\ek{\mero{\Cs}}^\x{q}\) introduced in \rapp{A1430} and the particular signature matrix \(\Jimq\)\index{j@\(\Jimq\)} defined by \eqref{Jre}.
 We call a subset \(\cD\) of \(\C\) \notii{discrete} if for every bounded subset \(\cB\) of \(\C\) the intersection \(\cD\cap\cB\) contains only a finite number of points.

\begin{defn} \label{def-sp}
 Let \(\ug\in\R\). Let \(\phi,\psi\in\ek{\mero{\Cs}}^\x{q}\). Then \(\copa{\phi}{\psi}\) is called a \notii{\tqSp{}} if there exists a discrete subset \(\cD\) of \(\Cs \) such that the following three conditions are fulfilled: 
 \begin{aeqi}{0} 
 \item\label{def-sp.i} \(\phi\) are \(\psi\) are holomorphic in \(\C \setminus (\rhl \cup \cD)\). 
 \item\label{def-sp.ii} \(\rank \tmatp{\phi (z)}{\psi(z)} =q\) for each \(z\in \C \setminus (\rhl \cup \cD)\).
 \item\label{def-sp.iii} For each \(z\in \C \setminus (\R \cup \cD)\),
 \begin{align} 
 \matp{\phi (z)}{\psi(z)}^\ad\rk*{\frac{-\Jimq}{2 \im z}}\matp{\phi (z)}{\psi(z)}
 &\in\Cggq\label{KD1}
 \intertext{and} 
 \matp{\rk{z-\ug} \phi (z)}{\psi(z)}^\ad\rk*{\frac{-\Jimq}{2 \im z}}\matp{\rk{z-\ug} \phi (z)}{ \psi(z)}
 &\in\Cggq.\label{KD2} 
 \end{align}
 \end{aeqi} 
 The set of all \tqSps{} will be denoted by \symba{\qSp}{p}.
\edefn

\bdefnl{D1151}
 Let \(\ug\in\R\).
 A pair \(\copa{\phi}{\psi} \in \qSp\) is said to be a \notii{proper \tqSp{}} if \(\det\psi\) does not vanish identically in \(\Cs\).
 The set of all proper \tqSps{} will be denoted by \symba{\qSpp}{p}.
\edefn

\begin{rem} \label{SP.1}
 Let \(\ug\in\R\), let \(\copa{\phi}{\psi} \in \qSp \), and let \(g\) be a \tqqa{matrix-valued} function which is meromorphic in \(\Cs \) such that \(\det g\) does not vanish identically. Then it is readily checked that \(\copa{\phi g}{\psi g} \in \qSp \).
\end{rem} 

 \rrem{SP.1} leads us to the following notion.
 
\bdefnl{D0710}
  Let \(\ug \in \R\) and let \(\copa{\phi_1}{\psi_1},\copa{\phi_2}{\psi_2}\in \qSp \). 
 We will call the Stieltjes pairs \(\copa{\phi_1}{\psi_1}\) and \(\copa{\phi_2}{\psi_2}\) \notii{equivalent} if there is a function \(\theta\in\ek{\mero{\Cs}}^\x{q}\) such that \(\det\theta\) does not identically vanish in \(\Cs\) and that
\beql{equi}
 \matp{\phi_2}{\psi_2}
 =\matp{\phi_1\theta}{\psi_1\theta}
\eeq
 is satisfied.
\edefn

\breml{R1155}
 It is easily checked that the relation introduced in \rdefn{D0710} is an equivalence relation on the set \(\qSp \). 
 For each \(\copa{\phi}{\psi}\in \qSp \), we denote by \symb{\copacl{ \phi}{\psi}} the equivalence class generated by  \(\copa{\phi}{\psi}\). 
 Furthermore, we write \symba{\langle \qSp \rangle}{p} for the set of all these equivalence classes.
\erem

 The following result contains an essential property of \tqS{s}.

\bpropl{P0817}
 Let \(\ug\in\R\) and let \(\copa{\phi}{\psi}\in\qSp\).
 Let \(\cD\) be a discrete subset of \(\Cs\) such that the conditions in \rdefn{def-sp} are satisfied.
 For each \(z\in\lhea\setminus\cD\), then \(\re\ek{\psi^\ad(z)\phi(z)}\in\Cggq\).
\eprop
\bproof
 From the definition of \(\lhea\) we see that
\beql{P0817.-1}
 \lhea\setminus\cD
 =\ek*{\rk{\lhea\setminus\R}\cup(-\infty,\ug)}\setminus\cD
 =\ek*{\lhea\setminus\rk{\R\cup\cD}}\cup\ek*{(-\infty,\ug)\setminus\cD}.
\eeq

 (I)~First we consider an arbitrary \(z\in\lhea\setminus\rk{\R\cup\cD}\).
 In view of condition~\ref{def-sp.iii} of \rdefn{def-sp}, then \eqref{KD1} and \eqref{KD2} hold true.
 Hence, \rrem{R1404} yields
\begin{align}\label{P0817.5}
 \frac{\im\ek{\psi^\ad(z)\phi(z)}}{\im z}&\in\Cggq&
&\text{and}&
 \frac{\im\ek{\rk{z-\ug}\psi^\ad(z)\phi(z)}}{\im z}&\in\Cggq
\end{align}
 follow.
 Thus, taking into account \(\ug-\re z\in(0,\infp)\) and \eqref{P0817.5}, we get
\beql{P0817.7}
 \frac{\ug-\re z}{\im z}\im\ek*{\psi^\ad(z)\phi(z)}
 \in\Cggq.
\eeq
\rrem{R1353} implies \(\im\ek{z\psi^\ad(z)\phi(z)}=\re\rk{z}\im\ek{\psi^\ad(z)\phi(z)}+\im\rk{z}\re\ek{\psi^\ad(z)\phi(z)}\).
 Consequently, because of \(\im z\neq0\), we infer
\beql{P0817.8}%
 \re\ek*{\psi^\ad(z)\phi(z)}
 =\frac{1}{\im z}\im\ek{z\psi^\ad(z)\phi(z)}+\frac{\ug-\re z}{\im z}\im\ek*{\psi^\ad(z)\phi(z)}-\frac{\ug}{\im z}\im\ek*{\psi^\ad(z)\phi(z)}.
\eeq
 In view of \(\ug\in\R\), then \rrem{R1353} yields
\(
 \im\ek{\ug\psi^\ad(z)\phi(z)}
 =\ug\im\ek{\psi^\ad(z)\phi(z)}
\) and, hence, 
\(
 \im\ek*{z\psi^\ad(z)\phi(z)}-\ug\im\ek*{\psi^\ad(z)\phi(z)}
 =\im\ek*{\rk{z-\ug}\psi^\ad(z)\phi(z)}
\).
 In view of, \eqref{P0817.8}, this implies
\beql{P0817.13}
 \re\ek*{\psi^\ad(z)\phi(z)}
 =\frac{\ug-\re z}{\im z}\im\ek*{\psi^\ad(z)\phi(z)}+\frac{\im\ek{\rk{z-\ug}\psi^\ad(z)\phi(z)}}{\im z}.
\eeq
 Combining \eqref{P0817.7}, \eqref{P0817.5}, and \eqref{P0817.13}  yields
\beql{P0817.14}
 \re\ek*{\psi^\ad(z)\phi(z)}
 \in\Cggq.
\eeq

 (II) Now we consider an arbitrary
\(
 z
 \in(-\infty,\ug)\setminus\cD
\).
 Since \(\cD\) is a discrete subset of \(\Cs\), there is a sequence \(\seq{z_n}{n}{1}{\infi}\) from \(\lhea\setminus\rk{\R\cup\cD}\) such that \(\lim_{n\to\infi}z_n=z\).
 By continuity of the functions \(\phi\) and \(\psi\) in \(\C\setminus\rk{\rhl\cup\cD}\), we have then
\beql{P0817.16}
 \lim_{n\to\infi}\re\ek*{\psi^\ad(z_n)\phi(z_n)}
 =\re\ek*{\psi^\ad(z)\phi(z)}.
\eeq
 Since part~(I) of the proof shows that \(\re\ek{\psi^\ad(z_n)\phi(z_n)}\in\Cggq\) holds true for all \(n\in\N\), equation \eqref{P0817.16} implies
\(
 \re\ek{\psi^\ad(z)\phi(z)}
 \in\Cggq
\).

 (III)~In view of \eqref{P0817.-1}, part~(I), and part~(II), the proof is complete.
\eproof
 
 Our next considerations show that there are intimate relations between the class \(\SFqa\) of \tlSFsq{} (see \rdefn{D1242}) and the set of all \tqSps{}.
 
 Denote by \(\OqqCs\) and \(\IqCs\) the constant \(\Cqq\)\nobreakdash-valued functions in \(\Cs\) with values \(\Oqq\) and \(\Iq\), respectively.

\bpropl{P0718}
 Let \(\ug\in\R\) and let \(f\in\SFqa\).
 Then:
 \benui
  \il{P0718.a} The pair \(\copa{f}{\IqCs}\) belongs to \(\qSpp\).
  \il{P0718.b} Let \(g\in\SFqa\).
  Then \(\copacl{f}{\IqCs}=\copacl{g}{\IqCs}\) if and only if \(f=g\).
 \eenui
\eprop
\bproof
 \eqref{P0718.a} Because of \(f\in\SFqa\) and the choice of \(\IqCs\), the functions \(f\) and \(\IqCs\) are holomorphic in \(\Cs\).
 Obviously, \(\rank\tmatp{f(z)}{\IqCs(z)}=q\) for each \(z\in\Cs\).
 We consider arbitrary \(z\in\C\setminus\R\).
 From \rcor{C1031} and~\zitaa{FKM15a}{\clem{4.2}} we conclude then
\(\frac{\im\ek{f(z)}}{\im z}\in\Cggq\) and \(\frac{\im\ek{\rk{z-\ug}f(z)}}{\im z}\in\Cggq\).
 In view of \rrem{R1404}, it follows
\(
 \tmatp{f(z)}{\IqCs(z)}^\ad\rk{\frac{-\Jimq}{2\im z}}\tmatp{f(z)}{\IqCs(z)}
 \in\Cggq
\)
 and
\(
 \tmatp{\rk{z-\ug}f(z)}{\IqCs(z)}^\ad\rk{\frac{-\Jimq}{2\im z}}\tmatp{\rk{z-\ug}f(z)}{\IqCs(z)}
 \in\Cggq
\).
 Thus, the pair \(\copa{f}{\IqCs}\) belongs to \(\qSpp\).
 
 \eqref{P0718.b} In view of \(\copacl{f}{\IqCs}=\copacl{g}{\IqCs}\) there exists a \tqqa{matrix-valued} function \(h\) which is meromorphic in \(\Cs\) such that \(\det h\) does not identically vanish and such that \(\tmatp{f}{\IqCs}=\tmatp{gh}{h}\).
 This implies first \(h=\IqCs\) and, consequently, \(f=g\).
\eproof

 \rprop{P0718} shows that the class \(\qSp\) can be considered as a projective extension of the class \(\SFqa\).
 
\bexaml{E1211}
 Let \(\ug\in\R\).
 Then \rprop{P0718} shows that \(\copa{\OqqCs}{\IqCs}\) belongs to \(\qSpp\).
 Furthermore, \rrem{R1404} yields \(\copa{\IqCs}{\OqqCs}\in\qSp\).
\eexam

\breml{R1221}
 Let \(\ug\in\R\).
 Then \rexam{E1211} shows that the set \(\qSpp\) is non-empty.
\erem

\bpropl{P0809}
 Let \(\ug\in\R\).
 Further, let \(\phi\) be a \tqqa{matrix-valued} function which is meromorphic in \(\Cs\) such that the condition \(\copa{\phi}{\IqCs}\in\qSp\) is satisfied.
 Then \(\phi\in\SFqa\).
\eprop
\bproof
 Our strategy of proof is based on applying \rprop{214-1}.
 In view of \rdefn{def-sp} and the choice of \(\phi\), we choose a discrete subset \(\cD\) of \(\Cs\) such that the following conditions are satisfied:
\begin{aeqi}{0} 
 \il{P0809.i} The function \(\phi\) is holomorphic in \(\C \setminus (\rhl \cup \cD)\). 
 \il{P0809.ii} For each \(z\in \C \setminus (\R \cup \cD)\),
\[
 \set*{\matp{\phi (z)}{\Iq}^\ad\rk*{\frac{-\Jimq}{2 \im z}}\matp{\phi (z)}{\Iq},\matp{\rk{z-\ug} \phi (z)}{\Iq}^\ad\rk*{\frac{-\Jimq}{2 \im z}}\matp{\rk{z-\ug} \phi (z)}{ \Iq}}
 \subseteq\Cggq.
\]
\end{aeqi}
 Obviously, we have
\beql{P0809.1}
 \Cs
 =\lhea\cup\ohe\cup\uhe.
\eeq
 From \eqref{P0809.1} we get
\beql{P0809.2}
 \C\setminus\rk*{\rhl\cup\cD}
 =\rk{\lhea\setminus\cD}\cup\rk{\ohe\setminus\cD}\cup\rk{\uhe\setminus\cD}.
\eeq
 In particular, from \ref{P0809.i} and \eqref{P0809.2} we conclude that \(\phi\) is holomorphic in each of the regions \(\lhea\setminus\cD\), \(\ohe\setminus\cD\), and \(\uhe\setminus\cD\).
 In view of~\ref{P0809.i}, \rprop{P0817} shows that \(\re\ek{\phi(z)}\in\Cggq\) for each \(z\in \lhea\setminus\cD\).
 Thus, \rlemp{L1508}{L1508.a} yields that \(\phi\) is holomorphic in \(\lhea\) and satisfies \(\re\ek{\phi(z)}\in\Cggq\) for all \(z\in\lhea\).
 For all \(z\in\C\setminus\rk{\R\cup\cD}\), we infer from \rrem{R1404} and~\eqref{P0809.ii} that
\(\frac{\im\ek{\phi(z)}}{\im z}\in\Cggq\) holds true.
 In view of \(\C\setminus\rk{\R\cup\cD}=\rk{\ohe\setminus\cD}\cup\rk{\uhe\setminus\cD}\), we obtain then \(\im\ek{\phi(z)}\in\Cggq\) for all \(z\in\ohe\setminus\cD\) and \(-\im\ek{\phi(z)}\in\Cggq\) for all \(z\in\uhe\setminus\cD\).
 Consequently, since \(\phi\) is holomorphic in \(\ohe\setminus\cD\) and \(\uhe\setminus\cD\), the application of \rlemp{L1508}{L1508.b} yields that the function \(\phi\) is holomorphic in \(\ohe\) and satisfies \(\im\ek{\phi(z)}\in\Cggq\) for all \(z\in\ohe\) and that the function \(\phi\) is holomorphic in \(\uhe\) and satisfies \(-\im\ek{\phi(z)}\in\Cggq\) for all \(z\in\uhe\).
 Taking into account \eqref{P0809.1}, we in particular see that \(\phi\) is holomorphic in \(\Cs\).
 The application of \rprop{214-1} brings \(\phi\in\SFqa\).
\eproof

 The following results complement the statements of \rpropss{P0718}{P0809}.
 We see now that the equivalence class of a proper element of \(\qSp\) is always represented by a function belonging to \(\SFqa\).
 
\bpropl{P1544}
 Let \(\ug\in\R\) and let \(\copa{\phi}{\psi} \in \qSpp\).
 Then:
\benui
 \il{P1544.a} The function \(S\defeq\phi\psi^\inv\) belongs to \(\SFqa\).
 \il{P1544.b} 
  \(\copa{S}{\IqCs}\in\qSp\) and \(\copacl{\phi}{\psi}=\copacl{S}{\IqCs}\).
\eenui
\eprop
\bproof
 By the choice of \(\copa{\phi}{\psi}\), we know that the function \(\det\psi\) does not identically vanish in \(\Cs\).
 Thus,~\eqref{P1544.b} follows from \rrem{SP.1} and \rdefn{D0710}.
 In view of~\eqref{P1544.b}, \rprop{P0809} yields \(S\in\SFqa\).
\eproof

\bcorl{C0726}
 Let \(\ug\in\R\) and let \(\rho_\ug\colon\rcset{\qSpp}\to\SFqa\) be defined by
\beql{rho}
 \rho_\ug\rk*{\copacl{\phi}{\psi}}
 \defeq\phi\psi^\inv.
\eeq
 Then \(\rho_\ug\) is well defined and bijective with inverse \(\iota_\ug\colon\SFqa\to\rcset{\qSpp}\) given by
\beql{iota}
 \iota_\ug\rk{F}
 \defeq\copacl{F}{\IqCs}.
\eeq
\ecor
\bproof
 
 Consider two pairs \(\copa{\phi_1}{\psi_1},\copa{\phi_2}{\psi_2}\in\qSpp\) with \(\copacl{\phi_1}{\psi_1}=\copacl{\phi_2}{\psi_2}\).
 In view of \rdefn{D0710}, there is a function \(\theta\in\ek{\mero{\Cs}}^\x{q}\) such that \(\det\theta\) does not identically vanish in \(\Cs\) satisfying \eqref{equi}, \ie{}, \(\phi_2=\phi_1\theta\) and \(\psi_2=\psi_1\theta\).
 Consequently, \(\phi_1\psi_1^\inv=\phi_2\psi_2^\inv\).
 Thus, since \rpropp{P1544}{P1544.a} shows that \(\phi\psi^\inv\in\SFqa\) for all \(\copa{\phi}{\psi}\in\qSpp\), the mapping \(\rho_\ug\) is well defined.
 From \rpropp{P0718}{P0718.a} we can conclude that \(\iota_\ug\) is well defined, too.
 For all \(F\in\SFqa\), we have obviously
\(
 \rk{\rho_\ug\circ\iota_\ug}\rk{F}
 =\rho_\ug\rk{\copacl{F}{\IqCs}}=F
\).
 Using \rpropp{P1544}{P1544.b}, we conclude
\(
 \rk{\iota_\ug\circ\rho_\ug}\rk{\copacl{\phi}{\psi}}
 =\iota_\ug\rk{\phi\psi^\inv}
 =\copacl{\phi\psi^\inv}{\IqCs}
 =\copacl{\phi}{\psi}
\)
 for all \(\copa{\phi}{\psi}\in\qSpp\).
\eproof

 Now we turn our attention to a particular subclass of \(\qSp\), which occupies an essential role in our subsequent considerations.
 
\bnotal{D1230}
 Let \(\ug\in\R\) and let \(A\in\Cqq\).
 We denote  by \symba{\qSpa{A}}{p} the set of all \(\copa{\phi}{\psi}\in\qSp\) such that \(\ran{\phi(z)}\subseteq\ran{A}\) for all points \(z\in\Cs\) which are points of holomorphicity of \(\phi\).
 Further, let
\(
 \qSppa{A}\defeq\qSpa{A}\cap\qSpp
\).
\enota

\bexaml{E1147}
 Let \(\ug\in\R\) and let \(A\in\Cqq\).
 Then the pair \(\copa{\OqqCs}{\IqCs}\) introduced in \rexam{E1211} belongs to \(\qSpa{A}\).
\eexam

\breml{R1244}
 Let \(\ug\in\R\) and let \(A\in\Cqq\) be such that \(\det A\neq0\).
 Then \(A^\mpi=A^\inv\) and, consequently,
 \(
  \qSpa{A}
  =\qSp
 \).
\erem

\begin{rem}\label{L1149}
 Let \(\ug \in \R\). 
 In view of \rdefn{def-sp}, \rnota{D1230}, and \rexam{E1147}, one can easily check that \(\rcset{\qSpa{ \Oqq }}=\set{\copacl{\OqqCs}{\IqCs}}\). 
\end{rem}

 Now we will see that the procedure of constructing subclasses of \(\qSp\) via \rnota{D1230} stands in full harmony with the equivalence relation in \(\qSp\) introduced in \rdefn{D0710}.
 
\bleml{L1252}
 Let \(\ug\in\R\), let \(A\in\Cqq\), and let \(\copa{\phi_1}{\psi_1}\in\qSpa{A}\).
 Let \(\copa{\phi_2}{\psi_2}\in\qSp\) be such that \(\copacl{\phi_1}{\psi_1}=\copacl{\phi_2}{\psi_2}\).
 Then \(\copa{\phi_2}{\psi_2}\in\qSpa{A}\).
\elem
\bproof
 By assumption and \rremp{L1631}{L1631.b}, we have
\(
 AA^\mpi\phi_1
 =\phi_1
\)
 and there exists a meromorphic \tqqa{matrix-valued} function \(g\) in \(\Cs\) with non-identically vanishing determinant such that
\(
 \tmatp{\phi_2}{\psi_2}
 =\tmatp{\phi_1}{\psi_1}g
\).
 Thus,
\(
 AA^\mpi\phi_2
 =AA^\mpi\rk{\phi_1g}
 =\rk{AA^\mpi\phi_1}g
 =\phi_1g
 =\phi_2
\).
 Hence,
\(
 \copa{\phi_2}{\psi_2}
 \in\qSpa{A}
\).
\eproof

 Now we turn our attention again to the topic opened by \rprop{P0817}.
 The following result shows that the class \(\qSpa{A}\) can be considered as a projective extension of the class \(\SFqr{A}\) introduced in \rnota{D1157}.
 
\breml{P1139}
 Let \(\ug\in\R\) and let \(f\in\SFqa\).
 According to \rpropp{P0718}{P0718.a}, then:
 \benui
  \il{P1139.a} \(\copa{f}{\IqCs}\) belongs to \(\qSpp\).
  \il{P1139.b} Let \(A\in\Cqq\).
  Then \(f\in\SFqr{A}\) if and only if \(\copa{f}{\IqCs}\in\qSpa{A}\).
 \eenui
\erem

\section{On a coupled pair of Schur-Stieltjes-type transforms}\label{S*8}
 The main goal of this section is to recall some basic facts which were necessary for the preparation of the elementary step of our Schur type algorithm for the class \(\SFqa\).
 The material is mostly taken from~\zitaa{MR3611471}{\csec{9}}.
 We will be led to a situation which, roughly speaking, looks as follows: Let \(\ug\in\R\), let \(A\in\Cpq\) and let \(F\colon\Cs\to\Cpq\).
 Then the matrix-valued functions \(\STao{F}{A}\colon\Cs\to\Cpq\)\index{\(\STao{F}{A}\)} and \(\STiao{F}{A}\colon\Cs\to\Cpq\)\index{\(\STiao{F}{A}\)} which are defined by
\beql{F8*1}
 \STaoa{F}{A}{z}
 \defeq-A\rk*{\Iq+(z-\ug)^\inv\ek*{F(z)}^\mpi A}
\eeq
and
\beql{F8*2}
 \STiaoa{F}{A}{z}
 \defeq-(z-\ug)^\inv   A\ek*{ \Iq  + A^\mpi  F(z)}^\mpi,
\eeq
 respectively, will be central objects in our further considerations.
 Against to the background of our later considerations, the matrix-valued functions \(\STao{F}{A}\) and \(\STiao{F}{A}\) are called the \notii{\taSSt{A}} of \(F\) and the \notii{\tiaSSt{A}} of \(F\).

 The generic case studied here concerns the situation where \(p=q\), where \(A\) is a complex \tqqa{matrix} with later specified properties and where \(F\in\SFqa\).
 
 An important theme of~\zita{MR3611471} was to choose, for a given function \(F\in\SFqa\), special matrices \(A\in\Cqq\) such that the function \(\STao{F}{A}\) and \(\STiao{F}{A}\), respectively, belong to \(\SFqa\) (see~\zitaa{MR3611471}{\cpropss{9.5}{9.10}}).
 Furthermore, from~\zitaa{MR3611471}{\cpropss{9.11}{9.13}} we know that under appropriate conditions the equations \(\STao{\rk{\STiao{F}{A}}}{A}=F\) and \(\STiao{\rk{\STao{F}{A}}}{A}=F\) hold true.
 
 Now we verify that in essential cases formulas \eqref{F8*1} and \eqref{F8*2} can be rewritten as linear fractional transformations with appropriately chosen generating matrix-valued functions.
 The role of these generating functions will be played by the matix polynomials \(\mHTu{A}\) and \(\mHTiu{A}\), which are studied in \rapp{A1557}.
 In the sequel, we use the terminology for linear fractional transformations of matrices, which are introduced in \rapp{S*B}.

\blemnl{\zitaa{MR3611471}{\clem{9.6}}}{L8*7}
 Let \(\ug\in\R\), let \(F\colon\Cs\to\Cpq\) be a matrix-valued function, and let \(A\in\Cpq\) be such that \(\ran{F(z)}\subseteq\ran{A}\) and \(\nul{F(z)}\subseteq\nul{A}\) for all \(z\in\Cs\).
 Then \(F(z)\in\dblftruu{-(z-\ug)A^\mpi}{\Iq-A^\mpi A}\) and \(\STaoa{F}{A}{z}=\lftrooua{p}{q}{\mHTua{A}{z}}{F(z)}\) for all \(z\in\Cs\).
\elem

 The following application of \rlem{L8*7} prepares our considerations in \rsec{S1302}.

\bpropnl{\zitaa{MR3611471}{\cprop{9.7}}}{L8*9}
 Let \(\ug\in\R\), let \(\kappa\in\NOinf\), let \(\seqska\in\Kggequa{\kappa}\), and let \(F\in\SFqas{\kappa}\).
 Then \(F(z)\in\dblftruu{-(z-\ug)\su{0}^\mpi}{\Iq-\su{0}^\mpi\su{0}}\) and \(\lftrooua{q}{q}{\mHTua{\su{0}}{z}}{F(z)}=\STao{F}{\su{0}}(z)\) for all \(z\in\Cs\).
\eprop

 At the end of this section we are going to consider the situation which will turn out to be typical for larger parts of our following considerations.
 Let \(A\in\Cggq\) and let \(G\) belong to the class \(\SFqr{A}\) introduced in \rnota{D1157}.
 Our aim is then to investigate the function \(\STiao{G}{A}\) given by \eqref{F8*2}.
 We begin by rewriting formula~\eqref{F8*2} as linear fractional transformation.
 In the sequel, we will often use the fact that, for each \(G\in\SFqa\), the matrix \(\gammaG\) given via \rnota{N1123} is \tnnH{}.

\bpropnl{\zitaa{MR3611471}{\cprop{9.9}}}{L8*12}
 Let \(\ug\in\R\), let \(A\in\Cggq\), and let \(G\in\SFdqaa{A}\).
 Then \(G(z)\in\dblftruu{(z-\ug)A^\mpi}{(z-\ug)\Iq}\) and \(\STiaoa{G}{A}{z}=\lftrooua{q}{q}{\mHTiua{A}{z}}{G(z)}\) for all \(z\in\Cs\).
\eprop

\section{On the transform \(\STao{F}{\su{0}}\) for functions \(F\) belonging to the class \(\SFqaskg{m}\)}\label{S0815}
 
 In~\zitaa{MR3611471}{\cSect{10}}, we considered the following situation:
 Let \(\ug\in\R\), let \(m\in\NO\), and let \(\seqs{m}\in\Kggeq{m}\).
 Then \rthm{T-P5*4} yields that the class \(\SFqas{m}\) is non-empty.
 If \(F\in\SFqas{m}\), then our interest in~\zitaa{MR3611471}{\cSect{10}} was concentrated on the \taSSt{\su{0}} \(\STao{F}{\su{0}}\) of \(F\).
 The following result on this theme is of fundamental importance.
 
\bthmnl{\zitaa{MR3611471}{\cthm{10.3}}}{T0837}
 Let \(\ug\in\R\), let \(m\in\N\), let \(\seqs{m}\in\Kggequa{m}\) with \saScht{} \(\seq{\su{j}^\sta{1}}{j}{0}{m-1}\), and let \(F\in\SFqas{m}\).
 Then \(\STao{F}{\su{0}}  \in\SFuqaa{m-1}{\seq{\su{j}^\sta{1}}{j}{0}{m-1}}\).
\ethm

 Now we will consider a function \(F\in\SFqaskg{m}\).
 We are interested in its \taSSt{\su{0}}.

\bthml{T149_5}
 Let \(\ug\in\R\), let \(m \in \N\), let \(\seqs{m} \in \Kggqua{m}\) with \saScht{} \(\seq{\su{j}^\sta{1}}{j}{0}{m-1}\), and let \(F \in \SFqaskg{m}\).
 Then \(\STao{F}{\su{0}}\) belongs to \(\SFuqaakg{m-1}{\seq{\su{j}^\sta{1}}{j}{0}{m-1}}\).
\ethm
\begin{proof}
 Because of \eqref{SFkg}, we see that
\(
 F
 \in\SFuqa{m}
\)
 and
\(
 \OSm{F}
 \in\MggqKskg{m}
\).
 In particular, this implies
\(
 \OSm{F}
 \in\MggquK{m}
\).
 We set
\begin{align}\label{T149_5.4}
 t_j&\defeq\suo{j}{\OSm{F}}&\text{for all }j&\in\mn{0}{m}.
\end{align}
 Because of \eqref{T149_5.4}, the application of \rcor{C1022} yields
\beql{T149_5.5}
 \seqt{m}
 \in\Kggeq{m}.
\eeq
 Furthermore, from \eqref{T149_5.4} and \eqref{F5*3} we conclude
\beql{T149_5.6}
 F
 \in\SFqaT{m}.
\eeq
 Using \eqref{T149_5.5} and \rrem{R1502}, we get \(\seqt{m}\in\Kggq{m}\).
 Because of this and \(\seqs{m}\in\Kggq{m}\), the application of \rlemp{R1738}{R1738.a} yields \(\su{j}^\ad=\su{j}\) and \(t_{j}^\ad=t_{j}\) for all \(j\in\mn{0}{m}\).
 From \eqref{T149_5.4} and the choice of \(F\) we infer
\begin{align}\label{T149_5.9}
 \su{j}&=t_{j}\text{ for all }j\in\mn{0}{m-1}&
&\text{and}&
 t_{m}&\leq\su{m}.
\end{align}
 We denote by \(\seq{\su{j}^\seinsalpha}{j}{0}{m-1}\) and \(\seq{t_{j}^\seinsalpha}{j}{0}{m-1}\) the \tlasnt{\ug}s of \(\seqs{m}\) and \(\seqt{m}\), respectively.
 Since \(\su{j}^\ad=\su{j}\) and \(t_{j}^\ad=t_{j}\) are true for each \(j\in\mn{0}{m}\) and \eqref{T149_5.9} is valid, \rlem{T149_4} yields that \(\seq{\su{j}^\seinsalpha}{j}{0}{m-1}\) and \(\seq{t_{j}^\seinsalpha}{j}{0}{m-1}\) are sequences of \tH{} matrices which satisfy
\beql{T149_5.10}
 t_{m-1}^\seinsalpha
 \leq\su{m-1}^\seinsalpha
\eeq
 and, in the case \(m\geq2\), moreover
\begin{align}\label{T149_5.11}
 \su{j}^\seinsalpha&=t_{j}^\seinsalpha&
 \text{for all }j&\in\mn{0}{m-2}.
\end{align}
 In view of \eqref{T149_5.5} and \eqref{T149_5.6}, the application of \rthm{T0837} yields \(\STao{F}{\su{0}}\in\SFuqaa{m-1}{\seq{t_{j}^\seinsalpha}{j}{0}{m-1}}\).
 Combining this with the identity \(\su{0}=t_{0}\), which follows from \eqref{T149_5.9}, we get
\(
 \STao{F}{\su{0}}
 \in\SFuqaa{m-1}{\seq{t_{j}^\seinsalpha}{j}{0}{m-1}}
\).
 Because of \eqref{T149_5.10} and \eqref{T149_5.11}, this implies that \(\STao{F}{\su{0}}\in\SFuqaakg{m-1}{\seq{\su{j}^\seinsalpha}{j}{0}{m-1}}\).
\end{proof}

\section{On the transform \(\STiao{F}{\su{0}}\) for functions \(F\in\SFuqaakg{m-1}{\seq{\su{j}^\seinsalpha}{j}{0}{m-1}}\)}\label{S0559}

 Let \(\ug\in\R\), let \(m\in\N\), let \(\seqs{m}\in\Kggeq{m}\) with \tseinsalpha{} \(\seq{\su{j}^\seinsalpha}{j}{0}{m-1}\), and let \(F\in\SFuqaa{m-1}{\seq{\su{j}^\seinsalpha}{j}{0}{m-1}}\).
 Then our interest in~\zitaa{MR3611471}{\csec{11}} was concentrated on the \tiaSSt{\su{0}} \(\STiao{F}{\su{0}}\) of \(F\).
 The following result on this theme is of fundamental importance for our considerations.
 
\bthmnl{\zitaa{MR3611471}{\cthm{11.3}}}{T10*9}
 Let \(\ug\in\R\), let \(m\in\N\), let \(\seqs{m}\in\Kggeqka\) with \tseinsalpha{} \(\seq{\su{j}^\seinsalpha}{j}{0}{m-1}\), and let \(F\in\SFuqaa{m-1}{\seq{\su{j}^\seinsalpha}{j}{0}{m-1}}\).
 Then \(\STiao{F}{\su{0}}\) belongs to \(\SFqas{m}\).
\ethm

 Against to the background of \rthmss{T10*9}{T0527}, we are lead to the investigation of the following question:
 Let \(\ug\in\R\), let \(m\in\N\), let \(\seqs{m}\in\Kggq{m}\) with \tseinsalpha{} \(\seq{\su{j}^\seinsalpha}{j}{0}{m-1}\), and let \(F\in\SFuqaakg{m-1}{\seq{\su{j}^\seinsalpha}{j}{0}{m-1}}\).
 Does then \(\STiao{F}{\su{0}}\) belong to \(\SFqaskg{m}\)?
 The following example shows that the answer to this question is negative.

\begin{exa} \label{T2110_3}
 Let \(\su{0} \defeq\tmat{1 & 0 \\ 0 & 0}\), let \(s_1 \defeq\tmat{1 & 1 \\ 1 & 1}\), and let \(F\colon\C \setminus [0, \infp) \to \Coo{2}{2}\) be given by \(F(z) \defeq  - \frac{1}{z} \su{0}\). 
 Then \((s_j)_{j=0}^1 \in \Kgguuu{2}{1}{0}\). 
 However, because of, \(\ran{\su{1}} \nsubseteq \ran{\su{0}}\), the sequence \((s_j)_{j=0}^1\) does not belong to \(\Duuu{2}{2}{1}\). 
 Thus, \((s_j)_{j=0}^1 \in \Kgguuu{2}{1}{0} \setminus\Duuu{2}{2}{1}\). 
 Obviously, because of \(\su{0}^\mpi = \so\) and \(\su{0}\su{0} = \so\), we have
\[
 \su{0}^{[1, 0]}
 = -\su{0} s_1^{[\sharp, 0]} \su{0} 
 = -\su{0} \ek*{ - \rk{ \su{0}^{[+,0]}}^\mpi  s_1^{[+,0]}\rk{ \su{0}^{[+,0]}}^\mpi  } \su{0} 
 = \su{0} \su{0}^\mpi  (-\NM\cdot \su{0} + s_1) \su{0}^\mpi  \su{0}
 = \su{0}.
\]
 Let \(\delta_0\) be the Dirac measure defined on \(\Bori{[0, \infp)}\) with unit mass at the point \(0\) and let \(\sigma \defeq  \su{0} \delta_0\). 
 Then \(\sigma\in \Mggoaakg{2}{}{[0,\infp)}{\seq{\su{j}^\sntaa{1}{0}}{j}{0}{0}}\). 
 One can easily see that \(\int_{[0,\infp)}{\rk{t - z}}^\inv\sigma(\dif t)= F(z)\) for each \(z\in \C\setminus[0, \infp)\).
 In view of \rthm{T3*2}, consequently, \(F\in \SFuuuakg{0}{2}{0}{\seq{\su{j}^\sntaa{1}{0}}{j}{0}{0}}\). 
 Obviously, if \(\delta_1\) is the Dirac measure defined on \(\Bori{[0, \infp)}\) with unit mass at the point \(1\), then \(\mu \defeq  \su{0}\delta_1\in\Mggoa{2}{[0,\infp)}\). 
 In view of \rremp{R1530}{R1530.a} and \(\su{0}^\mpi \su{0} = \su{0}\su{0} = \su{0}\), we get then 
\beql{T2110_3_7}\begin{split} 
 F^{[-,0, \su{0}]}(z) 
 &= -(z-0)^{-1}\su{0}\ek*{\Iu{2} + \su{0}^\mpi F(z)}]^\mpi  
 = -\frac{1}{z} \su{0} \rk*{ \Iu{2} - \frac{1}{z} \su{0}^\mpi  \su{0}}^\mpi  \\
 &=  \su{0} \rk{ \su{0} - z \Iu{2}}^\mpi   
 = \frac{1}{1-z}\su{0}
 = \int_{[0, \infp)} \frac{1}{t-z} \mu(\dif t).
\end{split}\eeq
 Furthermore, we obtain \(s_1^{(\mu)} = \int_{[0,\infp)} t \mu(\dif t) = \so\) and, hence, \(s_1 - s_1^{(\mu)} =\tmat{0 & 1 \\ 1 & 1 }\notin \Cggo{2}\). 
 This implies \(\mu \notin\Mggoaakg{2}{}{[0,\infp)}{\seqs{1}}\). 
 Because of \eqref{T2110_3_7}, then \(F^{[-, 0, \su{0}]} \notin \SFuuuakg{0}{2}{0}{\seqs{1}}\) follows.
\end{exa}

 As a consequence of \rexam{T2110_3} we have to look for a ``large'' proper subclass \(\KggDq{m}\) of \(\Kggq{m}\) with the following property:
 
 Let \(\ug\in\R\), let \(m\in\N\), let \(\seqs{m}\in\KggDq{m}\) with \tseinsalpha{} \(\seq{\su{j}^\seinsalpha}{j}{0}{m-1}\), and let \(F\in\SFuqaakg{m-1}{\seq{\su{j}^\seinsalpha}{j}{0}{m-1}}\).
 Then \(\STiao{F}{\su{0}}\) belongs to \(\SFqaskg{m}\).
 
 The search for such a class \(\KggDq{m}\) determines the direction of our next considerations.
 We will see that the class \(\Dqqu{m}\) of \tftd{} sequences \(\seqs{m}\) of complex \tqqa{matrices} (see \rdefn{D1658}) will turn out to provide the key for finding the desired class \(\KggDq{m}\).
 In order to prepare the proof of the main result of this section, we note that some technical results on the class \(\Dqqu{m}\) are given in~\zita{MR3611479}.
 In particular, we have the following:
 
\bpropnl{\zitaa{MR3611479}{\cprop{3.8(a)}}}{P1442}
 Let \(\ug\in\R\) and let \(m\in\NO\).
 Then \(\Kggeq{m}\subseteq\Hggequ{m}\subseteq\Dqqu{m}\).
\eprop

 Now we come to the central result of this section.
 
\bthml{T0904}
 Let \(\ug\in\R\), let \(m\in\N\), and let \(\seqs{m}\in\Kggq{m}\cap\Dqqu{m}\) with \tseinsalpha{} \(\seq{\su{j}^\seinsalpha}{j}{0}{m-1}\).
 Then:
 \benui
  \il{T0904.a} \(\seq{\su{j}^\seinsalpha}{j}{0}{m-1}\in\Kggq{m-1}\) and \(\SFuqaakg{m-1}{\seq{\su{j}^\seinsalpha}{j}{0}{m-1}}\neq\emptyset\).
  \il{T0904.c} If \(F\in\SFuqaakg{m-1}{\seq{\su{j}^\seinsalpha}{j}{0}{m-1}}\), then \(\STiao{F}{\su{0}}\in\SFqaskg{m}\).
 \eenui
\ethm
\bproof
 \rPart{T0904.a} follows immediately from \rthmp{P1546}{P1546.a} and \rthmp{T1*3+2}{T1*3+2.b}.
  Let \(F\in\SFuqaakg{m-1}{\seq{\su{j}^\seinsalpha}{j}{0}{m-1}}\).
  In view of~\eqref{T0904.a}, the application of \rlemp{R1738}{R1738.a} yields \(\su{j}^\seinsalpha\in\CHq\) for all \(j\in\mn{0}{m-1}\).
 Because of \eqref{SFkg}, we have
\(
 F
 \in\SFuqa{m-1}
\)
 and
\(
 \OSm{F}
 \in\MggquKakg{m-1}{\seq{\su{j}^\seinsalpha}{j}{0}{m-1}}
\).
  In particular, we infer
\(
 \OSm{F}
 \in\MggquK{m-1}
\).
 We set
\begin{align}\label{T0904.5}
 t_j&\defeq\suo{j}{\OSm{F}}&\text{for all }j&\in\mn{0}{m-1}.
\end{align}
 Taking into account \eqref{T0904.5}, the application of \rcor{C1022} yields
\(
 \seqt{m-1}
 \in\Kggeq{m-1}
\).
 By virtue of \rrem{R1502}, we get
\(
 \seqt{m-1}
 \in\Kggq{m-1}
\).
 Thus, \rlemp{R1738}{R1738.a} yields
\begin{align}\label{T0904.8}
 t_{j}&\in\CHq&\text{for all }j&\in\mn{0}{m-1},
\end{align}
 whereas \rlemp{R1738}{R1738.b} gives 
\(
 t_0
 \in\Cggq
\).
 From \eqref{T0904.5} and the choice of \(F\) we infer
\beql{T0904.10}
 t_{m-1}\leq\su{m-1}^\seinsalpha
\text{ and, in the case, }m\geq2\text{ moreover }
 t_{j}=\su{j}^\seinsalpha\text{ for all }j\in\mn{0}{m-2}.
\eeq
 In particular,
\(
 t_{0}
 \leq\su{0}^\seinsalpha
\).
 Combining \(\su{0}^\seinsalpha\in\CHq\) with \(t_0\in\Cggq\) and \(t_{0}\leq\su{0}^\seinsalpha\), then \rrem{R1417} and \rdefn{D1059} give
\begin{align}\label{T0904.13}
 \ran{t_0}&\subseteq\ran{ \su{0}^\seinsalpha  }\subseteq\ran{\su{0}}&
&\text{and}& 
 \nul{\su{0}}&\subseteq\nul{ \su{0}^\seinsalpha  }\subseteq\nul{t_0}.
\end{align}
 In view of \(\seqs{m}\in\Kggq{m}\), \rlemp{R1738}{R1738.b} gives
\(
 \su{0}
 \in\Cggq
\).
 Denote by \(\seqr{m}\) the \tsminuseinsalphaaa{\seqt{m-1}}{\su{0}}.
 Taking into account \(\seqt{m-1}\in\Kggeq{m-1}\) and \(\su{0}\in\Cggq\), then the application of~\zitaa{MR3611479}{\cprop{10.15}} yields
\(
 \seqr{m}
 \in\Kggeq{m}
\).
 In view of \(\seqt{m-1}\in\Kggeq{m-1}\), we infer from \rprop{P1442} that
\(
 \seqt{m-1}
 \in\Dqqu{m-1}
\).
 Denote by \(\seq{r_j^\seinsalpha}{j}{0}{m-1}\) the \tseinsalphaa{\seqr{m}}.
 Then, because of \(\seqt{m-1}\in\Dqqu{m-1}\) and \eqref{T0904.13}, the application of~\zitaa{MR3611479}{\cprop{10.8}} yields \(r_j^\seinsalpha=t_{j}\) for all \(j\in\mn{0}{m-1}\).
 Combining this with \eqref{T0904.5}, we get \(\OSm{F}\in\MggquKag{m-1}{\seq{r_{j}^\seinsalpha}{j}{0}{m-1}}\) and hence, in view of \eqref{F5*3}, then
\beql{T0904.20}
 F
 \in\SFuqAA{m-1}{\seq{r_{j}^\seinsalpha}{j}{0}{m-1}}.
\eeq
 From \(\seqt{m-1}\in\Kggeq{m-1}\) and \(r_j^\seinsalpha=t_{j}\) for all \(j\in\mn{0}{m-1}\) we infer
\(
 \seq{r_{j}^\seinsalpha}{j}{0}{m-1}
 \in\Kggeq{m-1}
\).
 Thus, in view of \eqref{T0904.20}, the application of \rthm{T10*9} implies
\beql{T0904.22}
 \STiao{F}{\su{0}}
 \in\SFuqAA{m}{\seqr{m}}.
\eeq
 Denote by \(\seq{v_j}{j}{0}{m}\) the \tsminuseinsalphaaa{\seq{\su{j}^\seinsalpha}{j}{0}{m-1}}{\su{0}}.
 Because of \eqref{T0904.8}--\eqref{T0904.10}, we infer from \rlem{T2110_1} that
\(
 r_j=v_j\) for all \(j\in\mn{0}{m-1}\) and that
\(
 r_m
 \leq v_m
\).
 Because of the assumption \(\seqs{m}\in\Dqqu{m}\), then~\zitaa{MR3611479}{\cprop{10.10}} yields \(\su{j}=v_j\) for all \(j\in\mn{0}{m}\).
 Consequently, from \eqref{T0904.22} we get \(\STiao{F}{\su{0}}\in\SFqaskg{m}\).
\eproof

 Finally, we turn our attention to an important consequence of \rthm{T0904}.
 
\bthml{T1140}
 Let \(\ug\in\R\), let \(m\in\N\), let \(\seqs{m}\in\Kggeq{m}\) with \tseinsalpha{} \(\seq{\su{j}^\seinsalpha}{j}{0}{m-1}\).
 Then \(\seq{\su{j}^\seinsalpha}{j}{0}{m-1}\in\Kggq{m-1}\) and \(\SFuqaakg{m-1}{\seq{\su{j}^\seinsalpha}{j}{0}{m-1}}\neq\emptyset\).
 Furthermore, if \(F\in\SFuqaakg{m-1}{\seq{\su{j}^\seinsalpha}{j}{0}{m-1}}\), then \(\STiao{F}{\su{0}}\in\SFqaskg{m}\).
\ethm
\bproof
 In view of \rrem{R1502} and \rprop{P1442}, the inclusion \(\Kggeq{m}\subseteq\Kggq{m}\cap\Dqqu{m}\) holds true.
 Hence, the application of \rthm{T0904} completes the proof.
\eproof

 Taking into account \rthmss{T1*3+2}{T1605}, we recognize the particular importance of \rthm{T1140} for the treatment of \rprob{\mproblem{\rhl}{m}{\leq}}.

\section{On the set \(\SFqaskg{0}\)}\label{S1351}
 In this section, for arbitrarily given \(\ug\in\R\) and \(\su{0}\in\Cggq\), we study the set \(\SFqaskg{0}\) introduced in \eqref{SFkg}.
 We will see that this set can be parametrized by a linear fractional transformation with the linear \taaa{2q}{2q}{matrix} polynomial \(\mHTiu{\su{0}}\) given via \eqref{VWaA} as generating matrix-valued function.
 The role of parameters will be taken by the set \(\qSpa{\su{0}}\) of pairs \(\copa{\phi}{\psi}\) of meromorphic \tqqa{matrix-valued} functions, which were introduced in \rnota{D1230}.

\begin{lem} \label{T2110_4}
 Let \(\ug\in\R\), let \(\su{0} \in \Cqqg\), and let \(F \in \SFqaskg{0}\). 
 For all \(w \in \C \setminus \R\), then
\beql{T2110_4_B1}
\frac{1}{\im w} \im F(w)
\geq \ek*{ F(w) }^\ad  \su{0}^\mpi  \ek*{ F(w) }.
\eeq
\end{lem}
\begin{proof}
 Let \(w \in \C \setminus \R\).
 Because of \(F \in \SFqaskg{0}\), we have \(F \in \SFuqa{0}\).
 Let \(\OSm{F}\) be the \taSm{} of \(F\) and let \(g_w\colon\rhl \to \C\) be defined by \(g_w(t) = \rk{t-w}^\inv\).
 Then \(g_w \in \LaaaC{\rhl}{\BorK}{\OSm{F}}\) and 
\beql{T2110_4_4}
F(w) 
= \int_{\rhl} g_w \dif\OSm{F}.
\eeq
 Hence, \(\im g_w \in \LaaaC{\rhl}{\BorK}{\OSm{F}}\) and
\beql{T2110_4_7}
\im F(w)
= \int_{\rhl} \im (g_w) \dif\OSm{F}.
\eeq
 For all \(t \in \rhl\), we have \(\im g_w(t)=\ek{\im\rk{w}}\abs{t-w}^{-2} = \im (w) \abs{g_w(t)}^2\) and, consequently, \(\abs{g_w(t)}^2 =\im g_w(t)/\im (w)\).
 This implies 
\begin{align} \label{T2110_4_10}
\abs{g_w}^2&\in \LaaaC{\rhl}{\BorK}{\OSm{F}}&
&\text{and}&
\int_{\rhl} \abs{g_w}^2 \dif\OSm{F}& = \frac{1}{\im (w)} \int_{\rhl} \im (g_w) \dif\OSm{F}.
\end{align}
 Comparing \eqref{T2110_4_7} and \eqref{T2110_4_10}, we conclude
\beql{T2110_4_11}
\im F(w)
=\ek*{\im(w)}\int_{\rhl} \abs{g_w}^2 \dif\OSm{F}.
\eeq
 From \eqref{T2110_4_10} and \rcor{AMH16}, we obtain 
\beql{T2110_4_12}
 \rk*{ \int_{\rhl} g_w \dif\OSm{F} }^\ad  \ek*{ \OSmA{F}{\rhl } }^\mpi  \rk*{ \int_{\rhl} g_w \dif\OSm{F} }
 \leq \int_{\rhl} \abs{g_w}^2 \dif\OSm{F}.
\eeq
 The matrices \(B\defeq  \OSma{F}{\rhl}\) and \(\su{0}\) are both \tnnH{}.
 In particular, using \rremp{L1631}{L1631.a}, we get \(B = B^\ad\) and \(\su{0} = \su{0}^\ad\).
 By virtue of \(F \in \SFqaskg{0}\),  we see that \(\OSm{F} \in \MggqKskg{0}\) and, consequently, \(\su{0} \geq \suo{0}{\OSm{F}} = \OSma{F}{\rhl} = B \geq \Oqq\) hold true. 
 Thus, taking additionally into account \(B^\ad=B$, \(\su{0}^\ad=\su{0}\), and \rlem{L1315}, we infer 
\beql{T2110_4_18}
 B^\mpi  
 \geq B^\mpi B \su{0}^\mpi  B B^\mpi .
\eeq
 Since \(F\) belongs to \(\SFuqa{0}\), in view of \rlem{L3*22} and \(B=\OSma{F}{\rhl}\) we obtain \(\ran{F(z)}=\ran{B}\) for all \(z\in\Cs\).
 Thus, \rremp{L1631}{L1631.b} implies
\beql{T2110_4_19}
 BB^\mpi F 
 = F.
\eeq
 Because of \(B^\ad=B\) and \rremp{L1631}{L1631.a}, we conclude \(B^\mpi B = (BB^\mpi )^\ad \). 
 Consequently, using additionally \eqref{T2110_4_11}, \eqref{T2110_4_12}, \eqref{T2110_4_4}, \(B=\OSma{F}{\rhl}\),  \eqref{T2110_4_18}, and \eqref{T2110_4_19}, we get then
\[\begin{split}
 &\frac{1}{\im w} \im F(w) 
 = \int_{\rhl} \abs{g_w}^2 \dif\OSm{F} 
 \geq \rk*{ \int_{\rhl} g_w \dif\OSm{F} }^\ad  \ek*{ \OSm{F}\rk*{\rhl}}^\mpi  \rk*{ \int_{\rhl} g_w \dif\OSm{F} }\\
 &= \ek*{F(w)}^\ad  B^\mpi  \ek*{F(w)} 
 \geq \ek*{F(w)}^\ad B^\mpi B \su{0}^\mpi  BB^\mpi  F(w)
 = \ek*{F(w)}^\ad (BB^\mpi )^\ad  \su{0}^\mpi  BB^\mpi  F(w) \\
 &=\ek*{BB^\mpi F(w)}^\ad  \su{0}^\mpi\ek*{BB^\mpi F(w)}
 =\ek*{F(w)}^\ad  \su{0}^\mpi  F(w).\qedhere
\end{split}\]
\end{proof}

\bpropl{T2110_5}
 Let \(\ug\in\R\), let \(\su{0} \in \Cqqg\), and let \(F \in \SFqaskg{0}\). 
 Further, let \(\mHTu{ \su{0}}\) be given by \eqref{VWaA} and let
\beql{T2110_5_V3}
\mHTu{ \su{0}} 
\cdot
\matp{F}{\Iq}
=
\matp{\phi}{\psi}
\eeq
 be the \tqqa{block} representation of \(\mHTu{ \su{0}} \cdot \tmatp{F}{\Iq}\).
 Then:
 \benui
  \il{T2110_5.a} The functions \(\phi\) and \(\psi\) are holomorphic in \(\Cs\).
  \il{T2110_5.d} The pair \(\copa{\phi}{\psi}\) belongs to \(\qSpa{\su{0}}\).
  \il{T2110_5.e} The inequality \(\det\ek{\rk{z-\ug}\su{0}^\mpi\phi(z)+\rk{z-\ug}\Iq\cdot\psi(z)}\neq0\) and the equation
\[
  F(z)
  =\ek*{\Oqq\cdot\phi(z)-\su{0}\psi(z)}\ek*{\rk{z-\ug}\su{0}^\mpi\phi(z)+\rk{z-\ug}\Iq\cdot\psi(z)}^\inv
\]
 hold true for all \(z\in\Cs\).
 \eenui
\eprop
\begin{proof}
 \eqref{T2110_5.a} Since \(F\) belongs to \(\SFqaskg{0}\), we have \(F \in \SFuqa{0}\) and \(F \in \Sqa\). 
 In particular, \(F\) is holomorphic in \(\Cs\).
 Thus, since \(\mHTu{\su{0}}\) is a \taaa{2q}{2q}{matrix} polynomial, we see from \eqref{T2110_5_V3} that \(\phi\) and \(\psi\) are holomorphic in \(\Cs\) as well.
 
 \eqref{T2110_5.d} For each \(z \in \Cs\), using \eqref{VWaA}, we obtain
 \[
\mat{\su{0}^\mpi,\Iq}
\mHTu{ \su{0}}(z)
=  
\mat{\su{0}^\mpi,\Iq}
\begin{bmatrix}
\rk{z-\ug} \Iq  & \su{0} \\ 
-\rk{z-\ug}\su{0}^\mpi  & \Iq  - \su{0}^\mpi \su{0}
\end{bmatrix}
=
\mat{\Oqq,\Iq}
\]
and, in view of \eqref{T2110_5_V3}, then 
\[\begin{split}
 q 
 = \rank \Iq
 = \rank \rk*{\mat{ \Oqq , \Iq}\matp{F(z)}{\Iq} } 
 &= \rank \rk*{ \mat{\su{0}^\mpi,\Iq}\mHTu{ \su{0}}(z)\matp{F(z)}{\Iq} }  \\
 &=\rank \rk*{ \mat{\su{0}^\mpi,\Iq}\matp{\phi(z)}{\psi(z)} } 
 \leq\rank \matp{\phi(z)}{\psi(z)}  
 \leq q.
\end{split}\]
 This implies
\begin{align}\label{T2110_5_1}
 \rank\matp{\phi(z)}{\psi(z)}&=q&\text{for all }z \in \Cs.
\end{align}
 Now consider an arbitrary
\(
 w \in \C \setminus \R
\).
 In view of \(F \in \Sqa\), then~\zitaa{FKM15a}{\clem{4.2}} yields
\beql{T2110_5_12}
 \frac{\im [\rk{w-\ug}F(w)]}{\im w}
 \in \Cqqg.
\eeq
 By virtue of of \(\su{0}\in\Cggq\), we have \(\su{0}^\ad=\su{0}\).
 Thus, \rprop{T2110_40} implies%
\beql{T2110_5_3}
\ek*{ \mHTua{\su{0}}{w} }^\ad  
\rk{ -\Jq }
\mHTua{\su{0}}{w}
=\ek*{ \diag\rk*{\rk{w-\ug}\Iq , \Iq} }^\ad 
\rk{ -\Jq }
\ek*{ \diag\rk*{\rk{w-\ug}\Iq , \Iq} }.
\eeq
 Because of \eqref{T2110_5_V3}, \eqref{T2110_5_3}, \rrem{R1404}, and \eqref{T2110_5_12}, we infer
\beql{T2110_5_4}\begin{split} 
&\matp{\phi(w)}{\psi(w)}^\ad \rk*{ \frac{-\Jimq}{2 \im w} }\matp{\phi(w)}{\psi(w)}
=\frac{1}{2 \im w}
\matp{F(w)}{\Iq}^\ad  \ek*{ \mHTu{ \su{0}}(w) }^\ad \rk{ -\Jq }\ek*{ \mHTu{ \su{0}}(w) }\matp{F(w)}{\Iq}  \\
&=\frac{1}{2 \im w}\matp{F(w)}{\Iq}^\ad  \ek*{ \diag\rk*{(w -\alpha)\Iq , \Iq} }^\ad \rk{ -\Jq }\ek*{ \diag\rk*{(w -\alpha)\Iq , \Iq} }\matp{F(w)}{\Iq}  \\
&=\frac{1}{2 \im w}\matp{(w -\alpha)F(w)}{\Iq}^\ad  \rk{ -\Jq }\matp{(w -\alpha)F(w)}{\Iq} 
=\frac{\im [\rk{w-\ug}F(w)]}{\im w}
\in\Cggq.
\end{split}\eeq
 From \(F \in \SFqaskg{0}\) we see that the \taSm{} \(\OSm{F}\) of \(S\) satisfies \(\OSm{F} \in \MggqKskg{0}\), which implies \(\Oqq \leq \OSma{F}{\rhl} = \suo{0}{\OSm{F}} \leq \su{0}\).
 Thus, \rrem{R1417} yields \(\ran{\OSma{F}{\rhl}} \subseteq \ran{\su{0}}\). 
 Hence, in view of \(F \in \SFuqa{0}\) and \rlem{L3*22}, we get that \(\ran{F(z)}=\ran{\OSm{F}{\rhl}} \subseteq \ran{\su{0}}\) for each \(z\in \Cs\).
 Consequently, by virtue of \rremp{L1631}{L1631.b}, we obtain then 
\beql{T2110_5_16}
\su{0}\su{0}^\mpi F
= F .
\eeq
 In view of \(\su{0}^\ad=\su{0}\), the application of \eqref{N101N} provides
\begin{multline}\label{T2110_5_5}
 \rk*{ \diag\rk*{\rk{w-\ug}\Iq , \Iq}\ek*{\mHTua{\su{0}}{w}}}^\ad 
\rk{ -\Jq }
\rk*{ \diag\rk*{\rk{w-\ug}\Iq , \Iq}\ek*{\mHTua{\su{0}}{w}}}  \\
= \abs{w-\ug}^2\rk*{
\ek*{ \diag\rk{ \su{0}\su{0}^\mpi , \Iq  }  }^\ad  
\rk{ -\Jimq} \ek*{ \diag\rk{ \su{0}\su{0}^\mpi , \Iq  }} - 2(\im w)\ek*{\diag \rk{ \su{0}^\mpi , \Oqq }}} \\
 + \ek*{ \diag \rk*{ \rk{w-\ug}^2(\Iq  - \su{0}\su{0}^\mpi ), \Iq } }^\ad 
\rk{ -\Jq }
\ek*{ \diag \rk*{ \rk{w-\ug}^2(\Iq  - \su{0}\su{0}^\mpi ), \Iq } }.
\end{multline}
 \rrem{R1404} shows that
\begin{align}\label{T2110_5_16_2}
\matp{F(w)}{\Iq}^\ad \rk{ - \Jq }\matp{F(w)}{\Iq}&=\im\ek*{F(w)}&
&\text{and}&
\matp{\Oqq}{\Iq}^\ad \rk{ - \Jq }\matp{\Oqq}{\Iq}&=\Oqq.
\end{align}
 By virtue of \(\su{0} \in \Cqqg\) and \(F \in \SFqaskg{0}\), \rlem{T2110_4} yields \eqref{T2110_4_B1}.
 Using \eqref{T2110_5_V3}, \eqref{T2110_5_5}, \eqref{T2110_5_16}, \eqref{T2110_5_16_2}, and \eqref{T2110_4_B1}, we conclude
\beql{T2110_5_7}\begin{split} 
&\matp{\rk{w-\ug} \phi(w)}{\psi(w)}^\ad \rk*{ \frac{-\Jimq}{2 \im w} }\matp{\rk{w-\ug} \phi(w)}{\psi(w)}  \\
&=\frac{1}{2 \im w} \matp{F(w)}{\Iq}^\ad \ek*{ \diag\rk*{(w -\alpha)\Iq , \Iq} \mHTu{ \su{0}}(w) }^\ad \rk{ -\Jq } \\
&\qquad\times\ek*{ \diag\rk*{(w -\alpha)\Iq , \Iq} \mHTu{ \su{0}}(w) }\matp{F(w)}{\Iq}  \\
&=\frac{1}{2 \im w} \matp{F(w)}{\Iq}^\ad\\
&\qquad\times\biggl\{\abs{w - \alpha}^2\rk*{\ek*{\diag\rk{\su{0}\su{0}^\mpi , \Iq}}^\ad \rk{ -\Jq }\ek*{\diag\rk{\su{0}\su{0}^\mpi , \Iq}}- 2\rk{\im w}\diag \rk{ \su{0}^\mpi , \Oqq }}\\
&\qquad+ \ek*{\diag\rk*{(w- \alpha)^2(\Iq -\su{0}\su{0}^\mpi ), \Iq }}^\ad \rk{ -\Jimq}\ek*{\diag\rk*{(w- \alpha)^2(\Iq -\su{0}\su{0}^\mpi ), \Iq}}\biggr\}\matp{F(w)}{\Iq}  \\
&= \frac{\abs{w-\alpha}^2}{2 \im w}\matp{\su{0}\su{0}^\mpi F(w)}{\Iq}^\ad \rk{ -\Jimq}\matp{\su{0}\su{0}^\mpi F(w)}{\Iq}- \abs{w - \alpha}^2\ek*{F(w)}^\ad \su{0}^\mpi \ek*{F(w)}  \\
&\qquad+ \frac{1}{2 \im w}\matp{\rk{w-\ug}^2(\Iq  - \su{0}\su{0}^\mpi )F(w)}{\Iq}^\ad \rk{ -\Jimq}\matp{\rk{w-\ug}^2(\Iq  - \su{0}\su{0}^\mpi )F(w)}{\Iq}  \\
&= \frac{\abs{w-\alpha}^2}{2 \im w}\matp{F(w)}{\Iq}^\ad \rk{ -\Jimq}\matp{F(w)}{\Iq}- \abs{w - \alpha}^2 \ek*{F(w)}^\ad \su{0}^\mpi \ek*{F(w)}  \\
&\qquad+ \frac{1}{2 \im w}\matp{\Oqq}{\Iq}^\ad \rk{ -\Jimq}\matp{\Oqq}{\Iq}  \\
&=\abs{w-\alpha}^2 \rk*{ \frac{1}{\im w} \im F(w) - \ek*{F(w)}^\ad \su{0}^\mpi \ek*{F(w)} }
\in\Cggq.
\end{split}\eeq
 From \eqref{T2110_5.a}, \eqref{T2110_5_1}, \eqref{T2110_5_4}, and \eqref{T2110_5_7} we see that
\(
\copa{\phi}{\psi}
\in \qSp 
\)
 holds true. 
 Thus, in view of \rremp{L1631}{L1631.a}, in order to complete the proof of \rpart{T2110_5.d}, it remains to check that \(\su{0}\su{0}^\mpi \phi = \phi\). 
 Since \(\phi\) and \(F\) are holomorphic in \(\Cs\) and, because of \eqref{T2110_5_V3} and \eqref{VWaA}, for all \(z \in \Cs\), we get 
\[\begin{split}
 \phi(z) 
 = \mat{\Iq, \Oqq} \mHTu{ \su{0}}(z)\matp{F(z)}{\Iq}
 &=\mat{\Iq,\Oqq}
 \bMat
  \rk{z-\ug}\Iq&\su{0}\\
  -\rk{z-\ug}\su{0}^\mpi&\Iq-\su{0}^\mpi\su{0}
 \eMat
 \matp{F(z)}{\Iq}\\
 &=\mat*{\rk{z-\ug}\Iq,\su{0}}\matp{F(z)}{\Iq}
 = \rk{z-\ug} F(z) + \su{0}
\end{split}\]
 and, according to \eqref{T2110_5_16}, consequently \(\su{0}\su{0}^\mpi  \phi(z)  = \rk{z-\ug}F(z)+\su{0} = \phi(z)\).
 
 \eqref{T2110_5.e} Let \(z\in\Cs\).
 Our proof is based on an application of \rprop{R_BEG_1}.
 The roles of the matrices \(E_1\) and \(E_2\) in \rprop{R_BEG_1} will be played by the matrices
\begin{align}\label{T2110_5.0}
 \mHTua{\su{0}}{z}
 &=\bMat\rk{z-\ug}\Iq&\su{0}\\-\rk{z-\ug}\su{0}^\mpi&\Iq-\su{0}^\mpi\su{0}\eMat&
&\text{and}&
 \mHTiua{\su{0}}{z}
 &=\bMat\Oqq&-\su{0}\\\rk{z-\ug}\su{0}^\mpi&\rk{z-\ug}\Iq\eMat.
\end{align}
 Taking into account \rrem{MR3611471_D2_2}, we have
\begin{align}\label{T2110_5.3}
 \rank\mat*{-\rk{z-\ug}\su{0},\Iq-\su{0}^\mpi\su{0}}
 &=q&
&\text{and}&
 \rank\mat*{\rk{z-\ug}\su{0}^\mpi,\rk{z-\ug}\Iq}
 =q.
\end{align}
 From \rrem{MR3611471_D1} we get
\beql{T2110_5.4}
 \ek*{\mHTiua{\su{0}}{z}}\ek*{\mHTua{\su{0}}{z}}
 =\rk{z-\ug}\cdot\diag\rk{\su{0}\su{0}^\mpi,\Iq}.
\eeq
 Because of \(z\in\Cs\), we obtain
\(
 \det\ek*{\rk{z-\ug}\cdot\Oqq\cdot F(z)+\rk{z-\ug}\Iq\cdot\Iq}
 =\rk{z-\ug}^q
 \neq0
\)
 and, in view of \eqref{T2110_5_16}, then
\begin{multline}\label{T2110_5.7}
 \ek*{\rk{z-\ug}\su{0}\su{0}^\mpi F(z)+\Oqq\cdot\Iq}\ek*{\rk{z-\ug}\cdot\Oqq\cdot F(z)+\rk{z-\ug}\Iq\cdot\Iq}^\inv\\
 =\rk{z-\ug}\su{0}\su{0}^\mpi\ek*{F(z)}\ek*{\rk{z-\ug}\Iq}^\inv
 =\su{0}\su{0}^\mpi F(z)
 =F(z).
\end{multline}
 By virtue of \eqref{T2110_5.0}--\eqref{T2110_5.7}, the application of \rprop{R_BEG_1} yields the inequality
\(
 \det\ek{\rk{z-\ug}\su{0}^\mpi\phi(z)+\rk{z-\ug}\Iq\cdot\psi(z)}
 \neq0
\)
 and the equation
\(
 \ek*{\Oqq\cdot\phi(z)-\su{0}\psi(z)}\ek*{\rk{z-\ug}\su{0}^\mpi\phi(z)+\rk{z-\ug}\Iq\cdot\psi(z)}^\inv
 =F(z)
\).
\end{proof}

\bpropl{P1218}
 Let \(\ug\in\R\) and let \(\su{0} \in \Cqqg\).
 Let \(\mHTiu{\su{0}}\) be given by \eqref{VWaA}.
 Further, let \(\copa{\phi}{\psi} \in \qSpa{\su{0}}\), and let
\beql{P1218.V1}
 \mHTiu{ \su{0}} \matp{\phi}{\psi}
 = \matp{X}{Y}
\eeq
 be the \tqqa{block} representation of \(\mHTiu{ \su{0}} \tmatp{\phi}{\psi}\). 
 Let \(\cD\) be a discrete subset of \(\Cs\) such that conditions~\ref{def-sp.i} and~\ref{def-sp.ii} in \rdefn{def-sp} are fulfilled.
 Then:
\benui
 \il{P1218.a} \(\det Y(z) \neq 0\) for all \(z \in \C \setminus \rk{ \rhl \cup \cD }\).
 \il{P1218.b} The functions \(X\) and \(Y\) are holomorphic in \(\C\setminus\rk{\rhl\cup\cD}\).
 \il{P1218.e} \(\copa{X}{Y}\) belongs to \(\qSpp\).
 \il{P1218.f} The function \(\det Y\) does not identically vanish in \(\Cs\) and \(F\defeq XY^\inv\) belongs to \(\SFqaskg{0}\).
\eenui
\eprop
\begin{proof}
 (I)~From the assumption \(\copa{\phi}{\psi} \in \qSpa{\su{0}}\) we see that 
\(
 \copa{\phi}{\psi}
 \in\qSp
\)
 is valid and, in view of \rremp{L1631}{L1631.b}, that
\beql{P1218.2}
 \su{0}\su{0}^\mpi \phi
 = \phi
\eeq
 holds true.
 Because of \rrem{T149_2}, the function \(\mHTiu{\su{0}}\) is holomorphic in \(\C\).
 Thus, taking into account that, by condition~\ref{def-sp.i} of \rdefn{def-sp}, the functions \(\phi\) and \(\psi\) are holomorphic in \(\C\setminus\rk{\rhl\cup\cD}\), we conclude from \eqref{P1218.V1} then:
\baeqi{3}
 \il{P1218.iv} The functions \(X\) and \(Y\) are holomorphic in \(\C\setminus\rk{\rhl\cup\cD}\).
\eaeqi
 Let
\(
 z
 \in\C\setminus\rk{\rhl\cup\cD}
\).
 Using~\ref{P1218.iv} and \eqref{VWaA}, we obtain
\[\begin{split}
 \matp{X(z)}{Y(z)}
 &=\mHTiua{\su{0}}{z}\matp{\phi(z)}{\psi(z)}\\
 &=\bMat\Oqq&-\su{0}\\\rk{z-\ug}\su{0}^\mpi&\rk{z-\ug}\Iq\eMat\matp{\phi(z)}{\psi(z)}
 =\matp{-\su{0}\psi(z)}{\rk{z-\ug}\ek{\su{0}^\mpi\phi(z)+\psi(z)}}.
\end{split}\]
 Hence,
\begin{align}\label{P1218.4}
 X(z)&=-\su{0}\psi(z)&
&\text{and}&
 Y(z)&=\rk{z-\ug}\ek*{\su{0}^\mpi\phi(z)+\psi(z)}.
\end{align}
 Now we are going to verify that \(\det Y(z)\neq0\).
 Let
\beql{P1218.5}
 v
 \in\Nul{Y(z)}.
\eeq
 Applying \eqref{P1218.4} and \eqref{P1218.5}, we infer
\(
 \rk{z-\ug}\ek{\su{0}^\mpi\phi(z)+\psi(z)}v
 =\ek{Y(z)}v
 =\Ouu{q}{1}
\)
 and, consequently,
\(
 \psi(z) v 
 = -\su{0}^\mpi \phi(z) v
\).
 Thus, \eqref{P1218.2} implies
\beql{P1218.7}
 - \su{0} \psi(z) v
 = \su{0} \su{0}^\mpi \phi(z) v
 = \phi(z) v.
\eeq
 Taking into account \eqref{P1218.7} and \(\su{0} \in \Cqqg\), we get 
\beql{P1218.8}
 -v^\ad  \ek*{\psi(z)}^\ad\ek*{\phi(z)} v 
 = v^\ad  \ek*{\psi(z)}^\ad  \su{0}\ek*{\psi(z)} v
 =\ek*{\sqrt{\so}\psi(z)v}^\ad\ek*{\sqrt{\so}\psi(z)v}
 = \normEs*{\sqrt{\so}\psi(z) v}.
\eeq 
 Obviously,
\beql{P1218.10}
 \C\setminus\rk*{\rhl\cup\cD}
 =\rk{\lhea\setminus\cD}\cup\ek*{\C\setminus\rk{\R\cup\cD}}.
\eeq
 
 (II)~Now we consider the particular case that
\beql{P1218.11}
 z
 \in\lhea\setminus\cD.
\eeq
 Because of \eqref{P1218.11}, \eqref{P1218.10}, and \eqref{P1218.8}, we have then
\beql{P1218.12}
 -v^\ad \psi^\ad(z) \phi(z)v
 =\normEs*{\sqrt{\so}\psi(z) v}.
\eeq
 In view of \eqref{P1218.10}, the choice of \(\cD\), and \eqref{P1218.11}, the application of \rprop{P0817} yields
\(
 \re\ek*{\psi^\ad(z)\phi(z)}
 \in\Cggq
\).
 Using \eqref{P1218.12}, we obtain then
\[
 0
 \leq\normEs*{\sqrt{\so}\psi(z) v}
 =\re\ek*{\normEs*{\sqrt{\so}\psi(z) v}}
 =\re\ek*{-v^\ad \psi^\ad(z) \phi(z)v}
 =-v^\ad\re\ek*{\psi^\ad(z) \phi(z)}v
 \leq0.
\]
 Thus, \(\normE{\sqrt{\so}\psi(z) v}=0\) and, consequently,
\beql{P1218.14}
 \sqrt{\so}\psi(z) v
 =\Ouu{q}{1}.
\eeq

 (III)~Now we consider the further particular case that
\(
 z
 \in\C\setminus\rk{\R\cup\cD}
\).
 Because of \eqref{P1218.10} and \eqref{P1218.8}, we have then
\beql{P1218.16}
 -v^\ad \psi^\ad(z) \phi(z)v
 =\normEs*{\sqrt{\so}\psi(z) v}.
\eeq
 In view of \eqref{P1218.10} and the choice of \(\cD\), we infer from condition~\ref{def-sp.iii} in \rdefn{def-sp} that \eqref{KD2} holds true.
 Hence, \rrem{R1404} yields
\beql{P1218.18}
 \frac{\im\ek{\rk{z-\ug}\psi^\ad(z)\phi(z)}}{\im z}
 \in\Cggq.
\eeq
 Using \(\ug\in\R\) and applying \eqref{P1218.16}, we get
\begin{multline}\label{P1218.20}
 -v^\ad\rk*{\frac{1}{\im z}\im\ek*{\rk{z-\ug}\psi^\ad(z)\phi(z)}}v
 =\frac{\im\rk{\rk{z-\ug}\ek{-v^\ad\psi^\ad(z)\phi(z)v}}}{\im z}\\
 =\frac{\im\ek{\rk{z-\ug}\normEs{\sqrt{\so}\psi(z) v}}}{\im z}
 =\frac{\im\rk{z-\ug}}{\im z}\normEs*{\sqrt{\so}\psi(z) v}
 =\normEs*{\sqrt{\so}\psi(z) v}.
\end{multline}
 Combining \eqref{P1218.20} and \eqref{P1218.18}, we obtain
\[
 0
 \leq\normEs*{\sqrt{\so}\psi(z) v}
 =-v^\ad\rk*{\frac{1}{\im z}\im\ek*{\rk{z-\ug}\psi^\ad(z)\phi(z)}}v
 \leq0.
\]
 Thus, \(\normE{\sqrt{\so}\psi(z) v}=0\).
 Consequently, \eqref{P1218.14} holds true.
 
 (IV)~Now we consider again the general case
\(
 z
 \in\C\setminus\rk{\rhl\cup\cD}
\).
 In view of \eqref{P1218.10}, part~(II), and part~(III), then \eqref{P1218.14} is proved.
 Using \eqref{P1218.7} and \eqref{P1218.14}, it follows
\(
 \phi(z)v
 =\su{0}\psi(z)v
 =\sqrt{\su{0}}\sqrt{\su{0}}\psi(z)v
 =\Ouu{q}{1}
\).
 The last equations and \eqref{P1218.5} imply
\(
 \tmatp{\phi(z)}{\psi(z)}v
 =\Ouu{2q}{1}
\).
 Because of \(\copa{\phi}{\psi}\in\qSp\) and the choice of \(\cD\), then condition~\ref{def-sp.ii} in \rdefn{def-sp} yields \(v=\Ouu{q}{1}\).
 Thus, taking into account \eqref{P1218.5}, we have \(\det Y(z)\neq0\).
 The proof of \rpart{P1218.a} is complete.
 
 \eqref{P1218.b} This follows from~\ref{P1218.iv}.
 
 \eqref{P1218.e} \rPart{P1218.a} provides \(\rank\tmatp{X(z)}{Y(z)}=q\) for all \(z\in\C\setminus\rk{\rhl\cup\cD}\).
 Now we consider an arbitrary \(z\in\C\setminus\rk{\R\cup\cD}\).
 From condition~\ref{def-sp.iii} in \rdefn{def-sp} we infer that \eqref{KD1} holds true.
 In view of \(\su{0}\in\Cggq\subseteq\CHq\), the application of \rcoro{C0932} yields
\begin{multline}\label{P1218.28}
 \ek*{\mHTiua{\su{0}}{z}}^\ad\rk{-\Jimq}\ek*{\mHTiua{\su{0}}{z}}\\
 =\ek*{\diag\rk*{\rk{z-\ug}\su{0},\su{0}^\mpi}}^\ad\rk{-\Jimq}\ek*{\diag\rk*{\rk{z-\ug}\su{0},\su{0}^\mpi}}+2\im\rk{z}\diag\rk{\Oqq,\su{0}}
\end{multline}
 and
\begin{multline}\label{P1218.29}
 \ek*{\diag\rk*{\rk{z-\ug}\Iq,\Iq}\mHTiua{\su{0}}{z}}^\ad\rk{-\Jimq}\ek*{\diag\rk*{\rk{z-\ug}\Iq,\Iq}\mHTiua{\su{0}}{z}}\\
 =\abs{z-\ug}^2\ek*{\diag\rk{\su{0},\su{0}^\mpi}}^\ad\rk{-\Jimq}\ek*{\diag\rk{\su{0},\su{0}^\mpi}}.
\end{multline}
 In view of \(\su{0}^\ad=\su{0}\), \rlem{L0827} yields
\beql{P1218.31}
  \ek*{\diag\rk{\su{0},\su{0}^\mpi}}^\ad\rk{-\Jimq}\ek*{\diag\rk{\su{0},\su{0}^\mpi}}
  =\ek*{\diag\rk{\su{0}\su{0}^\mpi,\Iq}}^\ad\rk{-\Jimq}\ek*{\diag\rk{\su{0}\su{0}^\mpi,\Iq}}.
\eeq
 and
\begin{multline}\label{P1218.32}
 \ek*{\diag\rk*{\rk{z-\ug}\su{0},\su{0}^\mpi}}^\ad\rk{-\Jimq}\ek*{\diag\rk*{\rk{z-\ug}\su{0},\su{0}^\mpi}}\\
 =\ek*{\diag\rk*{\rk{z-\ug}\su{0}\su{0}^\mpi,\Iq}}^\ad\rk{-\Jimq}\ek*{\diag\rk*{\rk{z-\ug}\su{0}\su{0}^\mpi,\Iq}}.
\end{multline}
 Using \eqref{P1218.V1} and \eqref{P1218.28}, we get
\beql{P1218.33}\begin{split}
 &\matp{X(z)}{Y(z)}^\ad\rk*{\frac{-\Jimq}{2\im z}}\matp{X(z)}{Y(z)}
 =\rk*{\mHTiua{\su{0}}{z}\matp{\phi(z)}{\psi(z)}}^\ad\rk*{\frac{-\Jimq}{2\im z}}\rk*{\mHTiua{\su{0}}{z}\matp{\phi(z)}{\psi(z)}}\\
 &=\matp{\phi(z)}{\psi(z)}^\ad\ek*{\mHTiua{\su{0}}{z}}^\ad\rk*{\frac{-\Jimq}{2\im z}}\ek*{\mHTiua{\su{0}}{z}}\matp{\phi(z)}{\psi(z)}\\
 &=\matp{\phi(z)}{\psi(z)}^\ad\rk*{\ek*{\diag\rk*{\rk{z-\ug}\su{0},\su{0}^\mpi}}^\ad\rk*{\frac{-\Jimq}{2\im z}}\ek*{\diag\rk*{\rk{z-\ug}\su{0},\su{0}^\mpi}}+\diag\rk{\Oqq,\su{0}}}\\
 &\qquad\times\matp{\phi(z)}{\psi(z)}\\
 &=\matp{\phi(z)}{\psi(z)}^\ad\ek*{\diag\rk*{\rk{z-\ug}\su{0},\su{0}^\mpi}}^\ad\rk*{\frac{-\Jimq}{2\im z}}\ek*{\diag\rk*{\rk{z-\ug}\su{0},\su{0}^\mpi}}\matp{\phi(z)}{\psi(z)}\\
 &\qquad+\matp{\phi(z)}{\psi(z)}^\ad\diag\rk{\Oqq,\su{0}}\matp{\phi(z)}{\psi(z)}.
\end{split}\eeq
 Applying \eqref{P1218.32} and \eqref{P1218.2}, it follows
\beql{P1218.34}\begin{split}
 &\matp{\phi(z)}{\psi(z)}^\ad\ek*{\diag\rk*{\rk{z-\ug}\su{0},\su{0}^\mpi}}^\ad\rk*{\frac{-\Jimq}{2\im z}}\ek*{\diag\rk*{\rk{z-\ug}\su{0},\su{0}^\mpi}}\matp{\phi(z)}{\psi(z)}\\
 &=\matp{\phi(z)}{\psi(z)}^\ad\ek*{\diag\rk*{\rk{z-\ug}\su{0}\su{0}^\mpi,\Iq}}^\ad\rk*{\frac{-\Jimq}{2\im z}}\ek*{\diag\rk*{\rk{z-\ug}\su{0}\su{0}^\mpi,\Iq}}\matp{\phi(z)}{\psi(z)}\\
 &=\rk*{\diag\rk*{\rk{z-\ug}\su{0}\su{0}^\mpi,\Iq}\matp{\phi(z)}{\psi(z)}}^\ad\rk*{\frac{-\Jimq}{2\im z}}\rk*{\diag\rk*{\rk{z-\ug}\su{0}\su{0}^\mpi,\Iq}\matp{\phi(z)}{\psi(z)}}\\
 &=\matp{\rk{z-\ug}\su{0}\su{0}^\mpi\phi(z)}{\psi(z)}^\ad\rk*{\frac{-\Jimq}{2\im z}}\matp{\rk{z-\ug}\su{0}\su{0}^\mpi\phi(z)}{\psi(z)}\\
 &=\matp{\rk{z-\ug}\phi(z)}{\psi(z)}^\ad\rk*{\frac{-\Jimq}{2\im z}}\matp{\rk{z-\ug}\phi(z)}{\psi(z)}.
\end{split}\eeq
 Because of \eqref{P1218.33} and \eqref{P1218.34}, we have
\beql{P1218.35}
 \matp{X(z)}{Y(z)}^\ad\rk*{\frac{-\Jimq}{2\im z}}\matp{X(z)}{Y(z)}
 =\matp{\rk{z-\ug}\phi(z)}{\psi(z)}^\ad\rk*{\frac{-\Jimq}{2\im z}}\matp{\rk{z-\ug}\phi(z)}{\psi(z)}+\ek*{\psi(z)}^\ad\su{0}\ek*{\psi(z)}.
\eeq
 Since \(\su{0}\) is \tnnH{}, we know that \(\ek{\psi(z)}^\ad\su{0}\ek{\psi(z)}\in\Cggq\).
 Consequently, \eqref{P1218.35} and \eqref{KD2} imply
\beql{P1218.37}
 \matp{X(z)}{Y(z)}^\ad\rk*{\frac{-\Jimq}{2\im z}}\matp{X(z)}{Y(z)}
 \in\Cggq.
\eeq
 Using \eqref{P1218.V1}, \eqref{P1218.29}, \eqref{P1218.31}, and \eqref{P1218.2}, we obtain
\beql{P1218.38}\begin{split}
 &\matp{\rk{z-\ug}X(z)}{Y(z)}^\ad\rk*{\frac{-\Jimq}{2\im z}}\matp{\rk{z-\ug}X(z)}{Y(z)}\\
 &=\rk*{\diag\rk*{\rk{z-\ug}\Iq,\Iq}\matp{X(z)}{Y(z)}}^\ad\rk*{\frac{-\Jimq}{2\im z}}\rk*{\diag\rk*{\rk{z-\ug}\Iq,\Iq}\matp{X(z)}{Y(z)}}\\
 &=\rk*{\diag\rk*{\rk{z-\ug}\Iq,\Iq}\mHTiua{\su{0}}{z}\matp{\phi(z)}{\psi(z)}}^\ad\rk*{\frac{-\Jimq}{2\im z}}\rk*{\diag\rk*{\rk{z-\ug}\Iq,\Iq}\mHTiua{\su{0}}{z}\matp{\phi(z)}{\psi(z)}}\\
 &=\matp{\phi(z)}{\psi(z)}^\ad\ek*{\diag\rk*{\rk{z-\ug}\Iq,\Iq}\mHTiua{\su{0}}{z}}^\ad\rk*{\frac{-\Jimq}{2\im z}}\ek*{\diag\rk*{\rk{z-\ug}\Iq,\Iq}\mHTiua{\su{0}}{z}}\matp{\phi(z)}{\psi(z)}\\
 &=\matp{\phi(z)}{\psi(z)}^\ad\rk*{\frac{\abs{z-\ug}^2}{2\im z}\ek*{\diag\rk{\su{0},\su{0}^\mpi}}^\ad\rk{-\Jimq}\ek*{\diag\rk{\su{0},\su{0}^\mpi}}}\matp{\phi(z)}{\psi(z)}\\
 &=\matp{\phi(z)}{\psi(z)}^\ad\rk*{\frac{\abs{z-\ug}^2}{2\im z}\ek*{\diag\rk{\su{0}\su{0}^\mpi,\Iq}}^\ad\rk{-\Jimq}\ek*{\diag\rk{\su{0}\su{0}^\mpi,\Iq}}}\matp{\phi(z)}{\psi(z)}\\
 &=\abs{z-\ug}^2\rk*{\diag\rk{\su{0}\su{0}^\mpi,\Iq}\matp{\phi(z)}{\psi(z)}}^\ad\rk*{\frac{-\Jimq}{2\im z}}\rk*{\diag\rk{\su{0}\su{0}^\mpi,\Iq}\matp{\phi(z)}{\psi(z)}}\\
 &=\abs{z-\ug}^2\matp{\su{0}\su{0}^\mpi\phi(z)}{\psi(z)}^\ad\rk*{\frac{-\Jimq}{2\im z}}\matp{\su{0}\su{0}^\mpi\phi(z)}{\psi(z)}
 =\abs{z-\ug}^2\matp{\phi(z)}{\psi(z)}^\ad\rk*{\frac{-\Jimq}{2\im z}}\matp{\phi(z)}{\psi(z)}.
\end{split}\eeq
 From \eqref{P1218.38} and \eqref{KD2} we conclude now \(\tmatp{\rk{z-\ug}X(z)}{Y(z)}^\ad\rk{\frac{-\Jimq}{2\im z}}\tmatp{\rk{z-\ug}X(z)}{Y(z)}\in\Cggq\).
 In view of \rpartss{P1218.a}{P1218.b} and \rdefn{def-sp}, the proof of~\eqref{P1218.e} is complete.

 \eqref{P1218.f} In view of~\eqref{P1218.e}, the application of \rprop{P1544} yields
\(
 F
 \in\SFqa
\).
 In order to complete the proof of~\eqref{P1218.f} we are going to apply \rlem{P1208}.
 From \(F\in\SFqa\) and \rprop{P1216} we get
\(
 \rstr_\ohe F
 \in\RFq
\)
 Now let
\(
 z
 \in\C\setminus\rk{\R\cup\cD}
\).
 Then we have \(\im z\neq0\) and \(\frac{1}{\im z}\im F(z) =\frac{F(z)-F^\ad(z)}{z-\ko z}\).
 Consequently, \rrem{L1409} and \eqref{P1218.37} provide us
\beql{P1218.44}\begin{split}
 &\frac{F(z)-F^\ad(z)}{z-\ko z}
 =\frac{1}{\im z}\im F(z)
 =\frac{1}{\im z}\im\rk*{\ek*{X(z)}\ek*{Y(z)}^\inv}\\
 &=\frac{1}{2\im z}\ek*{Y(z)}^\invad\matp{X(z)}{Y(z)}^\ad\rk{-\Jimq}\matp{X(z)}{Y(z)}\ek*{Y(z)}^\inv\\
 &=\ek*{Y(z)}^\invad\rk*{\matp{\rk{z-\ug}\phi(z)}{\psi(z)}^\ad\rk*{\frac{-\Jimq}{2\im z}}\matp{\rk{z-\ug}\phi(z)}{\psi(z)}+\ek*{\psi(z)}^\ad\su{0}\ek*{\psi(z)}}\ek*{Y(z)}^\inv\\
 &=\ek*{Y(z)}^\invad\matp{\rk{z-\ug}\phi(z)}{\psi(z)}^\ad\rk*{\frac{-\Jimq}{2\im z}}\matp{\rk{z-\ug}\phi(z)}{\psi(z)}\ek*{Y(z)}^\inv\\
 &\qquad+\ek*{Y(z)}^\invad\ek*{\psi(z)}^\ad\su{0}\ek*{\psi(z)}\ek*{Y(z)}^\inv.
\end{split}\eeq
 In view of \(\su{0}^\ad=\su{0}\) and the first equation in \eqref{P1218.4}, we obtain
\beql{P1218.45}\begin{split}
 &\ek*{Y(z)}^\invad\ek*{\psi(z)}^\ad\su{0}\ek*{\psi(z)}\ek*{Y(z)}^\inv
 =\ek*{Y(z)}^\invad\ek*{\psi(z)}^\ad\su{0}\su{0}^\mpi\su{0}\ek*{\psi(z)}\ek*{Y(z)}^\inv\\
 &=\ek*{Y(z)}^\invad\ek*{\psi(z)}^\ad\su{0}^\ad\su{0}^\mpi\su{0}\ek*{\psi(z)}\ek*{Y(z)}^\inv
 =\ek*{Y(z)}^\invad\ek*{\su{0}\psi(z)}^\ad\su{0}^\mpi\ek*{\su{0}\psi(z)}\ek*{Y(z)}^\inv\\
 &=\ek*{Y(z)}^\invad\ek*{-X(z)}^\ad\su{0}^\mpi\ek*{-X(z)}\ek*{Y(z)}^\inv
 =\ek*{Y(z)}^\invad\ek*{X(z)}^\ad\su{0}^\mpi\ek*{X(z)}\ek*{Y(z)}^\inv\\
 &=\ek*{F(z)}^\ad\su{0}^\mpi\ek*{F(z)}.
\end{split}\eeq
 Combining \eqref{P1218.44} and \eqref{P1218.45}, it follows
\begin{multline}\label{P1218.46}
 \frac{F(z)-F^\ad(z)}{z-\ko z}-\ek*{F(z)}^\ad\su{0}^\mpi\ek*{F(z)}\\
 =\ek*{Y(z)}^\invad\matp{\rk{z-\ug}\phi(z)}{\psi(z)}^\ad\rk*{\frac{-\Jimq}{2\im z}}\matp{\rk{z-\ug}\phi(z)}{\psi(z)}\ek*{Y(z)}^\inv. 
\end{multline}
 Applying \eqref{KD2} and \eqref{P1218.46}, we conclude
\beql{P1218.47}
 \frac{F(z)-F^\ad(z)}{z-\ko z}-\ek*{F(z)}^\ad\su{0}^\mpi\ek*{F(z)}
 \in\Cggq.
\eeq
 Furthermore, using again the first equation in \eqref{P1218.4}, we get
\[
 \Ran{F(z)}
 =\Ran{\ek*{X(z)}\ek*{Y(z)}^\inv}
 \subseteq\Ran{X(z)}
 =\Ran{-\su{0}\psi(z)}
 \subseteq\ran{\su{0}}
\]
 and, therefore,
\(
 \Ran{F(z)}
 \subseteq\ran{\su{0}}
\).
 Consequently, taking additionally into account \(\su{0}\in\Cggq\) and \eqref{P1218.47}, we conclude with the aid of \rlem{L.AEP} then
\beql{P1218.49}
 \bMat\su{0}&F(z)\\F^\ad(z)&\frac{F(z)-F^\ad(z)}{z-\ko z}\eMat
 \in\Cggo{2q}.
\eeq
 Since \eqref{P1218.49} is fulfilled for all \(z\in\C\setminus\rk{\rhl\cup\cD}\) and since \(F\) is holomorphic and in particular continuous in \(\Cs\), we obtain that \eqref{P1218.49} holds true for all \(z\in\C\setminus\R\).
 Thus, since \(F\) is holomorphic in \(\ohe\), the application of \rlem{P1208} yields that
\beql{P1218.50}
 \rstr_\ohe F\in\RFOq
\eeq
 and that the \tsm{} \(\mu\) of \(\rstr_\ohe F\) satisfies
\beql{P1218.51}
 \su{0}-\mu\rk{\R}
 \in\Cggq.
\eeq
 From \(F\in\SFqa\) and \eqref{P1218.50} we get \(F\in\SOFqalpha\).
 Moreover, in view of \rprop{P1216} the \taSm{} \(\OSm{F}\) of \(F\) satisfies \(\OSma{F}{\rhl}=\mu\rk{\R}\).
 Thus, \eqref{P1218.51} implies
\(
 \su{0}-\OSmA{F}{\rhl}
 \in\Cggq
\).
 Consequently, \eqref{P1218.50} shows that \(F\in\SFqaskg{0}\).
\end{proof}

\section{A first description of the set \(\SFqaskg{m}\)}\label{S1302}
 In this section, we give a Schur type algorithm for functions which belong to the class \(\SFuqa{0}\). 
 This enables us to construct an explicit bijective mapping between \(\langle \qSpa{ \su{0}^{[m;\ug]} } \rangle\) and \({\SFqaskg{m}}\).

\begin{prop}\label{141A_435}
 Let \(\ug \in \R\), let \(m \in \NO \), let \(\seqs{m} \in \Kggeq{m}\), let \(\seq{\su{j}^\sta{m}}{j}{0}{0}\) be the \tsaalphaa{m}{\seqs{m}}, let \(\copa{\phi}{\psi} \in \qSpa{\su{0}^\sta{m}}\), and let \(\cD\) be a discrete subset of \(\Cs \) such that conditions~\ref{def-sp.i} and~\ref{def-sp.ii} in \rdefn{def-sp} are fulfilled.
 Let \(\rmiupjkou{s}{m}\) be given for every choice of \(j,k\in\set{1,2}\) via \eqref{Def_fVam} and \eqref{BD_fVam}.
 Then 
\begin{align}\label{141A_435_A1}
\det
\ek*{ 
\rmiupswoua{s}{m}{z} \phi(z) + 
\rmiupseoua{s}{m}{z} \psi(z) } 
&\neq 0&\text{for each }
z&\in \C \setminus \rk*{ \rhl \cup \cD} 
\end{align}
and
\begin{align}\label{141A_435_A2}
\ek*{ 
\rmiupnwou{s}{m} \phi + 
\rmiupneou{s}{m} \psi 
}
\ek*{ 
\rmiupswou{s}{m} \phi + 
\rmiupseou{s}{m} \psi 
}^\inv
\in \SFqaSkg{m}.
\end{align}
\end{prop}
\begin{proof}
 We will prove \rprop{141A_435} by induction. 
 
 First we consider the case \(m=0\). 
 Because of \rrem{R1502} and~\zitaa{MR3014201}{\clem{2.9}}, we have \(\su{0} \in \Cqqg\). 
 From \rdefn{D1632} we get \(\su{0}^\sta{0} = \su{0}\) and, hence, \(\copa{\phi}{\psi} \in\qSpa{\su{0}}\).
 In view of \eqref{Def_fVam}, we obtain \(\rmiupou{s}{0} = \mHTiu{ \su{0}}\), and, consequently,  
\[
\mHTiu{ \su{0}}
\matp{\phi}{\psi}
= 
\rmiupou{s}{0}
\matp{\phi}{\psi} 
=
\begin{bmatrix}
\rmiupnwou{s}{0} \phi  
+\rmiupneou{s}{0} \psi
 \\
\rmiupswou{s}{0} \phi 
+\rmiupseou{s}{0} \psi 
\end{bmatrix}.
\]
 Thus, \rpropp{P1218}{P1218.a} shows that \eqref{141A_435_A1} holds true for \(m=0\). 
 \rpropp{P1218}{P1218.e} yields \eqref{141A_435_A2} for \(m=0\). 
 Thus, \rprop{141A_435} is proved for \(m=0\).
 
 Now we assume that there is an \(n \in \N\) such that \rprop{141A_435} is checked for each \(m \in \mn{0}{n-1}\). 
 We consider the case \(m=n\).
 Let \(t_j \defeq  s_j^\sta{1}\) for each \(j \in \mn{0}{n-1}\). 
 From \rthmp{P1546}{P1546.b} we see then that \((t_j)_{j=0}^{n-1} \in \Kggeq{n-1}\).
 Because of \rrem{R1510}, we also  have \(t_0^\sta{n-1} = \su{0}^\sta{n}\).
 Thus, \(\copa{\phi}{\psi}\in\qSpa{t_0^\sta{n-1}}\). 
 Since \rprop{141A_435} is assumed to be true for \(m = n-1\), we get 
\begin{align}\label{141A_435_3}
\det
\ek*{ 
\rmiupswoua{t}{n-1}{z} \phi(z) + 
\rmiupseoua{t}{n-1}{z} \psi(z) } 
&\neq 0
&\text{for each }
z &\in \C \setminus \rk*{ \rhl \cup \cD} 
\end{align}
 and that \(G\defeq  \lftpooua{q}{q}{\rmiupou{t}{n-1}}{ \copa{\phi}{\psi} }\) is a well-defined matrix-valued function for which
\begin{multline}\label{141A_435_4}
G=
\ek*{ 
\rmiupnwou{t}{n-1} \phi + 
\rmiupneou{t}{n-1} \psi 
}
\ek*{ 
\rmiupswou{t}{n-1} \phi + 
\rmiupseou{t}{n-1} \psi 
}^\inv\\
\in \SFuqAAkg{n-1}{\seq{t_j}{j}{0}{n-1}}
\end{multline}
 is true.
 Because of \rdefn{D1059}, we have \(\nul{\su{0}}\subseteq \nul{\su{0}^\sta{1}}\) and, in view of \eqref{F4*1}, then \(\SFdqaa{\su{0}^\sta{1}} \subseteq \SFdqaa{\su{0}}\). 
 Combining this with \rrem{R0544}, we obtain  
\[%
 \SFuqAAkg{n-1}{\seq{t_j}{j}{0}{n-1}}
 \subseteq \SFdqaa{t_0}
 = \SFdqaa{\su{0}^\sta{1}}
 \subseteq \SFdqaa{\su{0}}.
\]
 Hence, \eqref{141A_435_4} implies \(G \in \SFdqaa{\su{0}}\).
 Because of \rprop{L8*12}, we conclude then
\beql{141A_435_7}
 \det\ek*{\rk{z-\ug} \su{0}^\mpi  G(z) + \rk{z-\ug} \Iq }
 \neq 0
\eeq
 and 
\begin{align}\label{141A_435_8}
 \STiaoa{G}{\su{0}}{z}&=
 \lftrfuAa{q}{q}{\mHTiu{ \su{0}}}{G}{z}&\text{for each }z &\in \Cs. 
\end{align}
 By virtue of \rrem{Z3N}, we get \(\rank \mat{ \rk{z-\ug}\su{0}^\mpi , \rk{z-\ug}\Iq } = q\) for each \(z \in \Cs\). 
 \rlem{Bi13_B4} and \eqref{141A_435_3} yield  \(\rank \mat{ \rmiupswoua{t}{n-1}{z}, \rmiupseoua{t}{n-1}{z} } = q\) for each \(z \in \C \setminus \rk{ \rhl \cup \cD }\). 
 Hence, \eqref{141A_435_7}, \eqref{141A_435_3}, and  \rprop{Bi13_B8} show that the inequality \(\det \ek{ \rmiupswoua{s}{n}{z} \phi(z) +\rmiupseoua{s}{n}{z} \psi(z) } \neq 0\) and the equations
\beql{141A_435_9}
 \lftpfuAa{q}{q}{\mHTiu{ \su{0}} \rmiupou{t}{n-1}}{  \copa{\phi}{\psi} } {z}
 =\lftpoouA{q}{q}{\mHTiu{ \su{0}}(z) \rmiupou{t}{n-1}(z)}{\copA{\phi(z)}{\psi(z)} }
 = \lftrfuAa{q}{q}{\mHTiu{ \su{0}}}{G}{z}
\eeq
 are valid for each \(z \in \C \setminus \rk{ \rhl \cup \cD }\). 
 Because of \(\seqs{n} \in \Kggeq{n}\) and \rprop{P1442}, we have \(\seqs{n}\in\Kggq{n} \cap \Dqqu{n}\).
 Hence, from \eqref{141A_435_4} and \rthmp{T0904}{T0904.c} we obtain \(\STiao{G}{\su{0}}\in \SFqaskg{n}\).
 Combining this with \rrem{R1646}, \eqref{141A_435_9}, and \eqref{141A_435_8}, we get
\begin{multline*}
 \ek*{ \rmiupnwou{s}{n}\phi + \rmiupneou{s}{n}\psi } \ek*{ \rmiupswou{s}{n}\phi + \rmiupseou{s}{n}\psi }^\inv
 =\lftpoouA{q}{q}{\rmiupou{s}{n}}{  \copa{\phi}{\psi} } \\
 =\lftpoouA{q}{q}{\mHTiu{ \su{0}} \rmiupou{t}{n-1}}{  \copa{\phi}{\psi} }
 =\lftrfua{q}{q}{\mHTiu{ \su{0}}}{G} 
 =\STiao{G}{\su{0}}
 \in \SFqaSkg{n}.
\end{multline*}
 Consequently, \rprop{141A_435} is proved for \(m=n\) as well.
 The proof is complete.
\end{proof}

\begin{prop}\label{141_439B}
 Let \(\ug \in \R\), let \(m \in \NO \), let \(\seqs{m} \in \Kggeq{m}\), let \(\seq{\su{j}^\sta{m}}{j}{0}{0}\) be the \tsaalphaa{m}{\seqs{m}}, and let \(F \in \SFqaskg{m}\). 
 Then there exists a \tqS{} \(\copa{\phi}{\psi} \in \qSpa{ \su{0}^\sta{m} }\) such that \(\phi\) and \(\psi\) are holomorphic in \(\Cs\) and that, for each \(z \in \Cs\), the inequality \(\det \ek{ \rmiupswoua{s}{m}{z} \phi(z) + \rmiupseoua{s}{m}{z} \psi(z) } \neq 0\) and the representation 
\begin{multline*}
 F(z)
 = \ek*{ \rmiupnwoua{s}{m}{z} \phi(z) + \rmiupneoua{s}{m}{z} \psi(z) } \\
  \times\ek*{ \rmiupswoua{s}{m}{z} \phi(z) + \rmiupseoua{s}{m}{z} \psi(z) }^\inv
\end{multline*}
 of \(F\) hold true.
\end{prop}
\begin{proof}
 (I) Since \(\seqs{m}\) belongs to \(\Kggeq{m}\), from \rrem{R1502} and \rlemp{R1738}{R1738.b} we obtain \(\su{0} \in \Cqqg\). 
 Obviously,
\begin{align} \label{141_439B_0}
 \rank \mat*{ \rk{z-\ug}\su{0}^\mpi , \rk{z-\ug}\Iq  }& = q
 &\text{for all }z &\in \C \setminus \set{\alpha}. 
\end{align}
 \rrem{Z3N} and \rlem{Bi13_B4} show that 
\begin{align} \label{141_439B_0_A}
 \rank \mat*{ -\rk{z-\ug}\su{0}^\mpi , \Iq  - \su{0}^\mpi \su{0} } &= q
 &\text{for all }z &\in \C \setminus \set{\alpha}.
\end{align}

 (II) Now our proof works inductively.
 In the case \(m=0\), the assertion follows immediately applying \eqref{Def_fVam}, \eqref{VWaA}, \rdefn{D1632}, and \rprop{T2110_5}. 

 (III) Because of part~(II) of the proof, we can assume that there is a positive integer \(n\) such that \rprop{141_439B} is already proved for each \(m \in \mn{0}{n-1}\). 

 (IV) We consider now the case \(m=n\).
 Because of \(F \in \SFqaskg{n}\), the \taSm{} \(\sigma_F\) of \(F\) belongs to \(\MggqKskg{n}\).
 Hence, \(s_j = \suo{j}{\sigmau{F}}\) for each \(j \in \mn{0}{n-1}\). 
 In particular, \(\su{0} = \suo{0}{\sigmau{F}} = \OSma{F}{\rhl}\).
 Therefore, and in view of \rlem{L3*22}, for each \( z \in \Cs \), we get \(\ran{\su{0}} = \ran{\OSma{F}{\rhl}}= \ran{F(z)}\).
 Thus, \rrem{L1631} shows
\begin{align} \label{141_439B_H0}
\su{0}\su{0}^\mpi F(z) &= F(z) &\text{for each }z& \in \Cs .
\end{align}
 \rlemp{MR3611471_D2}{MR3611471_D2.b} yields  \(F(z) \in \dblftr{-\rk{z-\ug} \su{0}^\mpi }{ \Iq  - \su{0}^\mpi \su{0}}\) for each \(z \in \C \setminus \set{\alpha}\). 
 Consequently, 
\begin{align}\label{141_439B_H0_B}
\det \ek*{ -\rk{z-\ug} \su{0}^\mpi  F(z) + (\Iq  - \su{0}^\mpi \su{0}) }&
\neq 0
&\text{for all }z& \in \Cs .
\end{align}
 Moreover, from \rprop{L8*9} we get
\begin{align}\label{141_439B_H0_C}
\det 
\rk*{ 
\rk{z-\ug} \su{0}^\mpi
\ek*{\lftrfua{q}{q}{\mHTu{ \su{0}}}{F} } (z)
+ 
\rk{z-\ug}\Iq }
&\neq 0
&\text{for all }z& \in \Cs .
\end{align}
 Combining \eqref{VWaA}, \eqref{VWaA}, \eqref{141_439B_0_A}, \eqref{141_439B_0}, \eqref{141_439B_H0_B}, \eqref{141_439B_H0_C}, \rcor{DFK92_163}, \rrem{MR3611471_D1}, and \eqref{141_439B_H0}, we infer
\begin{multline}\label{141_439B_H2}
 \lftroouA{q}{q}{\mHTiu{ \su{0}(z)}}{\lftroouA{q}{q}{\mHTu{ \su{0}}(z)}{F(z)}}
 =\lftroouA{q}{q}{\mHTiu{ \su{0}}(z) \cdot \mHTu{ \su{0}}(z)}{ F(z) }
 =\lftroouA{q}{q}{\rk{z-\ug} \cdot \diag\rk{\su{0}\su{0}^\mpi , \Iq}}{ F(z) }\\
 =\ek*{ \rk{z-\ug} \su{0}\su{0}^\mpi F(z)} \ek*{ \rk{z-\ug} \Iq  }^\inv
 =\su{0}\su{0}^\mpi F(z) = F(z)
\end{multline}
 for each \(z \in \Cs \).  
 Let \(t_j \defeq  s_j^\sta{1}\) for each \(j \in \mn{0}{n-1}\).
 From \rthmp{P1546}{P1546.b} we see then  that \((t_j)_{j=0}^{n-1} \in \Kggeq{n-1}\). 
 Let \(R\) be the \taSSto{\su{0}}{F}.
 By virtue of \eqref{F8*1}, we obtain then \(R = F^{[+,\ug,\su{0}]}\). 
 Consequently, \rthm{T149_5} yields \(R \in \SFuqa{0} \ek{ (t_j)_{j=0}^{n-1}, \leq }\).
 Because of part~(III) of the proof, then there exists a pair \(\copa{\phi}{\psi} \in \qSpa{t_0}\) such that 
\beql{141_439B_H1_1}
\det \ek*{ 
\rmiupswoua{t}{n-1}{z} \phi(z) + 
\rmiupseoua{t}{n-1}{z} \psi(z) 
} \neq 0
\eeq
and 
\begin{multline*}
R(z)
= 
\ek*{ 
\rmiupnwoua{t}{n-1}{z} \phi(z) + 
\rmiupneoua{t}{n-1}{z} \psi(z)
} \\
\times
\ek*{ 
\rmiupswoua{t}{n-1}{z} \phi(z) + 
\rmiupseoua{t}{n-1}{z} \psi(z)
}^\inv
\end{multline*}
 hold true for each \(z \in \Cs\).
 Hence, from \eqref{BD_fVam} and \rnota{LinFracTrans} we get
\beql{141_439B_H3}
R(z)
= 
\ek*{\lftpfuA{q}{q}{\rmiupou{t}{n-1}}{ \copa{\phi}{\psi} }
} (z)
\eeq
 for each \(z \in \Cs\).
 Regarding \eqref{141_439B_H1_1}, \rlem{Bi13_B4} shows that
\beql{141_439B_H6}
\rank\mat*{\rmiupswoua{s}{n-1}{z}, \rmiupseoua{s}{n-1}{z}}
= q
\eeq
 for each \(z \in \Cs\). 
 In view of \eqref{141_439B_H0_C}, equation~\eqref{141_439B_H3} and \rprop{L8*9} yield \(F(z) \in \dblftr{-\rk{z-\ug}\su{0}^\mpi }{\Iq  -\su{0}^\mpi \su{0}}\) and
\beql{141_439B_H4}\begin{split}
 \lftpoouA{q}{q}{\rmiupou{t}{n-1}(z)}{ \copA{\phi(z)}{\psi(z)} }
& =\ek*{\lftpfuA{q}{q}{\rmiupou{t}{n-1}}{ \copa{\phi}{\psi} }} (z)
=R(z)
=F^{[+,\ug,\su{0}]}(z)  \\
&=\ek*{\lftrfua{q}{q}{\mHTu{ \su{0}}}{ F }}(z)
=\lftroouA{q}{q}{\mHTu{ \su{0}}(z)}{ F(z) }
\end{split}\eeq
 for each \(z \in \Cs\). 
 Hence, \eqref{141_439B_H4} and \eqref{141_439B_H0_C} provide us
\beql{141_439B_H5}
\det \rk*{\rk{z-\ug} \su{0}^\mpi\ek*{ \lftpoouA{q}{q}{\rmiupou{t}{n-1}}{ \copa{\phi}{\psi} } } (z)  + \rk{z-\ug} \Iq }
\neq 0
\eeq
 for each \(z \in \Cs\).
 Taking into account \eqref{VWaA}, \eqref{141_439B_H0_B}, \eqref{141_439B_H0_C}, \eqref{141_439B_H2}, \eqref{141_439B_H4}, \eqref{BD_fVam}, \eqref{141_439B_H6}, again \eqref{VWaA}, \eqref{141_439B_0}, \eqref{141_439B_H1_1}, \eqref{141_439B_H5},   \rprop{Bi13_B8}, \(t_j \defeq  s_j^\sta{1}\) for all \(j \in \mn{0}{n-1}\), \rrem{R1646}, again \eqref{BD_fVam}, and \rnota{LinFracTrans}, for each \(z \in \Cs\), we infer
\[\begin{split}
&F(z) 
=\lftroouA{q}{q}{\mHTiu{ \su{0}}(z)}{\lftroouA{q}{q}{\mHTu{ \su{0}}(z)}{F(z)}}
=\lftroouA{q}{q}{\mHTiu{ \su{0}}(z)}{\lftpoouA{q}{q}{\rmiupou{t}{n-1}(z)}{ \copA{\phi(z)}{\psi(z)} }} \\
&=\lftpoouA{q}{q}{\mHTiu{ \su{0}}(z) \cdot \rmiupou{t}{n-1}(z)}{ \copA{\phi(z)}{\psi(z)} }
=\ek*{\lftpoou{q}{q}{\rmiupou{s}{n}}\rk*{ \copa{\phi}{\psi} }} (z) \\
&=\ek*{
\rmiupnwou{s}{n}(z) \phi(z) + 
\rmiupneou{s}{n}(z) \psi(z)
}
\ek*{ 
\rmiupswou{s}{n}(z) \phi(z) + 
\rmiupseou{s}{n}(z) \psi(z)
}^\inv.
\end{split}\]
 By induction, the proof is complete. 
\end{proof}

 If \(M \in \Cpq\), if \(\copa{\phi_1}{\psi_1}\in \qSpa{M}\), and if \(\copa{\phi_2}{\psi_2} \in\qSp \) fulfills \(\copacl{\phi_1}{\psi_1} =\copacl{\phi_2}{\psi_2}\), then it is easily checked that \(\copa{\phi_2}{\psi_2}\in \qSpa{M}\).

 We will now collect the results of this section. 
 In this way, we obtain a first parametrization of \rprob{\mproblem{\rhl}{m}{\leq}} formulated in \rsect{S1417}. 

\bthml{T0945}
 Let \(\ug \in \R\), let \(m \in \NO \), and let \(\seqs{m} \in \Kggeq{m}\). 
 Let \(\seq{\su{j}^\sta{m}}{j}{0}{0}\) be the \tsaalphaa{m}{\seqs{m}}.
 Let \(\rmiupou{s}{m}\) be defined via \eqref{Def_fVam} and \eqref{VWaA}.
 Furthermore, let \eqref{BD_fVam} be the \tqqa{block} representation of \(\rmiupou{s}{m}\). 
 Then the following statements hold true: 
\begin{enui}
 \il{T0945.a} For each \(\copa{\phi}{\psi}\in \qSpa{ \su{0}^\sta{m}}\), the function  \(\det \rk{  \rmiupswou{s}{m} \phi + \rmiupseou{s}{m} \psi}\) is meromorphic in \(\Cs\) and does not vanish identically.
 Furthermore, 
\beql{N48}
 \rk{ \rmiupnwou{s}{m} \phi + \rmiupneou{s}{m} \psi} \rk{ \rmiupswou{s}{m} \phi + \rmiupseou{s}{m} \psi}^\inv  
\eeq
 belongs to \(\SFqaskg{m}\).
 \il{T0945.b} For every choice of \(F \in \SFqaskg{m}\), there exists a pair  \(\copa{\phi}{\psi}\in \qSpa{ \su{0}^\sta{m}}\) of \tqqa{matrix-valued} functions \(\phi\) and \(\psi\) which are holomorphic in \(\Cs \) such that, for each \(z \in \Cs\), the inequality \({\det \ek{  \rmiupswoua{s}{m}{z} \phi(z) + \rmiupseoua{s}{m}{z} \psi(z)} \neq 0}\) and the representation
\begin{multline}\label{N48NN}
 F(z) 
 =\ek*{ \rmiupnwoua{s}{m}{z} \phi(z) + \rmiupneoua{s}{m}{z} \psi(z)} \\
 \times\ek*{ \rmiupswoua{s}{m}{z} \phi(z) + \rmiupseoua{s}{m}{z} \psi(z)}^\inv
\end{multline} 
 of \(F\) hold true. 
 \il{T0945.c} Let \(\copa{\phi_1}{\psi_1},\copa{\phi_2}{\psi_2}\in \qSpa{ \su{0}^\sta{m}}\).
 Then the following statements are equivalent: 
\begin{aeqii}{0}
 \il{T0945.i}
 $\rk{ \rmiupnwou{s}{m} \phi_1 + \rmiupneou{s}{m} \psi_1} \rk{ \rmiupswou{s}{m} \phi_1 + \rmiupseou{s}{m} \psi_1}^\inv \\
 = \rk{ \rmiupnwou{s}{m} \phi_2 + \rmiupneou{s}{m} \psi_2} \rk{ \rmiupswou{s}{m} \phi_2 + \rmiupseou{s}{m} \psi_2}^\inv$.
 \il{T0945.ii} \(\copacl{\phi_1}{\psi_1}=\copacl{\phi_2}{\psi_2}\).
\end{aeqii}
\end{enui}
\ethm
\begin{proof}
 \eqref{T0945.a} Combine \rdefn{def-sp} and \rprop{141A_435}. 

 \eqref{T0945.b} This follows from \rprop{141_439B}.

 \eqref{T0945.c}
 \begin{imp}{T0945.i}{T0945.ii}
  Let
  \beql{T0945.1}
   F
   \defeq\rk{\rmiupnwou{s}{m}\phi_1+\rmiupneou{s}{m}\psi_1}\rk{\rmiupswou{s}{m}\phi_1+\rmiupseou{s}{m}\psi_1}^\inv.
  \eeq
  Then
  \beql{T0945.2}\begin{split}
   \matp{F}{\IqCs}
   &=\matp{\rmiupnwou{s}{m}\phi_1+\rmiupneou{s}{m}\psi_1}{\rmiupswou{s}{m}\phi_1+\rmiupseou{s}{m}\psi_1}\rk{\rmiupswou{s}{m}\phi_1+\rmiupseou{s}{m}\psi_1}^\inv\\
   &=\rmiupou{s}{m}\matp{\phi_1}{\psi_1}\rk{\rmiupswou{s}{m}\phi_1+\rmiupseou{s}{m}\psi_1}^\inv.
  \end{split}\eeq
  Let \(g_\ug\colon\Cs\to\C\) be defined via \(g_\ug(z)\defeq\rk{z-\ug}^{m+1}\).
  From \rlem{141_438} we obtain then
  \beql{T0945.4}
   \rmupou{s}{m}\rmiupou{s}{m}
   =g_\ug\cdot\diag\rk*{\su{0}^\sta{m}\rk{\su{0}^\sta{m}}^\mpi,\Iq}.
  \eeq
  Because of \(\copa{\phi_1}{\psi_1}\in\qSpa{\su{0}^\sta{m}}\) and \rremp{L1631}{L1631.b}, we have
  \(
   \su{0}^\sta{m}\rk{\su{0}^\sta{m}}^\mpi\phi_1
   =\phi_1
  \).
  Consequently, combining \eqref{T0945.2} and \eqref{T0945.4} gives us
  \beql{T0945.6}\begin{split}
   &\rmupou{s}{m}\rmiupou{s}{m}\matp{F}{\IqCs}\\
   &=\rmupou{s}{m}\rmiupou{s}{m}\matp{\phi_1}{\psi_1}\rk{\rmiupswou{s}{m}\phi_1+\rmiupseou{s}{m}\psi_1}^\inv\\
   &=g_\ug\diag\rk*{\su{0}^\sta{m}\rk{\su{0}^\sta{m}}^\mpi,\Iq}\matp{\phi_1}{\psi_1}\rk{\rmiupswou{s}{m}\phi_1+\rmiupseou{s}{m}\psi_1}^\inv\\
   &=g_\ug\matp{\su{0}^\sta{m}\rk{\su{0}^\sta{m}}^\mpi\phi_1}{\psi_1}\rk{\rmiupswou{s}{m}\phi_1+\rmiupseou{s}{m}\psi_1}^\inv\\
   &=g_\ug\matp{\phi_1}{\psi_1}\rk{\rmiupswou{s}{m}\phi_1+\rmiupseou{s}{m}\psi_1}^\inv.
  \end{split}\eeq
 From \eqref{T0945.1} and~\ref{T0945.i} we get
 \beql{T0945.7}
  F
  =\rk{\rmiupnwou{s}{m}\phi_2+\rmiupneou{s}{m}\psi_2}\rk{\rmiupswou{s}{m}\phi_2+\rmiupseou{s}{m}\psi_2}^\inv.
 \eeq
 Then, because of \eqref{T0945.7} and \(\copa{\phi_2}{\psi_2}\in\qSpa{\su{0}^\sta{m}}\), the above computations (see \eqref{T0945.1} and \eqref{T0945.6}) show that
 \beql{T0945.8}
  \rmupou{s}{m}\rmiupou{s}{m}\matp{F}{\IqCs}
  =g_\ug\matp{\phi_2}{\psi_2}\rk{\rmiupswou{s}{m}\phi_2+\rmiupseou{s}{m}\psi_2}^\inv.
 \eeq
 From \eqref{T0945.6}, \eqref{T0945.8}, and the definition of the function \(g_\ug\) we get
\[
 \matp{\phi_1}{\psi_1}\rk{\rmiupswou{s}{m}\phi_1+\rmiupseou{s}{m}\psi_1}^\inv
 =\matp{\phi_2}{\psi_2}\rk{\rmiupswou{s}{m}\phi_2+\rmiupseou{s}{m}\psi_2}^\inv
\]
 and, consequently,
\beql{T0945.9}
 \matp{\phi_2}{\psi_2}
 =\matp{\phi_1}{\psi_1}\rk{\rmiupswou{s}{m}\phi_1+\rmiupseou{s}{m}\psi_1}^\inv\rk{\rmiupswou{s}{m}\phi_2+\rmiupseou{s}{m}\psi_2}.
\eeq
 In view of
\[
 \rk{\rmiupswou{s}{m}\phi_1+\rmiupseou{s}{m}\psi_1}^\inv\rk{\rmiupswou{s}{m}\phi_2+\rmiupseou{s}{m}\psi_2}
 \in\ek*{\mero{\Cs}}^\x{q}
\]
 and the fact that the determinant of this function does not identically vanish in \(\Cs\), we see from \eqref{T0945.9} that \(\copacl{\phi_1}{\psi_1}=\copacl{\phi_2}{\psi_2}\).
 Thus,~\ref{T0945.ii} is satisfied.
 \end{imp}
  
 \begin{imp}{T0945.ii}{T0945.i}
  This implication holds trivially.
 \end{imp}
\end{proof}

 Let \(\ug \in \R\), let \(m \in \NO \), and let \(\seqs{m}\) be a sequence of complex \tqqa{matrices}.
 Denote by \(\seq{\su{j}^\sta{m}}{j}{0}{0}\) the \tsaalphaa{m}{\seqs{m}}.
 According to \rlem{L1252}, we write \symba{\rcset{\qSpa{ \su{0}^\sta{m} }}}{p} for the set of the equivalence classes the representatives of which belong to \(\qSpa{ \su{0}^\sta{m} }\).

\bcorol{C1507}
 Let \(\ug \in \R\), let \(m \in \NO \), and let \(\seqs{m} \in \Kggeq{m}\).
 Let \(\seq{\su{j}^\saalpha{m}}{j}{0}{0}\) be the \tsaalphaa{m}{\seqs{m}}.
 Then the mapping \(\Sigma\colon\rcset{\qSpa{ \su{0}^\sta{m}}}\to \SFqaskg{m}\) defined by 
\beql{Sigma}
 \Sigma\rk*{\copacl{\phi}{\psi}}
 \defeq\rk{\rmiupnwou{s}{m}\phi+\rmiupneou{s}{m}\psi}\rk{\rmiupswou{s}{m}\phi+\rmiupseou{s}{m}\psi}^\inv
\eeq
 is well defined and bijective.
\ecoro
\bproof
 Apply \rthm{T0945}.
\eproof

\section{Parametrization of the matricial Stieltjes moment problem in the general case}\label{S1314}
 \rthm{T0945} gives a parametrization of the moment problem in question with a parameter set which depends on the given data indicated by the matrix \(\su{0}^\sta{m}\).
 In this section, we prove a parametrization with a parameter set which is independent of the given data.
 We distinguish the following three cases:
\bAeqi{0}
 \il{S1314.I} The non-degenerate case \(\rank\su{0}^\sta{m}=q\).
 In view of~\zitaa{MR3611479}{\cprop{9.19}} this is equivalent to \(\seqs{m}\in\Kgq{m}\).
 \il{S1314.II} The completely degenerate case \(\rank\su{0}^\sta{m}=0\).
 In view of~\zitaa{MR3611479}{\cprop{9.22}} this is equivalent to \(\seqs{m}\in\Kggdq{m}\).
 \il{S1314.III} The degenerate, but not completely degenerate case \(1\leq\rank\su{0}^\sta{m}\leq q-1\).
\eAeqi

\subsection{The non-degenerate case}
 First we want to turn our attention to the particular cases in \rthm{T0945}, where the rank of the matrix \(\su{0}^\sta{m}\) attains the extremal values \(q\) and \(0\), respectively. We start with the full rank case.
  
\begin{thm}\label{T1206}
 Let \(\ug \in \R\), let \(m \in \NO \), and let \(\seqs{m} \in \Kgq{m}\).
 Let \(\seq{\su{j}^\sta{m}}{j}{0}{0}\) be the \tsaalphaa{m}{\seqs{m}}.
 Then:
\benui
 \il{T1206.a} \(\seqs{m}\in\Kggeq{m}\) and \(\rank \su{0}^\sta{m} = q\). 
 \il{T1206.b} Let \(\rmiupou{s}{m}\) be defined via \eqref{Def_fVam} and \eqref{VWaA}.
 Furthermore, let \eqref{BD_fVam} be the \tqqa{block} representation of \(\rmiupou{s}{m}\). 
 Then: 
\benuii
 \il{T1206.b1} For each \(\copa{\phi}{\psi}\in \qSp \), the function  \(\det \rk{  \rmiupswou{s}{m} \phi + \rmiupseou{s}{m} \psi}\) is meromorphic in \(\Cs\) and does not vanish identically.
 Furthermore, 
\[
 \rk{ \rmiupnwou{s}{m} \phi + \rmiupneou{s}{m} \psi} \rk{ \rmiupswou{s}{m} \phi + \rmiupseou{s}{m} \psi}^\inv
\]
 belongs to \(\SFqaskg{m}\)
 \il{T1206.b2} For every choice of \(F \in \SFqaskg{m}\), there exists a pair  \(\copa{\phi}{\psi} \in\qSp \) of \tqqa{matrix-valued} functions \(\phi\) and \(\psi\) which are holomorphic in \(\Cs \) such that, for each \(z \in \Cs\), the inequality \(\det \ek{  \rmiupswoua{s}{m}{z} \phi(z) + \rmiupseoua{s}{m}{z} \psi(z)} \neq 0\) and the representation \eqref{N48NN}
 of \(F\) hold true. 
 \il{T1206.b3} Let \(\copa{\phi_1}{\psi_1},\copa{\phi_2}{\psi_2}\in\qSp \).
 Then conditions~\ref{T0945.i} and~\ref{T0945.ii} stated in \rthmp{T0945}{T0945.c} are equivalent.
\eenuii
\eenui
\end{thm}
\begin{proof}
 \zitaa{MR3014201}{\cprop{2.20}} implies that \(\seqs{m}\in\Kggeq{m}\), whereas~\zitaa{MR3611479}{\cprop{9.19}} yields that \(\det\su{0}^\sta{m}\neq0\).
 Thus, \rpart{T1206.a} is proved.
 Because of~\eqref{T1206.a}, \rthm{T0945}, and \rrem{R1244}, \rpart{T1206.b} follows.
\end{proof}

 It should be mentioned that in the situation of \rthm{T1206} an alternate approach to the determination of the set \(\SFqaskg{m}\) was presented in the recent PhD~thesis~\zitas{Jes17,arXiv:1703.06759} of B.~Jeschke.
 His approach was inspired by the technique used by Yu.~M.~Dyukarev~\zita{MR2053150} in the case \(\ug=0\).
 Moreover, B.~Jeschke extended results of A.~E.~Choque~Rivero~\zita{MR3324594} on various kinds of matrix polynomials which are connected to the case \(\ug=0\) to the case of arbitrary \(\ug\in\R\).

\subsection{The completely degenerate case}
 Now we see in particular that in the completely degenerate case the moment problem~\mproblem{\iraa{\alpha}}{m}{\leq} has a unique solution. 

\begin{thm}\label{T1227}
 Let \(\ug \in \R\), let \(m \in \NO \), and let \(\seqs{m} \in \Kggdq{m}\).
 Let \(\seq{\su{j}^\sta{m}}{j}{0}{0}\) be the \tsaalphaa{m}{\seqs{m}}.
 Then:
\benui
 \il{T1227.a} The relations \(\seqs{m}\in\Kggeq{m}\) and \(\su{0}^\sta{m}=\Oqq\) hold true.
 \il{T1227.b} Let \(\rmiupou{s}{m}\) be defined via \eqref{Def_fVam} and \eqref{VWaA}.
 Furthermore, let \eqref{BD_fVam} be the \tqqa{block} representation of \(\rmiupou{s}{m}\). 
 Then
 \[
  \SFqaSkg{m}
  =\set*{\rmiupneou{s}{m} \rk{\rmiupseou{s}{m}}^\inv}.
 \]
 \il{T1227.c} \(\SFqas{m}=\SFqaskg{m}\).
\eenui
\end{thm}
\begin{proof}
 \eqref{T1227.a} From \rprop{P1628} we get \(\seqs{m}\in\Kggeq{m}\), whereas~\zitaa{MR3611479}{\cprop{9.20}} yields \(\su{0}^\sta{m}=\Oqq\).
 
 \eqref{T1227.b} Combine~\eqref{T1227.a}, \rthm{T0945}, and \rrem{L1149}.

 \eqref{T1227.c} In view of \rpart{T1227.b} and \rrem{R1445}, \rpart{T1227.c} follows from \rthmp{T1*3+2}{T1*3+2.a}.
\end{proof}

 Observe that the unique solution of \rprob{\mproblem{\rhl}{m}{\leq}} which occurs in the situation of \rthm{T1227} is a \tnnH{} measure concentrated on a finite set of points (see~\zita{MR2735313}).

\bcorl{C52}
 Let \(\ug\in\R\), let \(m\in\NO\), and let \(\seqs{m}\in\Kggdq{m}\).
 Then \(\MggqKsg{m}=\MggqKskg{m}\).
\ecor
\bproof
 Use \rthmss{T3*2}{T1227}.
\eproof

\subsection{The degenerate, but not completely degenerate case}%
 In this section, we prove a parametrization of the set of \taSt{s} of the solutions of the Stieltjes-type power moment problem \mproblem{\iraa{\alpha}}{m}{\leq} with free parameters in the case \(1\leq\rank\su{0}^\sta{m}\leq q-1\).

\begin{lem}\label{L1049} 
 Let \(M \in \Cqp\) be such that \(r \defeq  \rank M\) fulfills \(r \geq 1\). 
 Let \(u_1, u_2,\dotsc, u_r\) be an orthonormal basis of \(\ran{M}\), let \(U \defeq  \mat{u_1, u_2,\dotsc, u_r}\), and let \(Q\defeq\OPu{\nul{M^\ad}}\).
 Then the mapping \(\gamma\colon\aSp{r} \to \qSpa{M}\) given by  \(\gamma \rk{\copa{\phi}{\psi} }\defeq\copa{U \phi U^\ad}{U \psi U^\ad  +   Q  }\) is well defined and injective. 
\end{lem}
\begin{proof} 
 In view of \rremss{R0857}{G312N}, we have obviously
\begin{align}
  Q  ^\ad &=  Q  , & &&   Q  ^2&=   Q  \label{L1049.1},\\ 
U^\ad U &=\Iu{r}, & &\text{and}&  UU^\ad  = \OPu{\ran{M}} &= \OPu{\nul{M^\ad}^\bot } = \Iq  -   Q  .  
\label{L1049.2}
\end{align}
 From \eqref{L1049.2} and \eqref{L1049.1} we obtain
\beql{L1049.3}
U^\ad    Q  
= U^\ad UU^\ad   Q  
= U^\ad  \rk{ \Iq  -   Q  }   Q   
= \NM_{r \times q}.
\eeq
 Let \(\copa{\phi}{\psi}\in \aSp{r}\). 
 By \rdefn{def-sp}, there is a discrete subset \(\cD\) of \(\Cs\) such that the conditions~\ref{def-sp.i}--\ref{def-sp.iii}) of \rdefn{def-sp} are fulfilled. 
 Because of \rdefnp{def-sp}{def-sp.i}, the matrix-valued functions \(F\defeq  U \phi U^\ad \) and \(G\defeq  U\psi U^\ad  +   Q  \) are meromorphic and holomorphic in \(\C \setminus \rk{ \rhl \cup \cD }\). 
 Using \eqref{L1049.2}, we get then 
\begin{align} 
\label{L1049.4}
G^\ad F &= \rk{ \Iq  -UU^\ad  + U\psi^\ad U^\ad }\rk{ U\phi U^\ad  } = U \psi^\ad  \phi U^\ad , 
\\ 
\label{L1049.5}
F^\ad F &= U\phi^\ad \phi U^\ad ,
\end{align}  
 and, in view of \eqref{L1049.3} and  \eqref{L1049.1}, furthermore
\beql{L1049.6}
 G^\ad G 
 =U \psi^\ad U^\ad U\psi U^\ad  +   Q  ^\ad U\psi U^\ad  + U\psi^\ad  U^\ad   Q   +   Q  ^\ad   Q   \\
 =U\psi^\ad \psi U^\ad  +   Q  .
\eeq
 By virtue of \eqref{L1049.5} and \eqref{L1049.6}, we have 
\beql{L1049.7}
 F^\ad F + G^\ad G
 = U \rk{ \phi^\ad  \phi + \psi^\ad  \psi } U^\ad  +   Q  .
\eeq
 We consider now an arbitrary \(z \in \C \setminus \rk{ \rhl \cup \cD }\) and an arbitrary \(v \in \nul{ \tmatp{F(z)}{G(z)} }\). 
 Then \eqref{L1049.7} implies
\beql{L1049.8}\begin{split}
 0 
 &= v^\ad  \rk*{\ek{F(z)}^\ad\ek{F(z)} +\ek{G(z)}^\ad\ek{G(z)}} v \\
 &= \rk{ U^\ad v }^\ad  \rk*{\ek{\phi(z)}^\ad\ek{\phi(z)}+\ek{\psi(z)}^\ad\ek{\psi(z)}} \rk{ U^\ad v} + v^\ad     Q   v.   
\end{split}\eeq
 Taking into account \rdefnpp{def-sp}{def-sp.i}{def-sp.ii}, we get \(\rank \rk{\ek{\phi(z)}^\ad\ek{\phi(z)}+ \ek{\psi(z)}^\ad\ek{\psi(z)}}=\rank\tmatp{\phi(z) }{ \psi(z) } = r\) and, consequently, 
\beql{L1049.9}
 \ek{\phi(z)}^\ad\ek{\phi(z)}+ \ek{\psi(z)}^\ad\ek{\psi(z)}
 \in \Cgo{r}.
\eeq
 Obviously, \(  Q   =   Q     Q  ^\ad  \in \Cqqg\). 
 Thus, from \eqref{L1049.8} and \eqref{L1049.9},  we conclude \(U^\ad v = \Ouu{r}{1}\) and \(v^\ad   Q  v = 0\). 
 Since \eqref{L1049.2} shows then that 
\[
0 
= v^\ad UU^\ad v 
= v^\ad  \rk{ \Iq  -   Q   } v 
= v^\ad v - v^\ad   Q  v 
= v^\ad v 
= \normEs{v}
\]
 holds true, we infer \(v = \Ouu{q}{1}\). 
 Hence, \(\nul{ \tmatp{F(z)}{G(z)} }= \set{\Ouu{q}{1}}\) for each \(z \in \C \setminus \rk{ \rhl \cup \cD }\). 
 For every choice of \(\eta \in\set{1, z -\ug}\) and \(z\) in \(\C \setminus (\rhl \cup \cD)\), from \eqref{L1049.4} we obtain 
\[
 \frac{1}{\im z} \im \ek*{ \eta G^\ad (z) F(z) }
 = U \rk*{ 
\frac{1}{\im z} \im \ek*{ \eta \psi^\ad (z) \phi(z) }
}
U^\ad 
\]
 and, in view of \rrem{R1404} and \rdefnp{def-sp}{def-sp.iii}, consequently, 
\[\begin{split}
\begin{bmatrix} \eta F(z) \\ G(z) \end{bmatrix}^\ad 
\rk*{ \frac{-\Jimq}{2 \im z} }
\begin{bmatrix} \eta F(z) \\ G(z) \end{bmatrix} 
&= 
\frac{1}{\im z} \im \ek*{ \eta G^\ad (z) F(z) }
= 
U 
\rk*{ 
\frac{1}{\im z} \im \ek*{ \eta \psi^\ad (z) \phi(z) }
}
U^\ad 
 \\
&=
U
\begin{bmatrix} \eta \phi(z) \\ \psi(z) \end{bmatrix}^\ad 
\rk*{ \frac{-\Jimq}{2 \im z} }
\begin{bmatrix} \eta \phi(z) \\ \psi(z) \end{bmatrix} 
U^\ad  
\in \Cqqg.
\end{split}\]
 Thus, \(\copa{F}{G}\) belongs to \(\qSp \). 
 Since \(\ran{F} = \ran{U\phi U^\ad} \subseteq \ran{U} = \ran{M}\) and \rrem{L1631} imply \(MM^\mpi F=F\), we see that \(F\) belongs to \(\qSpa{M}\). 
 Consequently, the mapping \(\gamma\) is well defined. 
 Now we check that \(\gamma\) is injective. 
 For this reason, we consider arbitrary \(\copa{\phi_1}{\psi_1},\copa{\phi_2}{\psi_2}\in \aSp{r}\) such that \(\gamma \rk{ \copa{\phi_1}{\psi_1} }= \gamma \rk{ \copa{\phi_2}{\psi_2} }\). 
 Then \(U \phi_1 U^\ad  = U \phi_2 U^\ad \) and \(U \psi_1 U^\ad  +   Q   = U \psi_2 U^\ad  +   Q  \). 
 In view of \eqref{L1049.2}, this implies \(\phi_1 = \phi_2\) and \(\psi_1 = \psi_2\). 
 Thus, \(\gamma\) is injective.
\end{proof}

\bremnl{\zitaa{Mak14}{\clem{11.3}}}{R1230}
 Let \(\ug\in\R\), let \(W\) be a complex \tqqa{matrix} with \(W^\ad V=\Iq\).
 Let \(\copa{\phi_1}{\psi_1}\in\qSp\).
 Then \(\phi\) and \(\psi\) are functions meromorphic in \(\Cs\) and there exists a discrete subset \(\cD\) of \(\Cs \) such that the conditions~\ref{def-sp.i},~\ref{def-sp.ii}, and~\ref{def-sp.iii} of \rdefn{def-sp} are fulfilled.
 Hence, \(\phi_2\defeq V\phi_1\) and \(\psi_2\defeq W\psi_1\) are functions meromorphic in \(\Cs\) which are holomorphic in \(\C\setminus\rk{\rhl\cup\cD}\) such that \(\tmatp{\phi_2}{\psi_2}=\diag\rk{V,W}\cdot\tmatp{\phi_1}{\psi_1}\).
 Thus, \(\rank\tmatp{\phi_2(z)}{\psi_2(z)}=\rank\tmatp{\phi_1(z)}{\psi_1(z)} =q\) for all \(z\in\C\setminus\rk{\rhl \cup \cD}\) and
\begin{multline*}
  \matp{\rk{z-\ug}^\ell\phi_2(z)}{\psi_2(z)}^\ad\rk*{\frac{-\Jimq}{2 \im z}}\matp{\rk{z-\ug}^\ell\phi_2(z)}{\psi_2(z)}\\
  =\matp{\rk{z-\ug}^\ell\phi_1(z)}{\psi_1(z)}^\ad\rk*{\frac{-\Jimq}{2 \im z}}\matp{\rk{z-\ug}^\ell\phi_1(z)}{\psi_1(z)}
  \in\Cggq
\end{multline*}
 for all \(z\in\C\setminus\rk{\R\cup\cD}\) and all \(\ell\in\set{1,2}\).
 Consequently, \(\copa{\phi_2}{\psi_2}\in\qSp\).
\erem

\begin{lem}\label{L1410}
 Let \(M \in \Cqp\) be such that \(r \defeq  \rank M\) fulfills \(r \geq 1\). 
 Let \(u_1, u_2,\dotsc, u_r\) be an orthonormal basis of \(\ran{M}\), let \(U \defeq  \mat{ u_1, u_2,\dotsc, u_r }\), and let \(Q\defeq\OPu{\nul{M^\ad}}\).
 Furthermore, let \(\copa{F}{G}\in \qSpa{M}\). 
 Then: 
\begin{enui}
 \il{L1410.a} The matrix-valued function \(B \defeq  G -\iu F\) is meromorphic and the function \(\det B\) does not vanish identically.
 \il{L1410.b} Let \(\phi \defeq  U^\ad  FB^\inv U\) and let \(\psi \defeq  U^\ad  GB^\inv U\). 
 Then \(\copa{\phi}{\psi}\in \aSp{r}\) and 
\beql{L1410.B1}
\matp{F}{G} B^\inv 
= 
\begin{bmatrix}
U \phi U^\ad  \\
U \psi U^\ad  +   Q  
\end{bmatrix}.
\eeq
\end{enui}   
\end{lem}
\begin{proof}
 Obviously, \eqref{L1049.2} holds true. 
 In view of \rremss{G312N}{G318N}, we have then 
$UU^\ad  
= \Iq  -   Q   
= \Iq  - \OPu{\ran{M}^\bot } 
= MM^\mpi $. 
 Thus, because of \(\copa{F}{G}\in \qSpa{M}\), we get 
\beql{L1410.1b}
 UU^\ad F 
 = MM^\mpi F 
 = F.
\eeq
 For parts of the following, we adopt the method used in the proof of~\zitaa{Thi06}{\clem{1.6(a)}}. Since \(\copa{F}{G}\) belongs to \(\qSpa{M}\), we see that \(F\) and \(G\) are in \(\Cs \) meromorphic \(\Cqq\)\nobreakdash-valued functions and that there exists a discrete subset \(\mathcal{D}\) of \(\Cs\) such that conditions~\ref{def-sp.i}--\ref{def-sp.iii} of \rdefn{def-sp} hold true.
 Hence, \(B\) is meromorphic in \(\Cs\) fulfilling
\begin{multline*}
 \ek*{B(z)}^\ad\ek*{B(z)}
 =\rk*{\ek*{G(z)}^\ad+\iu\ek*{F(z)}^\ad}\ek*{G(z)-\iu F(z)}\\
 =\ek*{G(z)}^\ad\ek*{G(z)}+\ek*{F(z)}^\ad\ek*{F(z)}+2\im\rk*{\ek*{G(z)}^\ad\ek*{F(z)}}
 \geq\matp{F(z)}{G(z)}^\ad\matp{F(z)}{G(z)}
 \in\Cgq
\end{multline*}
 for all \(z\in\ohe\setminus\mathcal{D}\).
 In particular, the function \(\det B\) does not vanish in \(\ohe\setminus\mathcal{D}\).
 Thus, \(S \defeq  (G +iF)B^\inv\) is an in \(\Cs \) meromorphic \(\Cqq\)\nobreakdash-valued function which fulfills 
\begin{align}\label{L1410.2}
FB^\inv&= \frac{\iu}{2} \rk*{ \Iq  - S }&
&\text{and}&
GB^\inv&= \frac{1}{2} \rk*{ \Iq  + S }.
\end{align}
 Moreover, direct calculation shows
\[
 \Iq-\ek*{S(z)}^\ad\ek*{S(z)}
 =4\ek*{B(z)}^\invad\im\rk*{\ek*{G(z)}^\ad\ek*{F(z)}}\ek*{B(z)}^\inv
 \in\Cggq
\]
 for all \(z\in\ohe\setminus\mathcal{D}\). Using Riemann's theorem on removable singularities, we can conclude that \(\tilde{S} \defeq  \Rstr_{\ohe }S\) belongs to the Schur class \(\qqSp \).

 First we consider the case \(1 \leq r\leq q-1\). 
We choose then 
\(u_{r+1}, u_{r+2},\dotsc, u_q\in \Cq\) 
such that 
\(u_1, u_2,\dotsc, u_q\) is an orthonormal basis of \(\Cq\). 
Let 
\(V \defeq  \mat{ u_{r+1}, u_{r+2},\dotsc, u_q }\) 
and let \(W \defeq  \mat{U,V}\). 
Then 
\begin{align}\label{L1410.13}
V^\ad U = \NM_{(q-r) \times r}, 
\qquad 
V^\ad V = \Iu{q-r},
\end{align} 
and \(W^\ad W = \Iq \) hold true. 
This implies 
\(U^\ad V = \NM_{r \times (q-r)}\), \(WW^\ad  = \Iq \), and, because of \eqref{L1049.2}, 
hence 
\begin{align} \label{L1410.3}
\Iq  
= \mat{ U, V } \mat{ U, V } ^\ad  
= UU^\ad  + VV^\ad  
= \Iq  -   Q   + VV^\ad .
\end{align}
This shows that 
\begin{align} \label{L1410.4}
  Q   = VV^\ad .
\end{align}
From \eqref{L1410.1b} and \eqref{L1410.13} we see that 
\begin{align}\label{L1410.5} 
V^\ad F = V^\ad  UU^\ad F = \NM_{(q-r) \times q}.
\end{align}
Using \eqref{L1410.2} and \eqref{L1410.5}, we get 
\begin{align*}
V^\ad  - V^\ad S 
= V^\ad  (\Iq  - S) 
= -2i V^\ad FB^\inv 
= \NM_{(q-r) \times q}
\end{align*}
and, consequently, \(V^\ad S = V^\ad \). 
Thus, in view of \eqref{L1410.13}, we obtain 
\(V^\ad SU = V^\ad U = \NM_{(q-r) \times r}\) 
and 
\(V^\ad SV = V^\ad V =\Iu{q-r}\). 
Hence, 
\begin{align}\label{L1410.6}
W^\ad SW 
= 
\begin{bmatrix}
U^\ad SU & U^\ad SV \\
V^\ad SU & V^\ad SV
\end{bmatrix} 
= 
\begin{bmatrix}
U^\ad SU & U^\ad SV \\
\NM_{(q-r)\times r} & \Iu{q-r}
\end{bmatrix}.
\end{align}
Since \(S(z)\) is contractive for each \(z \in \ohe \), 
the matrix 
\(W^\ad S(z)W\) is contractive for each \(z \in \ohe \). 
Thus, \eqref{L1410.6} and 
\rrem{KM} 
yield that 
\(U^\ad S(z)V = \NM_{r \times (q-r)}\) for all \(z \in \ohe \). 
Taking into account that \(S\) is meromorphic in \(\Cs\), we obtain \(U^\ad S V = \NM_{r \times (q-r)}\). 
Because of \eqref{L1410.6}, this implies 
\(W^\ad SW = \diag \rk*{ U^\ad  S U, \Iu{q-r} }\). 
By virtue of \(WW^\ad  = \Iq \), then 
\(S = W \cdot \diag \rk*{ U^\ad SU, \Iu{q-r} } \cdot W^\ad \) 
follows. 
Furthermore, since \eqref{L1049.2} holds true, 
for all \(\zeta \in \C\), we have then 
\beql{L1410.7}\begin{split}
\Iq  + \zeta S 
&=W \ek*{ \Iq  + \diag \rk*{ U^\ad  (\zeta S) U, \zeta \Iu{q-r} } } W^\ad 
 \\
&=
W \cdot \diag \rk*{ U^\ad  (\Iq  + \zeta S) U, (1 + \zeta) \Iu{q-r} } \cdot W^\ad .
\end{split}\eeq
 According to \eqref{L1410.2} and\eqref{L1410.7}, we conclude 
\beql{L1410.8}\begin{split} 
 FB^\inv 
 &= \frac{\iu}{2} W \cdot \diag \rk*{ U^\ad  (\Iq  - S) U,  \Ouu{(q-r)}{(q-r)} } \cdot W^\ad  \\
 &=W \cdot \diag \rk{ U^\ad  FB^\inv U,  \Ouu{(q-r)}{(q-r)} } \cdot W^\ad  \\ 
 &= \mat{ U, V } \cdot \diag \rk{ \phi, \NM_{(q-r) \times  (q-r)} } \cdot \begin{bmatrix} U^\ad  \\ V^\ad  \end{bmatrix}
 = U \phi U^\ad 
\end{split}\eeq
and 
\beql{L1410.9}\begin{split} 
 GB^\inv 
 &= \frac{1}{2} W \cdot \diag \rk*{ U^\ad  (\Iq  + S) U, 2 \Iu{q-r} } \cdot W^\ad\\ %
 &=W \cdot \diag \rk{ U^\ad  GB^\inv U,  \Iu{q-r} } \cdot W^\ad
 = \mat{ U, V } \cdot \diag \rk{ \psi, \Iu{q-r} } \cdot \begin{bmatrix} U^\ad  \\ V^\ad  \end{bmatrix}%
 = U \psi U^\ad  + VV^\ad .
\end{split}\eeq
 By virtue of \eqref{L1410.9} and \eqref{L1410.4}, 
we get 
\(GB^\inv = U \psi U^\ad  +   Q  \). 
Combining this with \eqref{L1410.8}, 
we infer \eqref{L1410.B1}. 
Furthermore, \eqref{L1410.8} and \eqref{L1410.9} yield
\begin{align}\label{L1410.10}
\matp{F}{G} B^\inv W 
= 
\begin{bmatrix}
U \phi U^\ad  W \\
U \psi U^\ad  W + VV^\ad W
\end{bmatrix}.
\end{align}
 Using \(U^\ad W = U^\ad  \mat{ U,V } = \mat{ U^\ad U, U^\ad V } = \mat{ \Iu{r}, \NM_{r \times (q-r)} }\) and \(V^\ad W = V^\ad  \mat{ U,V } = \mat{ V^\ad U, V^\ad V } = \mat{ \NM_{(q-r) \times r}, \Iu{q-r} }\), then 
\begin{align}\label{L1410.11}
\begin{bmatrix}
FB^\inv W \\ GB^\inv W 
\end{bmatrix}
=
\begin{bmatrix}
\mat{ U  \phi, \NM_{q \times (q-r)} }
\\
\mat{ U  \psi, V }
\end{bmatrix}
\end{align}
follows. 
 Since \(\copa{F}{G}\) belongs to \(\qSp \) and since \(\det \rk{ B^\inv W }\) does not vanish identically, \rrem{SP.1} yields that \(\copa{F B^\inv W  }{ G  B^\inv W }\) belongs to \(\qSp \). 
 Consequently, \eqref{L1410.11} provides \(\smatp{\mat{ U  \phi, \NM_{q \times (q-r)} }}{\mat{ U  \psi, V }}\in \qSp \). 
 Thus, in view of \((W^\ad )^\ad W^\ad  = \Iq \) and \rrem{R1230}, we have \(\smatp{W^\ad  \mat{ U  \phi, \NM_{q \times (q-r)} }}{W^\ad \mat{ U  \psi, V }}\in\qSp \). 
 Because of \(U^\ad U = \Iu{r}\) and \(V^\ad U = \NM_{(q-r)\times r}\), 
we have
\begin{align*}
W^\ad  \mat{ U  \phi, \NM_{q \times (q-r)} }
= 
\begin{bmatrix}
U^\ad U \phi & \NM_{r \times (q-r)} 
\\
V^\ad U \phi & \Ouu{(q-r)}{(q-r)} 
\end{bmatrix}
=
\diag \rk{ \phi, \Ouu{(q-r)}{(q-r)} }
\end{align*}
and
\begin{align*}
W^\ad  \mat{ U  \psi, V }
=
\begin{bmatrix}
U^\ad U \psi & U^\ad V 
\\
V^\ad U \psi & V^\ad V
\end{bmatrix}
=
\diag \rk{ \psi, \Iu{q-r} }.
\end{align*}
 Hence, \(\smatp{\diag \rk{ \phi, \Ouu{(q-r)}{(q-r)} } }{\diag \rk{ \psi, \Iu{q-r} }}\in \qSp \). 
 \rdefn{def-sp} shows then that there is a discrete subset \(\cD\) of \(\Cs\) such that \(\phi\) and \(\psi\) are holomorphic in \(\Cs\), that \(\rank \tmatp{\phi(z) }{ \psi(z) } = q - \rank \Iu{q-r} = r\) for all \(z \in \C \setminus (\rhl \cup \cD)\), and, in view of \rrem{R55}, that
\begin{multline*}
 \diag
\rk*{
\begin{bmatrix}
\rk{z-\ug}^k \phi(z) \\ \psi(z)
\end{bmatrix}^\ad 
\rk*{
\frac{-\Jimq}{2  \im z}
}
\begin{bmatrix}
\rk{z-\ug}^k \phi(z) \\ \psi(z)
\end{bmatrix}, 
\Ouu{(q-r)}{(q-r)}
}\\
=
\begin{bmatrix}
\rk{z-\ug}^k \cdot \diag \rk{ \phi(z), \Ouu{(q-r)}{(q-r)} } 
\\
\diag \rk{ \psi(z), \Iu{q-r} }
\end{bmatrix}^\ad 
\rk*{
\frac{-\Jimq}{2  \im z}
}
\begin{bmatrix}
\rk{z-\ug}^k \cdot \diag \rk{ \phi(z), \Ouu{(q-r)}{(q-r)} } 
\\
\diag \rk{ \psi(z), \Iu{q-r} }
\end{bmatrix}\\
\in\Cggq
\end{multline*}
 for every choice of \(k \in \Z_{0,1}\) and \(z \in \C \setminus \rk{ \rhl \cup \cD}\). 
 In particular, \eqref{KD1} and \eqref{KD2} hold true for all \(z \in \C \setminus \rk{ \rhl \cup \cD}\). 
 Thus, \(\copa{\phi}{\psi}\in\aSp{r}\). 

 It remains to consider the case \(r = q\). 
 Then \(U\) is unitary. 
 Consequently, \(FB^\inv = U \phi U^\ad \) and \(GB^\inv = U \psi U^\ad \), which shows that 
\beql{L1410.12}
\begin{bmatrix} FB^\inv \\ GB^\inv \end{bmatrix}
= 
\begin{bmatrix}
U \phi U^\ad  \\ U \psi U^\ad 
\end{bmatrix}.
\eeq
 From \eqref{L1049.2} we conclude \(\Iq  = UU^\ad  = \Iq  -   Q  \) and, therefore, \(  Q   = \Oqq\). 
 This implies \(GB^\inv = U \psi U^\ad  +   Q  \). 
 Hence, \eqref{L1410.B1} holds true. 
 Since \(\copa{F}{G}\) belongs to \(\qSp \) from \eqref{L1410.12}, \rpart{L1410.a}, and \rrem{SP.1} we see that \(\copa{U \phi U^\ad  }{U \psi U^\ad}\) belongs to \(\qSp \) as well. 
 Since \(U\) is unitary, from \rdefn{def-sp} we get then that there is a discrete subset \(\cD\) of \(\Cs\) such that \(\phi\) and \(\psi\) are holomorphic in \(\C \setminus (\rhl \cup \cD)\), that \(\rank \tmatp{\phi(z) }{ \psi(z) } = q\) holds true for all \(z \in \C \setminus (\rhl \cup \cD)\), and, in view of \rrem{R1404} and \(U^\ad U=\Iq \), that 
\[\begin{split}
 &U
\begin{bmatrix}
\rk{z-\ug}^k \phi(z) \\ \psi(z)
\end{bmatrix}^\ad 
\rk*{
\frac{-\Jimq}{2  \im z}
}
\begin{bmatrix}
\rk{z-\ug}^k \phi(z) \\ \psi(z)
\end{bmatrix}
U^\ad\\
 &=U
\ek*{
\frac{1}{\im z} \im \rk*{ \ek*{ \psi(z) }^\ad  \ek*{ \rk{z-\ug}^k \phi(z) } }
} 
U^\ad
 =\frac{1}{\im z}\im \rk*{ \rk{z-\ug}^k U \ek*{ \psi(z) }^\ad \phi(z)U^\ad  }\\
 &=
 \begin{bmatrix}
\rk{z-\ug}^k U\phi(z)U^\ad  \\ U\psi(z) U^\ad 
\end{bmatrix}^\ad 
\rk*{
\frac{-\Jimq}{2  \im z}
}
\begin{bmatrix}
\rk{z-\ug}^k U\phi(z)U^\ad  \\ U\psi(z) U^\ad 
\end{bmatrix}
 \in\Cggq
\end{split}\]
 is fulfilled for every choice of \(k \in \Z_{0,1}\) and \(z \in \C \setminus (\rhl \cup \cD\). 
 Using \(U^\ad U=\Iq \) again, then, for each \(k \in \Z_{0,1}\) and each \(z \in \C \setminus (\rhl \cup \cD)\), we get \eqref{KD1} and \eqref{KD2}.
 Consequently, \(\copa{\phi}{\psi}\) belongs to \(\qSp \). 
\end{proof}

\begin{lem}\label{L1033}
 Let \(\ug\in\R\) and let \(M \in \Cqp\) be such that \(r \defeq  \rank M\) fulfills \(r \geq 1\). 
 Let \(u_1, u_2,\dotsc, u_r\) be an orthonormal basis of \(\ran{M}\), let \(U \defeq  \mat{ u_1, u_2,\dotsc, u_r }\), and let \(Q\defeq\OPu{\nul{M^\ad}}\). 
 Then \(\Gamma_U\colon\rcset{\aSp{r}}\to\rcset{\qSpa{M}}\) given by 
\beql{SAl_MG}
\Gamma_U 
\rk*{
\copacl{\phi}{ \psi}
}
\defeq\copacl{U \phi U^\ad }{U \psi U^\ad  +   Q  }
\eeq
 is well defined and bijective. 
\end{lem}
\begin{proof}
 Obviously, \eqref{L1049.1}, \eqref{L1049.2}, and \eqref{L1049.3} are valid. 
 Let \(\copa{\phi}{\psi}\in \aSp{r}\). 
 According to \rlem{L1049}, then \(\Gamma_U \rk{ \copacl{\phi}{ \psi} }\) belongs to \(\langle \qSpa{M}\rangle\). 
 We consider now arbitrary \(\copa{\phi_1}{\psi_1},\copa{\phi_2}{\psi_2}\in \aSp{r}\) which fulfill \(\copacl{\phi_1}{\psi_1}=\copacl{\phi_2}{\psi_2}\). 
 By virtue of \rdefn{D0710}, then there are a meromorphic \taaa{r}{r}{matrix-valued} function \(\theta\) and a discrete subset \(\cD\) of \(\Cs \) such that \(\phi_1\), \(\phi_2\), \(\psi_1\), \(\psi_2\), and \(\theta\) are holomorphic in \(\C \setminus (\rhl \cup \cD)\) and that \(\det \theta(z) \neq 0\), \(\phi_2(z) = \phi_1(z) \theta(z)\), and \(\psi_2(z) = \psi_1(z) \theta(z)\) hold true for each \(z \in \C \setminus (\rhl \cup \cD)\). 
 Then \(F_1 \defeq  U \phi_1 U^\ad \), \(F_2 \defeq  U \phi_2 U^\ad \), \(G_1 \defeq  U \psi_1 U^\ad  +   Q  \), \(G_2 \defeq  U \psi_2 U^\ad  +   Q  \), and \(T   \defeq  U \theta U^\ad  +   Q  \) are meromorphic matrix-valued functions which, in view of \eqref{L1049.2}, \eqref{L1049.3}, and \eqref{L1049.1}, fulfill 
\[\begin{split}
 F_1(z)T(z) 
 &=U \phi_1(z)U^\ad U \theta(z)U^\ad  + U \phi_1(z) U^\ad  \rk{ \Iq  - UU^\ad  }  \\
 &= U \phi_1(z) \theta(z) U^\ad  
 = U \phi_2(z) U^\ad  
 = F_2(z)
\end{split}\]
and 
\begin{multline*}
 G_1(z)T(z) 
 = U \psi_1(z)U^\ad U \theta(z)U^\ad  + U \psi_1(z) U^\ad    Q   +   Q   U \theta(z) U^\ad  +   Q  ^2\\
 = U \psi_1(z) \theta(z) U^\ad  +   Q  
 = G_2(z)
\end{multline*}
 for each \(z \in \Cs\). 
 Now let \(z \in \Cs\). 
 We are going to check that \(\det T(z) \neq 0\). 
 For this reason, let \(v \in \nul{ T(z) }\). 
 Then 
\beql{L1033_1}
 \Ouu{q}{1}
 = T(z) v
 = U \theta(z) U^\ad  v +   Q  v.
\eeq
 According to \eqref{L1049.2} and \eqref{L1033_1}, we get 
\(\Ouu{r}{1} = \theta(z) U^\ad  v + U^\ad  (\Iq  - UU^\ad )v = \theta(z)U^\ad v\).
 Because of \(\det \theta(z) \neq 0\), then \(U^\ad v = \Ouu{r}{1}\) follows. 
 Thus, taking into account \eqref{L1033_1} and \eqref{L1049.2} again, this implies 
\(
\Ouu{q}{1}
=   Q  v 
= (\Iq  - UU^\ad )v 
= v - UU^\ad v
= v\).
 Consequently, \(\nul{ T(z) } \subseteq \set{\Ouu{q}{1}}\) and, hence, \(\det T(z) \neq 0\). 
 Therefore, the pairs \(\copa{F_1}{G_1}\) and \(\copa{F_2}{G_2}\) are equivalent. 
 In other words, the mapping \(\Gamma_U\) is well defined.
 
 In order to prove that \(\Gamma_U\) is injective, we consider arbitrary \(\copa{\phi_1}{\psi_1},\copa{\phi_2}{\psi_2}\in \aSp{r}\) with \(\Gamma_U \rk{\copacl{\phi_1}{\psi_1}}=\Gamma_U \rk{\copacl{\phi_2}{\psi_2}}\). 
 Then there are a meromorphic \tqqa{matrix-valued} function \(\tilde\theta\) and a discrete subset \(\tilde \cD\) of \(\Cs\) such that \(\phi_1\), \(\phi_2\), \(\psi_1\), \(\psi_2\), and \(\tilde\theta\) are holomorphic in \(\C \setminus \rk{ \rhl \cup \tilde \cD }\) and that \(\det \tilde\theta (z) \neq 0\), \(U \phi_2(z)U^\ad  = U\phi_1(z)U^\ad  \tilde\theta(z)\), and 
\beql{L1033_2}
 U \psi_2(z)U^\ad  +   Q   
 = \ek*{ U \psi_1(z)U^\ad  +   Q   } \tilde\theta(z).
\eeq
 Thus, setting \(\Lambda \defeq  U^\ad  \tilde \theta U\), from \eqref{L1049.2} and \eqref{L1049.3}, we have 
\[
\phi_1 \Lambda 
= U^\ad U \phi_1 \Lambda 
= U^\ad U \phi_1 U^\ad  \tilde\theta U 
= U^\ad U \phi_2 U^\ad  U 
= \phi_2 
\]
and 
\[\begin{split}
\psi_1 \Lambda 
= U^\ad U \psi_1 \Lambda 
= U^\ad U \psi_1 U^\ad  \tilde\theta U
= U^\ad  \rk{ U \psi_2 U^\ad  +   Q   -   Q   \tilde\theta } U 
= \psi_2 + U^\ad    Q   (\Iq  - \tilde\theta) U 
= \psi_2.
\end{split}\]
 If \(r = q\), then \(UU^\ad  = \Iq \) and, consequently, \(\det \Lambda(z) = \det \tilde\theta(z) \neq 0\) for all \(z \in \C \setminus \rk{ \rhl \cup \tilde \cD }\). 
 Now we assume \(r \leq q - 1\). Then we choose vectors \(u_{r+1}, u_{r+2},\dotsc, u_q\in \Cq\) such that \(u_1, u_2,\dotsc, u_q\) is an orthonormal basis of \(\Cq\). 
 Put \(V \defeq  \mat{ u_{r+1}, u_{r+2},\dotsc, u_q }\). 
 Obviously, \eqref{L1410.13} 
 holds true.
 Multiplying equation \eqref{L1033_2} by \(V^\ad\) from the left-hand side and using \eqref{L1410.13}, we get \(V^\ad    Q   = V^\ad    Q   \tilde\theta\). 
 By virtue of \eqref{L1410.13} and \eqref{L1049.2}, this implies \(V^\ad  = V^\ad  \rk{ \Iq  - UU^\ad  } = V^\ad  \rk{ \Iq  - UU^\ad  } \tilde\theta = V^\ad  \tilde\theta\). 
 Thus, \(V^\ad  \tilde\theta U = V^\ad  U = \NM_{(q-r) \times r}\) and \(V^\ad  \tilde\theta V = V^\ad  V = \Iu{q-r}\). 
 Hence, \(W \defeq  \mat{U,V}\) fulfills \(W^\ad W= \Iq \) and 
\begin{align*}
W^\ad  \tilde\theta W 
= 
\begin{bmatrix}
U^\ad  \tilde\theta U & U^\ad  \tilde\theta V \\
V^\ad  \tilde\theta U & V^\ad  \tilde\theta V 
\end{bmatrix}
=
\begin{bmatrix}
\Lambda & U^\ad  \tilde\theta V \\
\NM_{(q-r)\times r} & \Iu{q-r} 
\end{bmatrix}. 
\end{align*}
 In particular, \(\det \Lambda(z) = \det \ek{W^\ad  \tilde\theta(z) W} = \det \tilde\theta(z) \neq 0\) for all \(z \in \C \setminus \rk{ \rhl \cup \tilde \cD }\). 
 Hence, in the case \(r = q\) as well as in the case \(r \leq q -1\), the function \(\det \Lambda\) does not vanish identically. 
 Consequently, \(\copacl{\phi_1}{\psi_1}=\copacl{\phi_2}{\psi_2}\). 
 Therefore, the mapping \(\Gamma_U\) is injective. 
 It remains to prove that \(\Gamma_U\) is surjective. 
 Let \(\copa{F}{G}\in \qSpa{M}\). 
 Then we know from \rlem{L1410} that \(B \defeq  G -\iu F\) is a \tqqa{matrix-valued} function which is meromorphic in \(\Cs\) for which \(\det B\) does not vanish identically. 
 Furthermore, \rlem{L1410} shows that \(\phi \defeq  U^\ad FB^\inv U\) and \(\psi \defeq  U^\ad GB^\inv U\) are meromorphic matrix-valued functions such that \(\copa{\phi}{\psi}\in \aSp{r}\) and \eqref{L1410.B1} hold true. 
 Consequently, \(\Gamma_U \rk{ \copacl{\phi}{ \psi} }=\copacl{F}{G}\). Hence, \(\Gamma_U\) is surjective as well. 
\end{proof}

\begin{rem} \label{SAl_VII4}
 Let \(M \in \Cqp\) be such that \(r \defeq  \rank M\) fulfills \(1 \leq r \leq q-1\). 
 Let \(u_1, u_2,\dotsc, u_q\) be an orthonormal basis of \(\Cq\) such that \(u_1, u_2,\dotsc, u_r\) is an orthonormal basis of \(\ran{M}\). 
 Let \(W \defeq  \mat{ u_1, u_2,\dotsc,u_q }\), let \(U \defeq  \mat{ u_1, u_2,\dotsc,u_r }\), and let \(V \defeq  \mat{ u_{r+1}, u_{r+2},\dotsc,u_q }\). 
 Then \(W = \mat{U,V}\) and \eqref{L1049.2} hold true. 
 Because \(WW^\ad  = \Iq \), we see that \eqref{L1410.3} and \eqref{L1410.4} are true. 
 Thus, the mapping \(\Gamma_U\colon\rcset{\aSp{r}}\to\rcset{\qSpa{M}}\) defined by \eqref{SAl_MG} fulfills 
\[%
\Gamma_U 
\rk*{ 
\copacl{\phi}{ \psi}
}
=\copacl{W \cdot \diag \rk{ \phi, \Ouu{(q-r)}{(q-r)} } \cdot W^\ad }{W \cdot \diag \rk{ \psi, \Iu{q-r} } \cdot W^\ad }.
\]
\end{rem}

 Now we obtain the announced description of the set \(\SFqaskg{m}\) in the degenerate, but not completely degenerate case as well:

\begin{thm}\label{SAl_PTII}
 Let \(\ug \in \R\), let \(m \in \NO \), and let \(\seqs{m} \in \Kggeq{m}\) be such that \(r \defeq  \rank \su{0}^\sta{m}\) fulfills \(1 \leq r \leq q-1\). 
 Let \(\rmiupou{s}{m}\) be defined via \eqref{Def_fVam} and \rrem{T149_2}.
 Furthermore, let \eqref{BD_fVam} be the \tqqa{block} representation of \(\rmiupou{s}{m}\). 
 Let \(u_1, u_2,\dotsc, u_q\) be an orthonormal basis of \(\Cq\) such that \(u_1, u_2,\dotsc, u_r\) is an orthonormal basis of \(\ran{\su{0}^\sta{m} }\) and let \(W \defeq  \mat{ u_1, u_2,\dotsc,u_q }\). 
 Then:
\begin{enui}
 \il{SAl_PTII.a} For each pair \(\copa{\phi}{\psi}\in \aSp{r}\), the function
\[
 \det \ek*{  \rmiupswou{s}{m} W \cdot \diag \rk{ \phi, \Ouu{(q-r)}{(q-r)} } + \rmiupseou{s}{m} W \cdot \diag \rk{ \psi, \Iu{q-r} } }
\]
 is meromorphic in \(\Cs\) and does not vanish identically.
 Furthermore, 
\begin{multline*}%
 F
 \defeq  \ek*{ 
\rmiupnwou{s}{m} 
W \cdot \diag \rk{ \phi, \Ouu{(q-r)}{(q-r)} } + 
\rmiupneou{s}{m} 
W \cdot \diag \rk{ \psi, \Iu{q-r} } 
} \\
 \times\ek*{ 
\rmiupswou{s}{m} 
W \cdot \diag \rk{ \phi, \Ouu{(q-r)}{(q-r)} } + 
\rmiupseou{s}{m} 
W \cdot \diag \rk{ \psi, \Iu{q-r} }
}^\inv
\end{multline*}
 belongs to \(\SFqaskg{m}\).

 \il{SAl_PTII.b} For each \(F \in \SFqaskg{m}\), there exists a pair  \(\copa{\phi}{\psi} \in\aSp{r}\) of \tqqa{matrix-valued} functions \(\phi\) and \(\psi\) which are holomorphic in \(\Cs \) such that, for each \(z \in \Cs\), the inequality
\[
 \det \ek*{  \rmiupswoua{s}{m}{z} W \cdot \diag \rk{ \phi(z), \Ouu{(q-r)}{(q-r)} } + \rmiupseoua{s}{m}{z} W \cdot \diag \rk{ \psi(z), \Iu{q-r} }}\neq 0
\]
 and the representation
\begin{multline*}%
 F(z) \\
 =\ek{ 
\rmiupnwoua{s}{m}{z} 
W \cdot \diag \rk*{ \phi(z), \Ouu{(q-r)}{(q-r)} } + 
\rmiupneoua{s}{m}{z} 
W \cdot \diag \rk*{ \psi(z), \Iu{q-r} } 
}\\
 \times\ek*{ 
\rmiupswou{s}{m} (z)
W \cdot \diag \rk*{ \phi(z), \Ouu{(q-r)}{(q-r)} } + 
\rmiupseou{s}{m} (z) 
W \cdot \diag \rk*{ \psi(z), \Iu{q-r} }
}^\inv
\end{multline*} 
 hold true. 
 \il{SAl_PTII.c} Let \(\copa{\phi_1}{\psi_1},\copa{\phi_2}{\psi_2}\in \aSp{r}\).
 Then
 \[\begin{split}
&\ek*{ 
\rmiupnwou{s}{m} 
W \cdot \diag \rk{ \phi_1, \Ouu{(q-r)}{(q-r)} } + 
\rmiupneou{s}{m} 
W \cdot \diag \rk{ \psi_1, \Iu{q-r} } 
}\\
&\quad\times\ek*{ 
\rmiupswou{s}{m} 
W \cdot \diag \rk{ \phi_1, \Ouu{(q-r)}{(q-r)} } + 
\rmiupseou{s}{m} 
W \cdot \diag \rk{ \psi_1, \Iu{q-r} }
}^\inv\\
 &=\ek*{ 
\rmiupnwou{s}{m} 
W \cdot \diag \rk{ \phi_2, \Ouu{(q-r)}{(q-r)} } + 
\rmiupneou{s}{m} 
W \cdot \diag \rk{ \psi_2, \Iu{q-r} } 
}\\
&\quad\times\ek*{ 
\rmiupswou{s}{m} 
W \cdot \diag \rk{ \phi_2, \Ouu{(q-r)}{(q-r)} } + 
\rmiupseou{s}{m} 
W \cdot \diag \rk{ \psi_2, \Iu{q-r} }
}^\inv
\end{split}\]
 if and only if \(\copacl{\phi_1}{\psi_1}=\copacl{\phi_2}{\psi_2}\).
\end{enui}
\end{thm}
\begin{proof}
 In view of \(WW^\ad  = \Iq \), the assertion follows immediately by application of \rthm{T0945}, \rlem{L1033}, and \rrem{SAl_VII4}.
\end{proof}

\begin{rem}\label{SAl_SII1}
 Let \(\ug \in \R\), let \(m \in \NO \), and let \(\seqs{m} \in \Kggeq{m}\) be such that \(r \defeq  \rank \su{0}^\sta{m}\) fulfills \(r\geq1\). 
 Let \(u_1, u_2,\dotsc, u_r\) be an orthonormal basis of \(\ran{\su{0}^\sta{m}}\) and let \(U \defeq  \mat{ u_1, u_2,\dotsc, u_r }\).
 Let \(\Gamma_U\colon\rcset{\aSp{r}}\to\langle\qSpa{ \su{0}^\sta{m} }\rangle\) be defined by \eqref{SAl_MG} and let \(\Sigma\colon\langle\qSpa{ \su{0}^\sta{m} }\rangle \to{\SFqaskg{m}}\) be given by \eqref{Sigma}. 
 Then one can see from \rlem{L1033} and \rcor{C1507} that \(\Sigma \circ \Gamma_U\) realizes a bijection between \(\rcset{\aSp{r}}\) and \({\SFqaskg{m}}\).
\end{rem}

\section{A connection between two moment problems}\label{S1144}
 Let \(\ug\in\R\), let \(m\in\NO\), and let \(\seqs{m}\in\Kggeq{m}\).
 In~\zitaa{MR3611471}{\cthm{13.1}} a complete description of the set of solutions of the moment problem~\mproblem{\rhl}{m}{=} was given in terms of the Stieltjes transforms of the solutions.
 Because of \rrem{R1445}, this set is a subset of the set of all Stieltjes transforms of solutions of the moment problem~\mproblem{\rhl}{m}{\leq}.
 The latter set was determined in \rthm{T0945} and \rsec{S1314}.
 Now we will demonstrate where the Stieltjes transforms of the solutions of the moment problem~\mproblem{\rhl}{m}{=} can be found amongst the Stieltjes transforms of the solutions of the moment problem~\mproblem{\rhl}{m}{\leq}.
 In the completely degenerate case, an answer to this question was already given in \rthm{T1227} and \rcor{C52}.
 
 Let \(\ug\in\R\).
 We set\index{P@\(\qSdpp\)}
\beql{qSdpp}
 \qSdpp
 \defeq\setaa*{\copa{\phi}{\psi}\in\qSpp}{\lim_{y\to\infp}\normS*{\rk{\phi\psi^\inv}(\iu y)}=0}.
\eeq
 If \(A\in\Cqq\), then let\index{P@\(\qSdppa{A}\)}
\beql{qSdppa}
 \qSdppa{A}
 \defeq\qSdpp\cap\qSpa{A}.
\eeq
 If \(m\in\NO\) and \(\seqs{m}\in\Kggeq{m}\), then we use the notation
\beql{Sm<=}
 \SFqaS{m}
 \defeq\setaa*{F\in\SFuqa{m}}{\OSm{F}\in\MggqKSg{m}}.
\eeq

\bleml{L1029}
 Let \(\ug\in\R\) and let \(\rho_\ug^\diamond\colon\rcset{\qSdpp}\to\SFdqa\) be defined by 
\beql{rhod}
 \rho^\diamond_\ug\rk*{\copacl{\phi}{\psi}}
 \defeq\phi\psi^\inv.
\eeq
 Then \(\rho^\diamond_\ug\) is well defined and bijective with inverse \(\iota^\diamond_\ug\colon\SFdqa\to\rcset{\qSdpp}\) given by
\beql{iotad}
 \iota^\diamond_\ug\rk{F}
 \defeq\copacl{F}{\IqCs}.
\eeq
\elem
\bproof
 First observe that \(\qSdpp\subseteq\qSpp\).
 In view of \eqref{F3*4}, we have \(\SFdqa\subseteq\SFqa\).
 Because of \rcor{C0726}, it is thus sufficient to show \(\rho_\ug\rk{\rcset{\qSdpp}}\subseteq\SFdqa\) and \(\iota_\ug\rk{\SFdqa}\subseteq\rcset{\qSdpp}\), where \(\rho_\ug\colon\rcset{\qSpp}\to\SFqa\) and \(\iota_\ug\colon\SFqa\to\rcset{\qSpp}\) are given by \eqref{rho} and \eqref{iota}, respectively. 
 Consider an arbitrary pair \(\copa{\phi}{\psi}\in\qSdpp\).
 According to \rcor{C0726}, then \(F\defeq\phi\psi^\inv=\rho_\ug\rk{\copacl{\phi}{\psi}}\) belongs to \(\SFqa\).
 Because of \eqref{qSdpp}, furthermore \(\lim_{y\to\infp}\normS{F(\iu y)}=0\).
 Thus, \(F\in\SFdqa\) by virtue of \eqref{F3*4}.
 Conversely, now we consider a function \(F\in\SFdqa\).
 Using \rcor{C0726}, then we see that \(\copacl{F}{\IqCs}=\iota_\ug\rk{F}\) belongs to \(\rcset{\qSpp}\).
 In particular, \(\copa{F}{\IqCs}\in\qSpp\).
 Furthermore, we have \(\lim_{y\to\infp}\normS{F(\iu y)}=0\) by virtue of \eqref{F3*4}.
 In view of \eqref{qSdpp}, thus \(\copacl{F}{\IqCs}\in\rcset{\qSdpp}\) follows.
\eproof
 
\bleml{L1645}
 Let \(\ug\in\R\), let \(A\in\Cqq\), and let \(\rho_{\ug,A}^\diamond\colon\rcset{\qSdppa{A}}\to\SFdqaa{A}\) be defined by 
\beql{rhodA}
 \rho_{\ug,A}^\diamond\rk*{\copacl{\phi}{\psi}}
 \defeq\phi\psi^\inv.
\eeq
 Then \(\rho_{\ug,A}^\diamond\) is well defined and bijective with inverse \(\iota^\diamond_{\ug,A}\colon\SFdqaa{A}\to\rcset{\qSdppa{A}}\) given by
\beql{iotadA}
 \iota^\diamond_{\ug,A}\rk{F}
 \defeq\copacl{F}{\IqCs}.
\eeq
\elem
\bproof
 Taking into account \eqref{qSdppa} and \eqref{F4*1}, we have \(\qSdppa{A}\subseteq\qSdpp\) and \(\SFdqaa{A}\subseteq\SFdqa\).
 In view of \rlem{L1029}, it is thus sufficient to prove \(\rho_\ug^\diamond\rk{\rcset{\qSdppa{A}}}\subseteq\SFdqaa{A}\) and \(\iota_\ug^\diamond\rk{\SFdqaa{A}}\subseteq\rcset{\qSdppa{A}}\), where \(\rho_\ug^\diamond\colon\rcset{\qSdpp}\to\SFdqa\) and \(\iota_\ug^\diamond\colon\SFdqa\to\rcset{\qSdpp}\) are given by \eqref{rhod} and \eqref{iotad}, respectively.
 First we consider an arbitrary pair \(\copa{\phi}{\psi}\in\qSdppa{A}\).
 According to \rlem{L1029}, then \(F\defeq\phi\psi^\inv=\rho_\ug^\diamond\rk{\copacl{\phi}{\psi}}\) belongs to \(\SFdqa\).
 Because of \eqref{qSdppa}, the pair \(\copa{\phi}{\psi}\) belongs to \(\qSpa{A}\).
 \rnota{D1230} and \rremp{L1631}{L1631.b} hence imply \(AA^\mpi\phi=\phi\).
 Since \(F\) is holomorphic in \(\Cs\), a continuity argument provides \(AA^\mpi F=F\). Thus, \(F\in\SFqr{A}\) according to \rremp{L1631}{L1631.b} and \rnota{D1157}.
 Consequently, \rrem{R1433} yields \(F\in\SFdqaa{A}\).
 Conversely, now we consider an arbitrary function \(F\in\SFdqaa{A}\).
 Because of \rlem{L1029}, then \(\copacl{F}{\IqCs}=\iota_\ug\rk{F}\) belongs to \(\rcset{\qSdpp}\).
 In particular, \(\copa{F}{\IqCs}\in\qSdpp\).
 Since \rrem{R1433} yields \(F\in\SFqr{A}\), we have \(\copa{F}{\IqCs}\in\qSpa{A}\) by virtue of \eqref{qSdpp} and \rnotass{D1157}{D1230}.
 In view of \eqref{qSdppa}, thus \(\copacl{F}{\IqCs}\in\rcset{\qSdppa{A}}\).
\eproof
 
\bthml{T1247}
 Let \(\ug\in\R\), let \(m\in\NO\), and let \(\seqs{m}\in\Kggeq{m}\) with \trasp{\ug} \(\seqfp{m}\).
 Then \(\tilde\Sigma^\diamond\colon\rcset{\qSdppa{\Spu{m}}}\to\SFqas{m}\) given by
\beql{tSigma}
 \tilde\Sigma^\diamond\rk*{\copacl{\phi}{\psi}}
 \defeq\rk{\rmiupnwou{s}{m}\phi+\rmiupneou{s}{m}\psi}\rk{\rmiupswou{s}{m}\phi+\rmiupseou{s}{m}\psi}^\inv
\eeq
 is well defined and bijective.
\ethm
\bproof
 First observe that \eqref{qSdppa} yields \(\qSdppa{\Spu{m}}\subseteq\qSpa{\Spu{m}}\).
 \rthm{T1615} shows that \(\Spu{m}=\su{0}^\sta{m}\).
 In view of \rcor{C1507} it is thus sufficient to show \(\Sigma\rk{\rcset{\qSdppa{\Spu{m}}}}=\SFqas{m}\), where \(\Sigma\colon\rcset{\qSpa{\Spu{m}}}\to\SFqaskg{m}\) is given by \eqref{Sigma}.
 First we consider an arbitrary pair \(\copa{\phi}{\psi}\in\qSdppa{\Spu{m}}\).
 According to \rlem{L1645}, then \(F\defeq\phi\psi^\inv\) belongs to \(\SFdqaa{\Spu{m}}\), the pair \(\copa{F}{\IqCs}\) belongs to \(\qSdppa{\Spu{m}}\), and \(\copacl{\phi}{\psi}=\copacl{F}{\IqCs}\).
 From~\zitaa{MR3611471}{\cprop{12.13} and \cthm{13.1(a)}} we can conclude that \(\Sigma\rk{\copacl{F}{\IqCs}}\) belongs to \(\SFqas{m}\).
 Because of \(\copacl{\phi}{\psi}=\copacl{F}{\IqCs}\), then \(\Sigma\rk{\copacl{\phi}{\psi}}\in\SFqas{m}\).
 Conversely, now we consider an arbitrary \(S\in\SFqas{m}\).
 From~\zitaa{MR3611471}{\cthm{13.1(a)}} we get that there is a function \(F\in\SFdqaa{\Spu{m}}\) with \(S=\rk{\rmiupnwou{s}{m}F+\rmiupneou{s}{m}}\rk{\rmiupswou{s}{m}F+\rmiupseou{s}{m}}^\inv\).
 Consequently, since \rlem{L1645} implies \(\copa{F}{\IqCs}\in\qSdppa{\Spu{m}}\), from \eqref{Sigma} we have \(\Sigma\rk{\copacl{F}{\IqCs}}=S\).
\eproof

\bleml{L0717}
 Let \(\ug\in\R\) and let \(A \in \Cqp\) with rank \(r\geq 1\). 
 Let \(u_1, u_2,\dotsc, u_r\) be an orthonormal basis of \(\ran{A}\), let \(U \defeq  \mat{ u_1, u_2,\dotsc, u_r }\), and let \(Q\defeq\OPu{\nul{A^\ad}}\). 
 Then \(\tilde\Gamma^\diamond_U\colon\rcset{\aSdpp{r}}\to\rcset{\qSdppa{A}}\) given by 
\beql{tGamma}
 \tilde\Gamma^\diamond_U\rk*{\copacl{\phi}{ \psi}}
 \defeq\copacl{U \phi U^\ad }{U \psi U^\ad  +   Q  }
\eeq
 is well defined an bijective.
\elem
\bproof
 First observe that \(\aSdpp{r}\subseteq\aSp{r}\) by virtue of \eqref{qSdpp}.
 Furthermore, \eqref{qSdppa} shows that \(\qSdppa{A}\subseteq\qSpa{A}\).
 In view of \rlem{L1033}, it is thus sufficient to show \(\Gamma_U\rk{\rcset{\aSdpp{r}}}=\rcset{\qSdppa{A}}\), where \(\Gamma_U\colon\rcset{\aSp{r}}\to\rcset{\qSpa{A}}\) is given by \eqref{SAl_MG}.
 Observe that the same reasoning as in the proof of \rlem{L1049} yields \eqref{L1049.1}, \eqref{L1049.2}, and \eqref{L1049.3}.
 According to \eqref{L1049.1} and \eqref{L1049.3}, we have
\beql{P}
 QU
 =Q^\ad U
 =\rk{U^\ad Q}^\ad
 =\Ouu{q}{r}.
\eeq
 First we consider an arbitrary pair \(\copa{\phi}{\psi}\in\aSdpp{r}\).
 Then \(\copa{\phi}{\psi}\) belongs to \(\aSpp{r}\) by virtue of \eqref{qSdpp} and \(\lim_{y\to\infp}\normS{\rk{\phi\psi^\inv}(\iu y)}=0\).
 Because of \rdefn{D1151}, then \(\copa{\phi}{\psi}\) belongs to \(\aSp{r}\) and \(\det\psi\) does not identically vanish in \(\Cs\).
 Let \(F\defeq  U \phi U^\ad \) and \(G\defeq  U\psi U^\ad  +   Q  \).
 Then \eqref{L1049.2} implies \eqref{L1049.4}.
 In view of \rlem{L1033}, the equivalence class \(\copacl{F}{G}=\Gamma_U\rk{\copacl{\phi}{\psi}}\) belongs to \(\rcset{\qSpa{A}}\).
 In particular, \(\copa{F}{G}\in\qSpa{A}\).
 \rnota{D1230} shows that \(\copa{F}{G}\) belongs to \(\qSp\).
 Using \eqref{L1049.2}, \eqref{L1049.3}, \eqref{P}, and \eqref{L1049.1}, we obtain
\[
 G\rk{U\psi^\inv U^\ad+Q}
 =U\psi U^\ad U\psi^\inv U^\ad+U\psi U^\ad Q+QU\psi^\inv U^\ad+Q^2
 =U\psi \psi^\inv U^\ad+Q^2
 =\Iq.
\]
 Hence, \(\det G\) does not identically vanish in \(\Cs\) and \(G^\inv=U\psi^\inv U^\ad+Q\).
 In particular, \(\copa{F}{G}\in\qSpp\).
 Using \eqref{L1049.2} and \eqref{L1049.3}, we get then
\[
 FG^\inv
 =U \phi U^\ad\rk{U\psi^\inv U^\ad+Q}
 =U \phi U^\ad U\psi^\inv U^\ad+U \phi U^\ad Q
 =U \phi\psi^\inv U^\ad.
\]
 The equation \(\lim_{y\to\infp}\normS{\rk{\phi\psi^\inv}(\iu y)}=0\) implies \(\lim_{y\to\infp}\normS{\rk{FG^\inv}(\iu y)}=0\).
 Thus, \(\copa{F}{G}\) belongs to \(\qSdpp\) according to \eqref{qSdpp}.
 In view of \(\copa{F}{G}\in\qSpa{A}\), then \(\copa{F}{G}\in\qSdppa{A}\) by virtue of \eqref{qSdppa}.
 Hence, \(\Gamma_U\rk{\copacl{\phi}{\psi}}\in\rcset{\qSdppa{A}}\). 
 Conversely, consider now an arbitrary pair \(\copa{F}{G}\in\qSdppa{A}\).
 Then \(\copa{F}{G}\) belongs to \(\qSdpp\cap\qSpa{A}\) because of \eqref{qSdppa}.
 Taking into account \eqref{qSdpp} and \rdefn{D1151}, we infer then \(\copa{F}{G}\in\qSp\), that \(\det G\) does not identically vanish in \(\Cs\), and that \(\lim_{y\to\infp}\normS{\rk{FG^\inv}(\iu y)}=0\).
 Since \(\Gamma_U\) is bijective by virtue of \rlem{L1033}, then there exists a pair \(\copa{\phi}{\psi}\in\aSp{r}\) with \(\copacl{F}{G}=\Gamma_U\rk{\copacl{\phi}{\psi}}\), \ie{} \(\copacl{F}{G}=\copacl{U\phi U^\ad}{U\psi U^\ad+Q}\).
 According to \rdefn{D0710}, there exists a function \(\theta\in\ek{\mero{\Cs}}^\x{q}\) such that \(\det\theta\) does not identically vanish in \(\Cs\) and that \(U\phi U^\ad=F\theta\) and \(U\psi U^\ad+Q=G\theta\) are satisfied.
 First assume \(r=q\).
 Then \(Q=\Oqq\) and \(U\) is unitary.
 Hence, \(\phi=U^\ad F\theta U\) and \(\psi=U^\ad G\theta U\).
 In particular, \(\det\psi\) does not identically vanish in \(\Cs\) and \(\psi^\inv=U^\ad\theta^\inv G^\inv U\).
 Thus, \(\copa{\phi}{\psi}\) belongs to \(\aSpp{r}\) and \(\phi\psi^\inv=U^\ad FG^\inv U\).
 Because of \(\lim_{y\to\infp}\normS{\rk{FG^\inv}(\iu y)}=0\), this implies \(\lim_{y\to\infp}\normS{\rk{\phi\psi^\inv}(\iu y)}=0\).
 Thus, \(\copa{\phi}{\psi}\) belongs to \(\aSdpp{r}\) according to \eqref{qSdpp}.
 Hence, \(\copacl{\phi}{\psi}\in\rcset{\aSdpp{r}}\) and \(\Gamma_U\rk{\copacl{\phi}{\psi}}=\copacl{F}{G}\).
 Now we assume \(r\leq q-1\).
 We choose then \(u_{r+1}, u_{r+2},\dotsc, u_q\in \Cq\) such that \(u_1, u_2,\dotsc, u_q\) is an orthonormal basis of \(\Cq\). 
 Setting \(V \defeq  \mat{ u_{r+1}, u_{r+2},\dotsc, u_q }\), we see that \(W\defeq\mat{U,V}\) is unitary.
 Hence,
\begin{align}\label{L0717.1}
 \bMat
  U^\ad U&U^\ad V\\
  V^\ad U&V^\ad V
 \eMat
 =W^\ad W
 &=\Iq&
&\text{and}&
 UU^\ad+VV^\ad
 =WW^\ad
 &=\Iq.
\end{align}
 From \eqref{L1049.2} and \eqref{L0717.1}, we get \(Q=VV^\ad\).
 Using additionally \eqref{L0717.1}, we have then
\[\begin{split}
 W^\ad G\theta W
 &=\matp{U^\ad}{V^\ad}\rk{U\psi U^\ad+VV^\ad}\mat{U,V}
 =\matp{U^\ad U}{V^\ad U}\psi\mat{U^\ad U,U^\ad V}+\matp{U^\ad V}{V^\ad V}\mat{V^\ad U,V^\ad V}\\
 &=\matp{\Iu{r}}{\Ouu{\rk{q-r}}{r}}\psi\mat{\Iu{r},\Ouu{r}{\rk{q-r}}}+\matp{\Ouu{r}{\rk{q-r}}}{\Iu{q-r}}\mat{\Ouu{\rk{q-r}}{r},\Iu{q-r}}
 =\diag\rk{\psi,\Iu{q-r}}
\end{split}\]
 and, analogously, \(W^\ad F\theta W=\diag\rk{\phi,\Ouu{\rk{q-r}}{\rk{q-r}}}\).
 In particular, \(\det\psi\) does not identically vanish in \(\Cs\) and \(\diag\rk{\psi^\inv,\Iu{q-r}}=W^\ad\theta^\inv G^\inv W\).
 Thus, \(\copa{\phi}{\psi}\) belongs to \(\aSpp{r}\) and
\[\begin{split}
 \diag\rk{\phi\psi^\inv,\Ouu{\rk{q-r}}{r}}
 =W^\ad F\theta WW^\ad\theta^\inv G^\inv W
 =W^\ad FG^\inv W.
\end{split}\]
 Because of \(\lim_{y\to\infp}\normS{\rk{FG^\inv}(\iu y)}=0\), this implies \(\lim_{y\to\infp}\normS{\rk{\phi\psi^\inv}(\iu y)}=0\).
 Thus, \(\copa{\phi}{\psi}\) belongs to \(\aSdpp{r}\) according to \eqref{qSdpp}.
 Hence, \(\copacl{\phi}{\psi}\in\rcset{\aSdpp{r}}\) and \(\Gamma_U\rk{\copacl{\phi}{\psi}}=\copacl{F}{G}\).
\eproof

 In \rthm{T1227}, we in particular stated a complete description of the set \(\SFqas{m}\) in the completely degenerate case.
 In the alternate situation, the following two results give parametrizations of \(\SFqas{m}\), where the sets of the free parameters are \(\rcset{\aSdpp{r}}\) and \(\SFdua{r}\), respectively.

\bthml{P1504}
 Let \(\ug\in\R\), let \(m\in\NO\), and let \(\seqs{m}\in\Kggeq{m}\) be such that \(r\defeq\rank\Spu{m}\geq1\), where \(\seqfp{m}\) is the \traspa{\ug}{\seqs{m}}.
 Let \(u_1, u_2,\dotsc, u_r\) be an orthonormal basis of \(\ran{\Spu{m}}\), let \(U \defeq  \mat{ u_1, u_2,\dotsc, u_r }\), and let \(\tilde\Gamma^\diamond_U\colon\rcset{\aSdpp{r}}\to\rcset{\qSdppa{\Spu{m}}}\) be given by \eqref{tGamma}.
 Let \(\tilde\Sigma^\diamond\colon\rcset{\qSdppa{\Spu{m}}}\to\SFqas{m}\) be given by \eqref{tSigma}.
 Then \(\tilde\Sigma^\diamond\circ\tilde\Gamma_U^\diamond\) is well defined and bijective.
\ethm
\bproof
 The assertion is an immediate consequence of \rlem{L0717} and \rthm{T1247}.
\eproof

\bthml{P1725}
 Let \(\ug\in\R\), let \(m\in\NO\), and let \(\seqs{m}\in\Kggeq{m}\) be such that \(r\defeq\rank\Spu{m}\geq1\), where \(\seqfp{m}\) is the \traspa{\ug}{\seqs{m}}.
 Let \(\iota^\diamond_\ug\colon\SFdua{r}\to\rcset{\aSdpp{r}}\) be given by \eqref{iotad}.
 Let \(u_1, u_2,\dotsc, u_r\) be an orthonormal basis of \(\ran{\Spu{m}}\), let \(U \defeq  \mat{ u_1, u_2,\dotsc, u_r }\), and let \(\tilde\Gamma^\diamond_U\colon\rcset{\aSdpp{r}}\to\rcset{\qSdppa{\Spu{m}}}\) be given by \eqref{tGamma}.
 Let \(\tilde\Sigma^\diamond\colon\rcset{\qSdppa{\Spu{m}}}\to\SFqas{m}\) be given by \eqref{tSigma}.
 Then \(\tilde\Sigma^\diamond\circ\tilde\Gamma_U^\diamond\circ\iota^\diamond_\ug\) is well defined and bijective.
\ethm
\bproof
 The assertion is an immediate consequence of \rlem{L1029} and \rthm{P1504}.
\eproof

\bcorl{C1543}
 Let \(\ug\in\R\), let \(m\in\NO\), and let \(\seqs{m}\in\Kggeq{m}\) be such that \(r\defeq\rank\Spu{m}\) fulfills \(1\leq r\leq q-1\), where \(\seqfp{m}\) is the \traspa{\ug}{\seqs{m}}.
 Let \(\iota^\diamond_\ug\colon\SFdua{r}\to\rcset{\aSdpp{r}}\) be given by \eqref{iotad}.
 Let \(u_1, u_2,\dotsc, u_q\) be an orthonormal basis of \(\Cq\) such that \(u_1, u_2,\dotsc, u_r\) is an orthonormal basis of \(\ran{\Spu{m}}\) and let \(W \defeq  \mat{ u_1, u_2,\dotsc,u_q }\).
 Let \(U \defeq  \mat{ u_1, u_2,\dotsc, u_r }\) and let \(\tilde\Gamma^\diamond_U\colon\rcset{\aSdpp{r}}\to\rcset{\qSdppa{\Spu{m}}}\) be given by \eqref{tGamma}.
 Let \(\tilde\Sigma^\diamond\colon\rcset{\qSdppa{\Spu{m}}}\to\SFqas{m}\) be given by \eqref{tSigma}.
 Then the mapping \(\Theta\defeq\tilde\Sigma^\diamond\circ\tilde\Gamma_U^\diamond\circ\iota^\diamond_\ug\colon\SFdua{r}\to\SFqas{m}\) admits, for each \(f\in\SFdua{r}\), the representation
\begin{multline*}
 \ek*{\Theta\rk{f}}\rk{z}
 =\ek*{ \rmiupnwou{s}{m} W \cdot \diag \rk*{ f(z), \Ouu{(q-r)}{(q-r)} } + \rmiupneou{s}{m} W}\\
 \times\ek*{ \rmiupswou{s}{m} W \cdot \diag \rk*{f(z), \Ouu{(q-r)}{(q-r)} } + \rmiupseou{s}{m} W }^\inv
\end{multline*}
 for all \(z\in\Cs\).
\ecor
\bproof
 For all \(\copa{\phi}{\psi}\in\qSdpp\), in view of \eqref{qSdpp}, the comparison of \eqref{SAl_MG} and \eqref{tGamma} shows that \(\tilde\Gamma^\diamond_U\rk{\copacl{\phi}{\psi}}=\Gamma_U\rk{\copacl{\phi}{\psi}}\).
 Consider an arbitrary function \(f\in\SFdua{r}\).
 Then \(\copa{f}{\IuCs{r}}\) belongs to \(\aSp{r}\) and \rrem{SAl_VII4} yields  \(\copa{UfU^\ad}{U\cdot\IuCs{r}\cdot U^\ad+\OPu{\nul{\Spu{m}^\ad}}}=\copa{W\cdot\diag\rk{f,\OuuCs{q-r}}\cdot W^\ad}{\IqCs}\).
 Consequently, using \eqref{iotad}, \eqref{tGamma}, \eqref{tSigma}, and a continuity argument completes the proof.
\eproof

\appendix
\section{Some facts from matrix theory}\label{A1347}

\breml{R1353}
 Let \(A\in\Cqq\).
 Then straightforward computations show that \(\re\rk{zA}=\re\rk{z}\re\rk{A}-\im\rk{z}\im\rk{A}\) and \(\im\rk{zA}=\re\rk{z}\im\rk{A}+\im\rk{z}\re\rk{A}\) hold true for all \(z\in\C\).
 In particular, \(\re\rk{\iu A}=-\im A\) and \(\im\rk{\iu A}=\re A\) are valid.
\erem

\bleml{L1625}
 Let \(A\in\CHq\).
 Then \(B^\ad\rk{-\Jimq}B=-\Jimq\), where
\beql{BA1}
 B
 \defeq\bMat\Iq&A\\-A^\mpi&\Iq-A^\mpi A\eMat.
\eeq
\elem
\bproof
 In view of \(A^\ad=A\), the application of \rremp{L1631}{L1631.a} and \rrem{R1133} yield \(\rk{A^\mpi}^\ad=A^\mpi\) and \(A^\mpi A=AA^\mpi\).
 A straightforward calculation brings \(B^\ad\rk{-\Jimq}B=-\Jimq\).
\eproof

\breml{R0857}
 Let \(\mathcal{U}\) be a linear subspace of \(\Cq\) with dimension \(d\geq1\) and let \(u_1,u_2,\dotsc,u_d\) be an orthonormal basis of \(\mathcal{U}\).
 Then \(U\defeq\mat{u_1,u_2,\dotsc,u_d}\) fulfills \(U^\ad U=\Iu{d}\) and \(UU^\ad=\OPu{\mathcal{U}}\).
\erem

\begin{rem} \label{G312N}
For each \(A\in \Cpq\), we have \(\ran{A}^\bot  = \nul{A^\ad}\) and \(\nul{A}^\bot  = \ran{A^\ad}\). 
\end{rem}

\begin{rem}\label{R1417}
Let \(A \in \Cqqg\), and let \(B \in \CHq\), be such that \(B - A \in \Cqqg\). 
Then \(B \in \Cqqg\), \(\ran{A}\subseteq\ran{B}\), and \(\nul{B} \subseteq \nul{A}\).
\end{rem}

\begin{rem} \label{L1631} 
 Let \(A \in \Cpq\).
 Then the following statements hold true:
 \benui
  \il{L1631.a} \((A^\mpi )^\mpi  = A\), \((A^\mpi )^\ad =(A^\ad )^\mpi \), \(\nul{A^\mpi}=\nul{A^\ad}\), \(\ran{A^\mpi}=\ran{A^\ad}\), and \(\rank (AA^\mpi ) = \rank A\). 
  \il{L1631.b} Let \(r \in \N\) and \(B\in \Coo{p}{r}\).
  Then \(\ran{B} \subseteq \ran{A}\) if and only if \(AA^\mpi B = B\).
  \il{L1631.c} Let \(s \in \N\) and \(B\in \Coo{s}{q}\).
  Then \(\nul{A} \subseteq \nul{C}\) if and only if \(CA^\mpi A = C\).
 \eenui
\end{rem}

\begin{prop} \label{MZ}
Let \(A \in \Cpq\) and let \(G \in \Cqp\). 
Then \(G = A^\mpi \) if and only if \(AG = \OPu{\ran{A}}\) and \(GA = \OPu{\ran{G}}\) hold true.
\end{prop}

A proof of \rprop{MZ} is given, \eg{}, in~\cite[Theorem 1.1.1, p. 15]{MR1152328}.

\begin{rem} \label{G318N}
Let \(A \in \Cpq\). In view of \rprop{MZ} and \rrem{L1631}, then one can easily see that \(\Iq  - AA^\mpi  = \OPu{\ran{A}^\bot }\).
\end{rem}

\breml{R1133}
 If \(A\in\CHq\), then \(AA^\mpi=A^\mpi A\).
\erem

\bleml{L.AEP}
 Let \(E \in \Coo{(p+q)}{(p+q)}\) and let
\beql{Eabcd}
E 
=
\begin{bmatrix}
a & b \\ 
c & d 
\end{bmatrix}
\eeq
 be the block partition of \(E\) with \tppa{block} \(a\). 
 Then the following statements are equivalent:
\begin{itemize}
\item[(i)] The matrix \(E\) is \tnnH{}.
\item[(ii)] \(a\in \Cppg\), \(\ran{b} \subseteq \ran{a}\), \(c=b^\ad \), and \(d - ca^\mpi b \in \Cqqg\).
\item[(iii)] \(d\in \Cqqg\), \(\ran{c} \subseteq \ran{d}\), \(b=c^\ad \), and \(a -bd^\mpi c \in \Cppg\).
\end{itemize}
\elem

 A proof of \rlem{L.AEP} is given, \eg{}, in~\zitaa{MR1152328}{\clem{1.1.9}}.

\begin{rem}\label{KM}
 Let the matrix \(E \in \C^{(p+r) \times (q+r)}\) be contractive and let \eqref{Eabcd} be the block partition of \(E\) with \tpqa{block} \(a\). 
 If \(d =\Iu{r}\), then it is readily checked that 
 \(b = \NM_{p \times r}\) and \(c = \NM_{r \times q}\). 
\end{rem}

\blemnl{\zitaa{MR3380267}{\clem{A.7}}}{L1315}
 Let \(A,B\in\CHq\).
 Then the following statements are equivalent:
 \baeqi{0}
  \il{L1315.i} \(\Oqq\leq B\leq A\).
  \il{L1315.ii} \(\Oqq\leq B^\mpi BA^\mpi BB^\mpi\leq B^\mpi\) and \(\nul{A}\subseteq\nul{B}\).
  \il{L1315.iii} \(\Oqq\leq BA^\mpi B\leq B\) and \(\nul{A}\subseteq\nul{B}\).
  \il{L1315.iv} \(\Ouu{2q}{2q}\leq\tmat{A&B\\ B&B}\).
 \eaeqi
 If~\ref{L1315.i} is fulfilled, then \(\nul{BA^\mpi B}=\nul{B}\) and \(\ran{BA^\mpi B}=\ran{B}\).
\elem

 We will use the \taaa{2q}{2q}{signature} matrices\index{j@\(\Jimq\)}\index{j@\(\Jreq\)}
\begin{align}
 \Jimq&\defeq\bMat\Oqq&-\iu\Iq\\ \iu\Iq&\Oqq\eMat&%
&\text{and}&
 \Jreq&\defeq\bMat\Oqq&-\Iq\\-\Iq&\Oqq\eMat\label{Jre}.
\end{align}

\breml{R1404}
 For all \(A,B\in\Cqq\), the equations \(\tmatp{A}{B}^\ad\rk{-\Jimq}\tmatp{A}{B}=2\im\rk{B^\ad A}\) and \(\tmatp{A}{B}^\ad\rk{-\Jreq}\tmatp{A}{B}=2\re\rk{B^\ad A}\) hold true.
 In particular, \(\tmatp{A}{\Iq}^\ad\rk{-\Jimq}\tmatp{A}{\Iq}=2\im A\) and \(\tmatp{A}{\Iq}^\ad\rk{-\Jreq}\tmatp{A}{\Iq}=2\re A\) are valid for each \(A\in\Cqq\).
\erem

\breml{L1409}
 Let \(A,B\in\Cqq\).
 Suppose \(\det B\neq0\).
 In view of \rrem{R1404}, it is readily checked then that \(\im\rk{AB^\inv}=B^\invad\tmatp{A}{B}^\ad\rk{-\Jimq}\tmatp{A}{B}B^\inv\).
\erem

\begin{rem}\label{R55}
 Suppose \(q \geq 2\). 
 Let \(r\in Z_{1, q-1}\), let \(A\in \C^{r \times r}\), and let \(B \in \C^{r \times r}\). 
 In view of \rrem{R1404}, it is readily checked then that 
\begin{multline*}
\begin{bmatrix}
\diag \rk{ A, \Ouu{(q-r)}{(q-r)} }\\
\diag \rk{ B, \Iu{q-r} }
\end{bmatrix}^\ad 
\rk{ - \Jq }
\begin{bmatrix}
\diag \rk{ A, \Ouu{(q-r)}{(q-r)} }\\
\diag \rk{ B, \Iu{q-r} }
\end{bmatrix}\\
= 
\diag 
\rk*{ 
\matp{A}{B}^\ad 
\rk{ -\Jq }
\matp{A}{B},\quad 
\Ouu{(q-r)}{(q-r)}
}.
\end{multline*}
\end{rem}

\breml{R1634}
 Let \(H\in\Cqq\).
 Then \(\ek{\diag\rk{H,\Iq}}^\ad\rk{-\Jimq}\ek{\diag\rk{H,\Iq}}=\tmat{\Oqq&\iu H^\ad\\-\iu H&\Oqq}\).
\erem

\bleml{L0827}
 Let \(\ug\in\R\) and \(H\in\CHq\).
 Then:
\benui
 \il{L0827.a} \(\ek{\diag\rk{H,H^\mpi}}^\ad\rk{-\Jimq}\ek{\diag\rk{H,H^\mpi}}=\ek{\diag\rk{HH^\mpi,\Iq}}^\ad\rk{-\Jimq}\ek{\diag\rk{HH^\mpi,\Iq}}\).
 \il{L0827.b} \(\ek{\diag\rk{\rk{z-\ug}H,H^\mpi}}^\ad\rk{-\Jimq}\ek{\diag\rk{\rk{z-\ug}H,H^\mpi}}\\
 =\ek{\diag\rk{\rk{z-\ug}HH^\mpi,\Iq}}^\ad\rk{-\Jimq}\ek{\diag\rk{\rk{z-\ug}HH^\mpi,\Iq}}\) for each \(z\in\C\).
\eenui
\elem
\bproof
 \eqref{L0827.a} In view of \(H^\ad=H\), from \rremp{L1631}{L1631.a} we infer \(\ek{\diag\rk{H,H^\mpi}}^\ad=\diag\rk{H,H^\mpi}\).
 Furthermore, because of \(H^\ad=H\) and \rrem{R1133}, we have \(H^\mpi H=HH^\mpi\).
 Using additionally \rrem{R1634} and \(\rk{HH^\mpi}^\ad=HH^\mpi\), we conclude then
\begin{multline*}
 \ek*{\diag\rk{H,H^\mpi}}^\ad\rk{-\Jimq}\ek*{\diag\rk{H,H^\mpi}}
 =\diag\rk{H,H^\mpi}
 \bMat\Oqq&\iu\Iq\\-\iu\Iq&\Oqq\eMat
 \ek*{\diag\rk{H,H^\mpi}}\\
 =\bMat\Oqq&\iu HH^\mpi\\-\iu H^\mpi H&\Oqq\eMat
 =\bMat\Oqq&\iu HH^\mpi\\-\iu HH^\mpi&\Oqq\eMat
 =\ek*{\diag\rk{HH^\mpi,\Iq}}^\ad\rk{-\Jimq}\ek*{\diag\rk{HH^\mpi,\Iq}}.
\end{multline*}

 \eqref{L0827.b} Using~\eqref{L0827.a}, for each \(z\in\C\), we obtain
\[\begin{split}\\
 &\ek*{\diag\rk*{\rk{z-\ug}H,H^\mpi}}^\ad\rk{-\Jimq}\ek*{\diag\rk*{\rk{z-\ug}H,H^\mpi}}\\
 &=\ek*{\diag\rk*{\rk{z-\ug}\Iq,\Iq}}^\ad\ek*{\diag\rk{H,H^\mpi}}^\ad\rk{-\Jimq}\ek*{\diag\rk{H,H^\mpi}}\ek*{\diag\rk*{\rk{z-\ug}\Iq,\Iq}}\\
 &=\ek*{\diag\rk*{\rk{z-\ug}\Iq,\Iq}}^\ad\ek*{\diag\rk{HH^\mpi,\Iq}}^\ad\rk{-\Jimq}\ek*{\diag\rk{HH^\mpi,\Iq}}\ek*{\diag\rk*{\rk{z-\ug}\Iq,\Iq}}\\
 &=\ek*{\diag\rk*{\rk{z-\ug}HH^\mpi,\Iq}}^\ad\rk{-\Jimq}\ek*{\diag\rk*{\rk{z-\ug}HH^\mpi,\Iq}}.\qedhere
\end{split}\]
\eproof

\section{Particular results of the integration theory with respect \tnnH{} measures}
\label{Appendix_Int}
 Let \(\Om\) be a non-empty set and let \(\A\) be a \(\sigma\)\nobreakdash-algebra on \(\Om\).
 If \(\nu\) is a measure on the measurable space \(\rk{ \Om, \A }\), for each \(s \in (0, \infp)\), then we use \(\Loaaaa{s}{\Om}{\A}{\nu}{\C}\) to denote the set of all \(\A\)\nobreakdash-\(\BorC \)\nobreakdash-measurable functions \(f\colon \Om \to \C\) such that \(\int_{\Om} \abs{f}^s \dif\nu < \infp\).
 A matrix-valued function \(\mu\) whose domain is \(\A\) and whose values belong to the set \(\Cqqg\) of all \tnnH{} complex \tqqa{matrices} is called \tnnH{} \tqqa{measure} on \(\rk{ \Om, \A }\) if it is countable additive, \ie{}, if \(\mu\) fulfills \(\mu \rk{ \bigcup^{\infi}_{k=1} A_k } = \sum^{\infi}_{k=1} \mu (A_k)\) for each sequence \(\rk{ A_k }_{k=1}^{\infi}\) of pairwise disjoint sets that belong to \(\A\). 
 By \(\Mggqa{\Om,\A}\) we denote the set of all \tnnH{} \tqqa{measure}s on \(\rk{ \Om, \A }\). 
 We use some basic facts from the integration theory with respect to \tnnH{} measures (see~\zita{MR0080280} and~\zita{MR0163346}). 

 Let \(\mu = \matauuo{\mu_{jk}}{j,k}{1}{q} \in \Mggqa{\Om,\A}\).
 Then each entry function \(\mu_{jk}\) is a complex measure on \(\rk{ \Om, \A }\) and, in particular, the measure \(\tau\defeq  \tr \mu\) is a finite (\tnn{} real-valued) measure, the so-called trace measure of \(\mu\).
 For every choice of \(j\) and \(k\) in \(\mn{1}{q}\), the complex measure \(\mu_{jk}\) is absolutely continuous with respect to \(\tau\) and, consequently, the matrix-valued function \(\mu'_{\tau} =\matauuo{\dif\mu_{jk}/\dif\tau}{j,k}{1}{q}\) of the Radon-Nikodym derivatives is well defined up to sets of zero \(\tau\)\nobreakdash-measure.
 We will write \(\BorCuu{p}{q} \) for the \(\sigma\)\nobreakdash-algebra of all Borel subsets of \(\Cpq\).
 An ordered pair \(\copa{\Phi}{\Psi}\) consisting of an \(\A\)\nobreakdash-\(\BorCuu{p}{q} \)\nobreakdash-measurable function \(\Phi\colon\Om \to \Cpq\) and an \(\A\)\nobreakdash-\(\BorCuu{r}{p}\)\nobreakdash-measurable function \(\Psi\colon\Om \to \C^{r \times p}\) is said to be left-integrable with respect to \(\mu\) if \(\Phi \mu'_{\tau} \Psi^\ad \) belongs to \(\ek{ \LaaaC{\Om}{\A}{\tau} }^{p \times r}\).
 In this case, for each \(A \in \A\), the integral \(\int_A \Phi \dif\mu \Psi^\ad  \defeq  \int_A \Phi(\omega) \mu'_{\tau}(\omega) \Psi^\ad (\omega) \tau(\dif\omega)\) is well defined. 
 The class of all \(\A\)\nobreakdash-\(\BorCuu{p}{q} \)\nobreakdash-measurable functions \(\Phi\colon\Om \to \Cpq\) for which \(\copa{\Phi}{\Phi}\) is left-integrable with respect to \(\mu\) is denoted by \(\aaLsqa{p}{q}{\Om}{\A}{\mu}\).
 Observe that if \(\Phi \in \aaLsqa{p}{q}{\Om}{\A}{\mu}\) and if \(\Psi \in \aaLsqa{r}{q}{\Om}{\A}{\mu}\), then the pair \(\copa{\Phi}{\Psi}\) is integrable with respect to \(\mu\).
 Note that one can easily check that the set \(\LaaaC{\Om}{\A}{\mu}\), consisting of all Borel-measurable functions \(f\colon\Omega\to\C\) for which the integral \(\int_\Omega f\dif\mu\) exists, coincides with \(\LaaaC{\Om}{\A}{\tau}\).

\begin{rem} \label{M1534}
 The set \(\LaaaC{\Om}{\A}{\mu}\) coincides with the set of all \(\A\)\nobreakdash-\(\BorC \)\nobreakdash-measurable functions \(f\colon \Om \to \C\) for which the pair \(\copa{ f \Iq }{ \Iq  }\) is left-integrable with respect to \(\mu\). 
 If \(f \in \LaaaC{\Om}{\A}{\mu}\), then it is readily checked that \(\int_A f \dif\mu = \int_A \rk{ f \Iq } \dif\mu \Iq ^\ad \) for all \(A \in \A\).
 Furthermore, if \(f\) and \(g\) are \(\A\)\nobreakdash-\(\BorC \)\nobreakdash-measurable complex-valued functions defined on \(\Om\), then \(f \ko{g} \in \LaaaC{\Om}{\A}{\mu}\) if and only if the pair \(\copa{ f\Iq }{ g\Iq  }\) is left-integrable with respect to \(\mu\). 
 In this case, one can easily see that \(\int_A f \ko{g} \dif\mu = \int_A \rk{ f\Iq  } \dif\mu \rk{ g\Iq  }^\ad \) for all \(A \in \A\). 
\end{rem}

\begin{rem} \label{NBR}
 Let \((\Om,\A)\) be a measurable space, let \(\mu \in \Mggqa{\Om,\A}\), and let \(f \in \LaaaC{\Om}{\A}{\tr\mu}\). 
 Then \(\ko{f}\) belongs to \(\LaaaC{\Om}{\A}{\tr\mu}\) as well and \(\int_A \ko{f} \dif\mu = \rk{ \int_A f \dif\mu }^\ad \) for all \(A \in \A\).
\end{rem}

\begin{prop} \label{AMH11}
 Let \((\Om,\A)\) be a measurable space and let \(\mu \in \Mggqa{\Om,\A}\).
 Furthermore, let 
\begin{align} \label{AMH11_1}
 \Phi&\in \aaLsqa{p}{q}{\Om}{\A}{\mu}&
&\text{and}&
 \Psi&\in \aaLsqa{r}{q}{\Om}{\A}{\mu}.
\end{align}
 Then 
\begin{align} \label{AMH11_2}
 \Nul{ \int_{\Om} \Phi \dif\mu \Phi^\ad  }&\subseteq\Nul{ \int_{\Om} \Psi \dif\mu \Phi^\ad  },&
 \Ran{ \int_{\Om} \Phi \dif\mu \Psi^\ad  }&\subseteq\Ran{ \int_{\Om} \Phi \dif\mu \Phi^\ad  },
\end{align}
and
\beql{AMH11_3}
 \rk*{ \int_{\Om} \Psi \dif\mu \Phi^\ad  }\rk*{ \int_{\Om} \Phi \dif\mu \Phi^\ad  }^\mpi \rk*{ \int_{\Om} \Phi \dif\mu \Psi^\ad  }
 \leq\int_{\Om} \Psi \dif\mu \Psi^\ad .
\eeq
\end{prop}
\begin{proof}
 Because of \eqref{AMH11_1}, we have \(\copa{\Phi}{\Psi} \in \aaLsqa{\rk{p+r}}{q}{\Om}{\A}{\mu}\) and 
\[
 \begin{bmatrix}
  \int_{\Om} \Phi \dif\mu \Phi^\ad & \int_{\Om} \Phi \dif\mu \Psi^\ad \\
  \int_{\Om} \Psi \dif\mu \Phi^\ad & \int_{\Om} \Psi \dif\mu \Psi^\ad 
 \end{bmatrix}
 =\int_{\Om}\matp{\Phi}{\Psi}\dif\mu\matp{\Phi}{\Psi}^\ad 
 \in \Cggo{(p+r)}.
\]
 In view of \rlem{L.AEP}, this implies the second inclusion in \eqref{AMH11_2} and \eqref{AMH11_3}. 
 Obviously,
\begin{align*}
 \rk*{ \int_{\Om} \Phi \dif\mu \Psi^\ad  }^\ad &=\int_{\Om} \Psi \dif\mu \Phi^\ad  &
 &\text{and}&
 \rk*{ \int_{\Om} \Phi \dif\mu \Phi^\ad  }^\ad &=\int_{\Om} \Phi \dif\mu \Phi^\ad . 
\end{align*}
 Thus, the first inclusion in \eqref{AMH11_2} follows from the second one and \rrem{G312N}.
\end{proof}

\begin{cor} \label{AMH13}
 Let \((\Om,\A)\) be a measurable space, let \(\mu \in \Mggqa{\Om,\A}\), and let \(f,g \in \Loaaaa{2}{\Om}{\A}{\tr\mu}{\C}\). 
 Then \(\set{ f, g, \abs{f}^2, \abs{g}^2, \ko{f}g, f\ko{g}}\subseteq \LaaaC{\Om}{\A}{\mu}\),
\begin{align} \label{AMH13_2}
 \Nul{ \int_{\Om} \abs{f}^2 \dif\mu }&\subseteq\Nul{ \int_{\Om} \ko{f}g \dif\mu },&
 \Ran{ \int_{\Om} \ko{f}g \dif\mu }&\subseteq\Ran{ \int_{\Om} \abs{f}^2 \dif\mu }, 
\end{align}
 and 
\beql{AMH13_3}
 \rk*{ \int_{\Om} \ko{f}g \dif\mu }\rk*{ \int_{\Om} \abs{f}^2 \dif\mu }^\mpi \rk*{ \int_{\Om} f\ko{g} \dif\mu }
 \leq \int_{\Om} \abs{g}^2 \dif\mu.
\eeq
\end{cor}
\begin{proof}
 Since \(f\) and \(g\) belong to \(\Loaaaa{2}{\Om}{\A}{\tr\mu}{\C}\) and because of \(\tau (\Om) < \infp\), we have \(\set{f, g, \abs{f}^2, \abs{g}^2, \ko{f}g, f\ko{g}}\subseteq \LaaaC{\Om}{\A}{\tr\mu}\). 
 In view of \(\LaaaC{\Om}{\A}{\mu} = \LaaaC{\Om}{\A}{\tau}\), then \(\set{ f, g, \abs{f}^2, \abs{g}^2, \ko{f}g, f\ko{g}}\subseteq \LaaaC{\Om}{\A}{\mu}\) follows. 
 Setting \(\Phi \defeq  f\Iq \) and \(\Psi \defeq  g\Iq \), we get \eqref{AMH13_2} and \eqref{AMH13_3} from \rprop{AMH11} and \rrem{M1534}.
\end{proof}

\begin{cor} \label{AMH15}
 Let \((\Om,\A)\) be a measurable space, let \(\mu \in \Mggqa{\Om,\A}\), and let \(f \in \Loaaaa{2}{\Om}{\A}{\tr\mu}{\C}\).
 Then \(\set{ f,  \ko{f}, \abs{f}^2}\subseteq \LaaaC{\Om}{\A}{\mu}\),
\begin{gather} 
 \Nul{ \int_{\Om} \abs{f}^2 \dif\mu }\subseteq\Nul{ \int_{\Om} f \dif\mu }\cap \Nul{ \rk*{ \int_{\Om} f \dif\mu }^\ad  } \label{AMH15_1}\\
 \Ran{ \int_{\Om} f \dif\mu }+\Ran{ \rk*{ \int_{\Om} f \dif\mu }^\ad  } \subseteq\Ran{ \int_{\Om} \abs{f}^2 \dif\mu },\label{AMH15_2}\\
 \rk*{ \int_{\Om} f \dif\mu}^\ad  \rk*{ \int_{\Om} \abs{f}^2 \dif\mu}^\mpi \rk*{ \int_{\Om} f \dif\mu}\leq \mu (\Om),\label{AMH15_3}
\end{gather}
and 
\beql{AMH15_4}
 \rk*{ \int_{\Om} f \dif\mu} \rk*{ \int_{\Om} \abs{f}^2 \dif\mu}^\mpi \rk*{ \int_{\Om} f \dif\mu}^\ad 
 \leq \mu (\Om).
\eeq
\end{cor}
\begin{proof}
 Setting \(g \defeq  1_{\Om}\) and applying \rcor{AMH13}, we conclude \(\nul{ \int_{\Om} \abs{f}^2 \dif\mu}\subseteq \nul{ \int_{\Om} \ko{f} \dif\mu}\) and \(\ran{\int_{\Om} F \dif\mu}\subseteq \ran{\int_{\Om} \abs{f}^2 \dif\mu}\). 
 Substituting \(f\) by \(\ko{f}\), we obtain \(\nul{ \int_{\Om} \abs{f}^2 \dif\mu}\subseteq \nul{ \int_{\Om} f \dif\mu}\) and \(\ran{\int_{\Om} \ko{f} \dif\mu}\subseteq \ran{\int_{\Om} \abs{f}^2 \dif\mu}\). 
 Because of \rrem{NBR}, then \eqref{AMH15_1} and \eqref{AMH15_2}. 
 The inequality in \eqref{AMH15_3} is an immediate consequence of \rcor{AMH15}, where \(g\defeq 1_{\Om}\) is chosen in \eqref{AMH15_3}.
 Similarly, \eqref{AMH15_4} follows from \rcor{AMH15} where \(f\) is substituted by \(\ko{f}\).
\end{proof}

\begin{cor} \label{AMH16}
 Let \((\Om,\A)\) be a measurable space, let \(\mu \in \Mggqa{\Om,\A}\), and let \(g \in \Loaaaa{2}{\Om}{\A}{\tr\mu}{\C}\). 
 Then \(\set{ g,  \ko{g}, \abs{g}^2}\subseteq \LaaaC{\Om}{\A}{\mu}\),
\begin{gather} 
 \Nul{ \mu(\Om) }\subseteq\Nul{ \int_{\Om} g \dif\mu }\cap\Nul{ \rk*{ \int_{\Om} g \dif\mu }^\ad  },\label{AMH16_1}\\
 \Ran{ \int_{\Om} g \dif\mu }+\Ran{ \rk*{ \int_{\Om} g \dif\mu }^\ad  }\subseteq\Ran{ \mu (\Om) },\label{AMH16_2}\\
 \rk*{ \int_{\Om} g \dif\mu}\rk*{ \mu (\Om) }^\mpi \rk*{ \int_{\Om} g \dif\mu}^\ad 
 \leq\rk*{  \int_{\Om} \abs{g} \dif\mu }^\ad, \label{AMH16_3}
\end{gather}
and 
\beql{AMH16_4}
 \rk*{ \int_{\Om} g \dif\mu}^\ad \rk*{ \mu (\Om) }^\mpi \rk*{ \int_{\Om} g \dif\mu}
 \leq\rk*{  \int_{\Om} \abs{g} \dif\mu }^\ad .
\eeq
\end{cor}
\begin{proof}
 Let \(f \defeq  1_{\Omega}\).
 Applying \rcor{AMH13}, we conclude \(\nul{ \mu(\Om) } \subseteq \nul{ \int_{\Om} g \dif\mu }\) and \(\ran{\int_{\Om} \ko{g} \dif\mu  }\subseteq \ran{\mu(\Om) }\). 
 Substituting \(g\) by \(\ko{g}\) then  \(\nul{ \mu(\Om) } \subseteq \nul{ \rk{ \int_{\Om} g \dif\mu }^\ad  }\) and \(\ran{\int_{\Om} g \dif\mu  }\subseteq \ran{\mu(\Om) }\) follow.
 In view of \rrem{NBR}, we get then \eqref{AMH16_1} and \eqref{AMH16_2}. 
 Similarly, \eqref{AMH16_3} and \eqref{AMH16_4} can be proved using \rcor{AMH13} and \rrem{NBR}.
\end{proof}

\section{Meromorphic matrix-valued functions}\label{A1430}
 By a \notii{region} we mean an open connected subset of \(\C\).
 Let \(\mathcal{G}\) be a region in \(\C\) and let \(f\) be a complex function in \(\mathcal{G}\).
 Then \(f\) is called \notii{meromorphic in \(\mathcal{G}\)} if there exists a discrete subset \symba{\pol{f}}{p} of \(\mathcal{G}\) such that \(f\) is holomorphic in \symba{\hol{f}\defeq\mathcal{G}\setminus\pol{f}}{h} whereas \(f\) has a pole in each point of \(\pol{f}\).
 We denote by \symba{\mero{\mathcal{G}}}{m} the set of all meromorphic functions in \(\mathcal{G}\).
 The notation \symba{\holo{\mathcal{G}}}{h} stands for the set of all complex-valued holomorphic functions in \(\mathcal{G}\).
 
 Now we want to extend these notions to matrix-valued functions.
 Let \(\mathcal{G}\) be a region in \(\C\) and let \(r,s\in\N\).
 
 Let \(f=\mat{f_{jk}}_{\substack{j=1,\dotsc,r\\ k=1,\dotsc,s}}\in\ek{\mero{\mathcal{G}}}^\xx{r}{s}\).
 Then the sets \(\hol{f}\defeq\bigcap_{j=1}^r\bigcap_{k=1}^s\hol{f_{jk}}\)\index{h@\(\hol{f}\defeq\bigcap_{j=1}^r\bigcap_{k=1}^s\hol{f_{jk}}\)} and \(\pol{f}\defeq\bigcup_{j=1}^r\bigcup_{k=1}^s\pol{f_{jk}}\)\index{p@\(\pol{f}\defeq\bigcup_{j=1}^r\bigcup_{k=1}^s\pol{f_{jk}}\)} are called the holomorphicity set of \(f\) and the pole set of \(f\), respectively.
 Then one can easily see that \(\pol{f}\) is a discrete subset of \(\mathcal{G}\) and that \(\hol{f}\cup\pol{f}=\mathcal{G}\) and \(\hol{f}\cap\pol{f}=\emptyset\) hold true.
 We consider an \(f\in\ek{\mero{\mathcal{G}}}^\xx{r}{s}\) also as a mapping \(f\) between the sets \(\hol{f}\) and \(\Coo{r}{s}\).
 
 Now we consider two classes of \tqqa{matrix-valued} holomorphic functions.
 
 Let \(\mathcal{G}\) be a region in \(\C\).
 A function \(f\in\ek{\holo{\mathcal{G}}}^\x{q}\) is called a \tqCf{} (resp.\ \tqHNf{}) in \(\mathcal{G}\) if \(\re f\rk{\mathcal{G}}\subseteq\Cggq\) (resp.\ \(\im f\rk{\mathcal{G}}\subseteq\Cggq\)).
 We denote by \symba{\Cfq{\mathcal{G}}}{c} (resp.\ \(\Rfq{\mathcal{G}}\)) the set of all \tqCf{s} (resp.\ \tqHNf{s}) in \(\mathcal{G}\).

\breml{R1505}
 Let \(\mathcal{G}\) be a region in \(\C\).
 Then \rrem{R1353} shows that \(\Rfq{\mathcal{G}}=\setaa{\iu f}{f\in\Cfq{\mathcal{G}}}\) and \(\Cfq{\mathcal{G}}=\setaa{-\iu g}{g\in\Rfq{\mathcal{G}}}\).
\erem

\bleml{L1508}
 Let \(\mathcal{G}\) be a region in \(\C\) and let \(\mathcal{D}\) be a discrete subset of \(\mathcal{G}\).
 \benui
  \il{L1508.a} Let \(f\in\Cfq{\mathcal{G}\setminus\mathcal{D}}\). Then each \(w\in\mathcal{D}\) is a removable singularity of \(f\) and the extended function \(\tilde f\) satisfies \(\tilde f\in\Cfq{\mathcal{G}}\).
  \il{L1508.b} Let \(f\in\Rfq{\mathcal{G}\setminus\mathcal{D}}\). Then each \(w\in\mathcal{D}\) is a removable singularity of \(f\) and the extended function \(\tilde f\) satisfies \(\tilde f\in\Rfq{\mathcal{G}}\).
 \eenui
\elem
\bproof
 \rPart{L1508.a} follows from~\zitaa{MR1152328}{\clem{2.1.9}}.
 \rPart{L1508.b} is a consequence of~\eqref{L1508.a} and \rrem{R1505}.
\eproof

\section{On Linear Fractional Transformations of Matrices}\label{S*B}
 In this appendix, we summarize some basic facts on linear fractional transformations of matrices.
 A systematic treatment of this topic was handled by V.~P.~Potapov in~\zita{MR566141} (see also~\zitaa{MR1152328}{\csec{1.6}}).
 We slightly extend the concept developed by V.~P.~Potapov by studying the more general version of linear fractional transformations of pairs of complex matrices. 
 It should be mentioned that V.~P.~Potapov~\zitaa{MR566141}{pp.~80--81} observed that sometimes there are situations where linear fractional transformations of pairs of complex matrices arise, but not treated this case.
 We did not succeed in finding a convenient hint in the public literature. 
 That's why we state the corresponding results with short proofs.

\begin{nota}\label{LinFracTrans}
 Let \(E \in \Coo{(p+q)}{(p+q)}\) and let \eqref{Eabcd} be the block partition of \(E\) with \tppa{block} \(a\). 
 If \(\rank \mat{c,d} = q\), then the linear fractional transformations \symba{\lftroou{p}{q}{E} \colon \dblftr{c}{d} \to \Cpq}{t} and \symba{\lftpoou{p}{q}{E} \colon \dblftp{c}{d} \to \Cpq}{t} are defined by 
\begin{align*}
 \lftrooua{p}{q}{E}{x}& \defeq  (ax+ b)(cx + d)^\inv&%
&\text{and}&
 \lftpoouA{p}{q}{E}{\copa{x}{y}} &\defeq  (ax + by)(cx + dy)^\inv.%
\end{align*}
\end{nota}

\begin{lem} \label{Bi13_B4}
 Let \(c \in \Cqp\) and \(d \in \Cqq\).
 Then the following statements are equivalent:
\begin{aeqi}{0}
 \il{Bi13_B4.i} The set \(\dblftr{c}{d} \defeq\setaa{ x \in \Cpq}{\det \rk{ cx + d} \neq 0}\) is non-empty.
 \il{Bi13_B4.ii} The set \(\dblftp{c}{d}\defeq\setaa{\copa{x}{y} \in \Cpq \times\Cqq}{\det \rk{ cx + dy  } \neq 0}\) is non-empty.
 \il{Bi13_B4.iii} \(\rank \mat{c,d} = q\).
\end{aeqi}
 Furthermore, \(\dblftp{c}{d}\) is a subset of the set \(\mathcal{Q}_{p\times q}\) of all pairs \(\copa{x}{y} \in \Cpq \times\Cqq\) which fulfill \(\rank \tmatp{x}{y} = q\).
\end{lem}
\begin{proof}
\bimp{Bi13_B4.i}{Bi13_B4.ii} 
 Choose \(y=\Iq \).
\eimp

\bimp{Bi13_B4.ii}{Bi13_B4.iii} 
 If \(\copa{x}{y} \in \dblftp{c}{d}\), then 
\beql{Bi13_B4_1}
 q
 = \rank (cx + dy)
 = \rank \rk*{\mat{c,d}\matp{x}{y}}
 \leq \min \set*{\rank \mat{c,d}, \rank \matp{x}{y} }
 \leq q.
\eeq
\eimp

\bimp{Bi13_B4.iii}{Bi13_B4.i} 
 This implication is proved, \eg{}, in \cite[\clem{1.6.1}, \cpage{52}]{MR1152328}.
\eimp

 The rest follows from \eqref{Bi13_B4_1}.
\end{proof}

\begin{prop}[see, \eg{}~\zitaa{MR1152328}{\cprop{1.6.3}}]\label{PB*1}
 Let \(a_1,a_2\in\Cpp\), let \(b_1,b_2\in\Cpq\), let \(c_1,c_2\in\Cqp\), and let \(d_1,d_2\in\Cqq\) be such that \(\rank\mat{ c_1,d_1}=\rank\mat{ c_2,d_2}=q\).
 Furthermore, let
\begin{align} \label{BDE}
E_1& \defeq  
\begin{bmatrix}
a_1 & b_1 \\ c_1 & d_1 
\end{bmatrix},&
E_2 &\defeq  
\begin{bmatrix}
a_2 & b_2 \\ c_2 & d_2 
\end{bmatrix},
\end{align}
 let \(E_2\defeq\bsma a_2&b_2\\ c_2&d_2\esma\), let \(E\defeq E_2E_1\), and let \eqref{Eabcd} be the block representation of \(E\) with \tppa{block} \(a\). Then \(\mathcal{Q}\defeq\setaa{x\in\dblftruu{c_1}{d_1}}{\lftrooua{p}{q}{E_1}{x}\in\dblftruu{c_2}{d_2}}\) is a nonempty subset of the set \(\dblftruu{c}{d}\) and \(\lftrooua{p}{q}{E_2}{\lftrooua{p}{q}{E_1}{x}}=\lftrooua{p}{q}{E}{x}\) holds true for all \(x\in\mathcal{Q}\).
\end{prop}

\begin{prop}\label{Bi13_B8}
 Let \(E_1, E_2 \in \Coo{(p+q)}{(p+q)}\) and let the block partitions of \(E_1\) and \(E_2\) with \tppa{block}s \(a_1\) and \(a_2\) be given by \eqref{BDE}. 
 Let \(E \defeq  E_2 E_1\) and let \eqref{Eabcd} be the block partition of \(E\) with \tppa{block} \(a\). 
 Suppose that \(\rank\mat{c_1, d_1} = q\) and \(\rank\mat{c_2, d_2} = q\) hold true.
 Let \(\tilde{\mathcal{Q}} \defeq \setaa{\copa{x}{y} \in\dblftp{c_1}{d_1}}{\lftpooua{p}{q}{E_1}{ \copa{x}{y} } \in \dblftr{c_2}{d_2}}\). 
 Then \(\dblftp{c}{d} \cap \dblftp{c_1}{d_1}= \tilde{\mathcal{Q}}\).
 Furthermore, if \(\dblftp{c}{d} \cap \dblftp{c_1}{d_1} \neq \emptyset\), then \(\lftrfua{p}{q}{E_2}{\lftpfua{p}{q}{E_1}{\copa{x}{y} }}=\lftpfua{p}{q}{E}{ \copa{x}{y} }\) for all \(\copa{x}{y} \in \dblftp{c}{d} \cap \dblftp{c_1}{d_1}\).
\end{prop}
\begin{proof} 
 It is readily checked that
\begin{align} \label{Bi13_B8_1}
a& = a_2a_1 + b_2c_1,&
b &= a_2b_1 + b_2d_1,&
c &= c_2a_1 + d_2c_1,&
&\text{and}&
d& = c_2b_1 + d_2d_1.
\end{align}
 Suppose \(\tilde{\mathcal{Q}}\neq \emptyset\). 
 We consider an arbitrary \(\copa{x}{y} \in \tilde{\mathcal{Q}}\). 
 Then \(\rk{ \copa{x}{y} } \in \dblftp{c_1}{d_1}\) and \(z\defeq\lftpfuA{p}{q}{E_1}{\copa{x}{y}}\) belongs to \(\dblftr{c_2}{d_2}\), \ie{}, \(\det (c_2z + d_2) \neq 0\) is true. 
 Since \eqref{Bi13_B8_1} implies 
\begin{multline}\label{Bi13_B8_2} 
c_2 z + d_2 
= c_2 (a_1x + b_1y) (c_1x + d_1y)^\inv + d_2 (c_1x + d_1y) (c_1x + d_1y)^\inv  \\
= \ek*{ c_2 (a_1x + b_1y) + d_2 (c_1x + d_1y) } (c_1x + d_1y)^\inv  
= (cx + dy)(c_1x + d_1y)^\inv,  
\end{multline}
 we get det \((c x + dy) \neq 0\), \ie{}, that \(\copa{x}{y}\) belongs to \(\dblftr{c}{d}\). 
 Thus, \(\tilde{\mathcal{Q}} \subseteq \dblftr{c}{d} \cap \dblftp{c_1}{d_1}\) is checked. 
 Similar to \eqref{Bi13_B8_2}, we conclude \(a_2 z + b_2 = (ax + by)(c_1x + d_1y)^\inv\).
 Consequently, using \eqref{Bi13_B8_2}, then we also obtain 
\begin{multline*}
\lftrfuA{p}{q}{E_2} { \lftpfuA{p}{q}{E_1}{\copa{x}{y}} }
=(a_2 z + b_2 )(c_1x + d_1y)^\inv   \\
= (ax + by)(c_1x + d_1y)^\inv\ek*{ (cx + dy)(c_1x + d_1y)^\inv}^\inv
= \lftrfuA{p}{q}{E}{\copa{x}{y}}.
\end{multline*}
 It remains to prove that \(\dblftp{c}{d} \cap \dblftp{c_1}{d_1} \subseteq \tilde{\mathcal{Q}}\) is fulfilled. 
 For this reason, we assume that \(\copa{x}{y}\) belongs to  \(\dblftp{c}{d} \cap \dblftp{c_1}{d_1}\). 
 Setting \(z \defeq  \lftpfua{p}{q}{E_1}{\copa{x}{y}}\) again, we infer that \eqref{Bi13_B8_2} holds true. 
 Because of \eqref{Bi13_B8_2} and \(\copa{x}{y} \in \dblftp{c}{d}\), we get \(\det \rk{ c_2 z + d_2 } \neq 0\), \ie{}, \(z \in \dblftr{c_2}{d_2}\). 
 Hence, \(\copa{x}{y}\) belongs to \(\tilde{\mathcal{Q}}\).
\end{proof} 

 From \rprop{Bi13_B8} immediately the following result follows:

\bcornl{\zitaa{MR1152328}{\cprop{1.6.1}}}{DFK92_163}
 Let \(E_1, E_2 \in \Coo{(p+q)}{(p+q)}\), let \(E \defeq  E_2 E_1\), and let \eqref{BDE} and \eqref{Eabcd} be the block partitions of \(E_1\), \(E_2\), and \(E\) with \tppa{block}s \(a_1\), \(a_2\), and \(a\), respectively. 
 Suppose that \(\rank\mat{c_1, d_1} = q\) and \(\rank\mat{c_2, d_2} = q\) hold true.
 Then \(\dblftr{c}{d} \cap \dblftr{c_1}{d_1}=\setaa{x \in \dblftr{c_1}{d_1}}{\lftrooua{p}{q}{E_1}{x} \in \dblftr{c_2}{d_2}}\).
 Furthermore, if \(\dblftr{c}{d} \cap \dblftr{c_1}{d_1} \neq \emptyset\), then \(\lftrooua{p}{q}{E_2}{\lftrooua{p}{q}{E_1}{ x } }= \lftrooua{p}{q}{E}{ x }\) for all \(x \in \dblftr{c}{d} \cap \dblftr{c_1}{d_1}\).
\ecoro

 We modify \rnota{LinFracTrans} for matrix-valued functions:
 Let \(\cG\) be a non-empty subset of \(\C\), let \(V\colon\cG \to \Coo{(p+q)}{(p+q)}\) be a matrix-valued function, and let \(V = \tmat{v_{11} & v_{12} \\ v_{21} & v_{22}}\) be the block partition of \(V\) with \tppa{block} \(v_{11}\). 
 Then we make the following conventions: 
 If \(F\colon\cG \to \Cpq\) fulfills \(\det \ek{ v_{21}(z) F(z) + v_{22}(z) } \neq 0\) for all \(z \in \cG\), then we use \symba{\lftrfua{p}{q}{V}{F}}{t} to denote the function \(\lftrfua{p}{q}{V}{F}\colon\cG \to \Cpq\) defined by \(\ek{\lftrfua{p}{q}{V}{F}} (z) \defeq   \lftroouA{p}{q}{V(z)}{F(z)}\). 
 Furthermore, if \(F_1\colon\cG \to \Cpq\) and \(F_2\colon\cG \to \Cpq\) are matrix-valued functions with \(\det \ek{ v_{21}(z) F_1(z) + v_{22}(z) F_2(z) } \neq 0\) for all \(z \in \cG\), then \(\lftpfua{p}{q}{V}{(F_1,F_2)}\colon\cG \to \Cpq\) is given by \(\ek{\lftpfua{p}{q}{V}{(F_1,F_2)}} (z) \defeq   \lftpoouA{p}{q}{V(z)}{(F_1(z),F_2(z))}\).

 The following example shows that the conditions \(\dblftp{c_1}{d_1} \neq \emptyset\) and \(\dblftp{c_2}{d_2} \neq \emptyset\) do not imply \(\dblftp{c}{d}\cap \dblftp{c_1}{d_1} \neq \emptyset\). 

\begin{exa}
 With \(a_1 = b_1 = d_1 = a_2 = b_2 = d_2 = \Oqq\), and \(c_1 = c_2 = \Iq \) we have \(E = \NM_{2q \times 2q}\), \(\copa{ \Iq }{ \Iq  } \in \dblftp{c_1}{d_1}\), \(\Iq  \in \dblftr{c_2}{d_2}\), and \(\dblftp{c}{d} = \emptyset\).
\end{exa}

\bpropl{R_BEG_1}
 Let \(E_1, E_2 \in \Coo{(p+q)}{(p+q)}\) and let \eqref{BDE} be the block partitions of \(E_1\) and \(E_2\) with \tppa{block}s \(a_1\) and \(a_2\). 
 Let \(E \defeq  E_2 E_1\) and let \eqref{Eabcd} be the block partition of \(E\) with \tppa{block} \(a\). 
 Suppose \(\rank\mat{c_1, d_1} = q\) and \(\rank\mat{c_2, d_2} = q\). 
 Furthermore, let \(\copa{x}{y}\) and \(\copa{\tilde x}{ \tilde y}\) be from \(\Cpq \times \Cqq\) such that \(E_1\tmatp{x}{y}=\tmatp{\tilde x}{ \tilde y}\).  
 Then:
\benui
 \il{R_BEG_1.a} \(\copa{x}{y}\in\dblftp{c}{d}\) if and only if \(\copa{\tilde x}{ \tilde y}\in\dblftp{c_2}{d_2}\).
 \il{R_BEG_1.b} If \(\copa{x}{y} \in \dblftp{c}{d}\), then \(\lftpooua{p}{q}{E}{\copa{x}{y}} =\lftpooua{p}{q}{E_2}{ \copa{\tilde x}{ \tilde y}}\).
\eenui
\eprop
\bproof
 By assumption, we have \(\tilde x=a_1x+b_1y\) and \(\tilde y=c_1x+d_1y\).
 In view of \eqref{Bi13_B8_1}, then
\[
 a_2\tilde x+b_2\tilde y
 =a_2\rk{a_1x+b_1y}+b_2\rk{c_1x+d_1y}
 =\rk{a_2a_1+b_2c_1}x+\rk{a_2b_1+b_2d_1}y
 =ax+by
\]
 and
\[
 c_2\tilde x+d_2\tilde y
 =c_2\rk{a_1x+b_1y}+d_2\rk{c_1x+d_1y}
 =\rk{c_2a_1+d_2c_1}x+\rk{c_2b_1+d_2d_1}y
 =cx+dy.
\]
 Now~\eqref{R_BEG_1.a} and~\eqref{R_BEG_1.b} immediately follow.
\eproof

\section{The Matrix Polynomials \(\mHTiu{A}\) and \(\mHTu{A}\)}\label{A1557}
 Now we study special matrix polynomials, which were already intensively used in~\zita{MR3611471}.

\begin{rem}\label{T149_2} 
 Let \(\ug\in\R\) and let \(A \in \Cpq\). 
 Then \symba{\mHTiu{A}\colon\C\to \Coo{(p+q)}{(p+q)}}{v} and \symba{\mHTu{A}\colon\C\to \Coo{(p+q)}{(p+q)}}{w} given by
\begin{align}\label{VWaA}
\mHTiu{A}(z)&\defeq
\begin{bmatrix}
\Opp  & -A \\ \rk{z-\ug}A^\mpi  & \rk{z-\ug} \Iq 
\end{bmatrix}&
&\text{and}&
\mHTua{A}{z}& \defeq  \begin{bmatrix}
\rk{z-\ug} \Ip  & A \\ -\rk{z-\ug}A^\mpi  & \Iq  - A^\mpi A
\end{bmatrix},
\end{align}
 respectively, are linear \taaa{(p+q)}{(p+q)}{matrix} polynomials and, in particular, holomorphic in \(\C\).
 For all \(z\in\C\), furthermore
\begin{align}
   \mHTiua{A}{z}&=\diag\rk*{\Ip,\rk{z-\ug}\Iq}\cdot\bMat\Opp&-A\\A^\mpi&\Iq\eMat\label{VAZ}
  \intertext{and}
   \mHTua{A}{z}&=\bMat\Ip&A\\-A^\mpi&\Iq-A^\mpi A\eMat\cdot\diag\rk*{\rk{z-\ug}\Ip,\Iq}.\label{WAZ}
\end{align}
\end{rem}

 The use of the matrix polynomial \(\mHTiu{A}\) was inspired by some constructions in the paper~\cite{MR2038751}. 
 In particular, we mention~\cite[formula~(2.3)]{MR2038751}. 
 In their constructions, Hu and Chen used Drazin inverses instead of Moore-Penrsoe inverses of matrices. 
 Since both types of generalized inverses coincide for \tH{} matrices (see, \eg{}~\cite[\cprop{A.2}]{MR3014199}), we can conclude that in the generic case the matrix polynomials \(\mHTiu{ A}\) coincide for \(\ug = 0\) with the functions used in~\cite{MR2038751}.

\bremanl{\zitaa{MR3611471}{\crem{D.1}}}{MR3611471_D1}
 Let \(A \in \Cpq\) and let \(\ug \in \C\). 
 For each \(z \in \C\), then \(\mHTiu{ A} (z) \mHTu{ A} (z)= \rk{z-\ug} \cdot \diag \rk{ AA^\mpi , \Iq  }\) and \(\mHTu{ A} (z) \mHTiu{ A} (z) = \rk{z-\ug} \cdot \diag \rk{ AA^\mpi , \Iq  }\).
\erema

 Now we are going to study the linear fractional transformation generated by the matrix \(\mHTu{ A} (z)\) for arbitrarily given 
\(z \in \C\setminus \set{\alpha}\). 

\begin{lem}[\zitaa{MR3611471}{\clem{D.2}}]\label{MR3611471_D2}
 Let \(\ug \in \C\), let \(A \in \Cpq\), and let \(z \in \C \setminus \set{\alpha}\). 
 Then: 
\begin{enui}
 \il{MR3611471_D2.a} The matrix \(-\rk{z-\ug}^\inv A\) belongs to \(\dblftr{-\rk{z-\ug}A^\mpi }{ \Iq  - A^\mpi A}\). In particular, \(\dblftr{-\rk{z-\ug}A^\mpi }{ \Iq  - A^\mpi A} \neq \emptyset\).
 \il{MR3611471_D2.b} Let \(X \in \Cpq\) be such that \(\ran{A} \subseteq \ran{X}\) and \(\nul{A} \subseteq \nul{X}\). 
 Then \(X\) belongs to \(\dblftr{-\rk{z-\ug}A^\mpi }{ \Iq  - A^\mpi A}\) and the equation \(\ek{ -\rk{z-\ug}A^\mpi X + \Iq  - A^\mpi A }^\inv ={-\rk{z-\ug}^\inv X^\mpi A + \Iq  - A^\mpi A}\) holds true.
\end{enui}
\end{lem}

\begin{rem}\label{MR3611471_D2_2}
 Let \(\ug \in \C\), let \(A \in \Cpq\), and let \(z \in \C \setminus \set{\alpha}\). 
 In view of \rlemss{MR3611471_D2}{Bi13_B4}, then \(\dblftp{-\rk{z-\ug}A^\mpi }{ \Iq  - A^\mpi A} \neq \emptyset\).
\end{rem}

\begin{rem}\label{Z3N}
 Let \(\ug \in \C\), let \(A \in \Cpq\), and let \(z \in \C \setminus \set{\alpha}\). 
 Then \(\rank \mat{ \rk{z-\ug}A^\mpi , \rk{z-\ug}\Iq  } = q\) and the matrix \(\Oqq\) obviously belongs to \(\dblftr{\rk{z-\ug}A^\mpi }{ \rk{z-\ug}\Iq}\).
 In particular, \(\dblftr{\rk{z-\ug}A^\mpi }{ \rk{z-\ug}\Iq} \neq \emptyset\) and, in view of \rlem{Bi13_B4}, furthermore \(\dblftp{\rk{z-\ug}A^\mpi }{ \rk{z-\ug}\Iq}\neq \emptyset\).
\end{rem}

\bpropl{T2110_40}
 Let \(\ug \in \R\), let \(A \in \CHq\), and let the \taaa{2q}{2q}{matrix} polynomial \(\mHTu{A}\) be given via \eqref{VWaA}.
 Furthermore, let \(\Jimq\) be given via \eqref{Jre}.
 Let \(z \in \C\).
 Then 
\beql{WS}
\ek*{ \mHTua{A}{z} }^\ad  
\rk{ -\Jq }
\mHTua{A}{z}
=\ek*{ \diag\rk*{\rk{z-\ug}\Iq , \Iq} }^\ad 
\rk{ -\Jq }
\ek*{ \diag\rk*{\rk{z-\ug}\Iq , \Iq} }
\eeq
 and
\begin{multline}\label{N101N}
 \ek*{ \diag\rk*{\rk{z-\ug}\Iq , \Iq} \cdot \mHTu{A} (z) }^\ad 
\rk{ -\Jq }
\ek*{ \diag\rk*{\rk{z-\ug}\Iq , \Iq} \cdot \mHTu{A} (z) }  \\
= \abs{z-\ug}^2\rk*{
\ek*{ \diag\rk{ AA^\mpi , \Iq  }  }^\ad  
\rk{ -\Jimq} \ek*{ \diag\rk{ AA^\mpi , \Iq  }} - 2\im\rk{z} \cdot \diag \rk{ A^\mpi , \Oqq }} \\
 + \ek*{ \diag \rk*{ \rk{z-\ug}^2(\Iq  - AA^\mpi ), \Iq } }^\ad 
\rk{ -\Jq }
\ek*{ \diag \rk*{ \rk{z-\ug}^2(\Iq  - AA^\mpi ), \Iq } }.
\end{multline}
\eprop
\begin{proof}
 In view of \rrem{T149_2}, we have \eqref{WAZ}.
 Setting \eqref{BA1} and using \(A^\ad=A\), \rlem{L1625} yields \(B^\ad\rk{-\Jimq}B=-\Jimq\).
 Consequently, \eqref{WS} follows.
 Let
\beql{T149_2.2}
 E(z)
 \defeq\ek*{\diag\rk*{\rk{z-\ug}\Iq,\Iq}\mHTua{A}{z}}^\ad\rk{-\Jimq}\ek*{\diag\rk*{\rk{z-\ug}\Iq,\Iq}\mHTua{A}{z}}.
\eeq
 and let
\beql{T149_2.3}
 E(z)
 =\bMat E_{11}(z)&E_{12}(z)\\E_{21}(z)&E_{22}(z)\eMat
\eeq
 be the \tqqa{block} decomposition of \(E(z)\).
 Obviously, we have
\beql{T149_2.6}
 E(z)
 =\bMat\iu\ko{\rk{z-\ug}}A^\mpi&\iu\ko{\rk{z-\ug}}^2\Iq\\-\iu\rk{\Iq-A^\mpi A}&\iu\ko{\rk{z-\ug}}A\eMat
 \bMat\rk{z-\ug}^2\Iq&\rk{z-\ug}A\\-\rk{z-\ug}A^\mpi&\Iq-A^\mpi A\eMat.
\eeq
 By virtue of \(A^\ad=A\), \rremss{L1631}{R1133} yield \(\rk{A^\mpi}^\ad=A^\mpi\) and \(AA^\mpi=A^\mpi A\), which implies \(\rk{\Iq-A^\mpi A}A=\rk{\Iq-AA^\mpi}A=A-AA^\mpi A=\Oqq\).
 From \eqref{T149_2.3} and  \eqref{T149_2.6} we conclude
\begin{align*}
 E_{11}(z)
 &=\iu\ko{\rk{z-\ug}}\rk{z-\ug}^2A^\mpi-\iu\ko{\rk{z-\ug}}^2\rk{z-\ug}A^\mpi
 =\iu\abs{z-\ug}^2\ek*{\rk{z-\ug}-\ko{\rk{z-\ug}}}A^\mpi\notag\\
 &=\iu\abs{z-\ug}^2\ek*{\rk{z-\ug}-\rk{\ko z-\ug}}A^\mpi
 =\iu\abs{z-\ug}^2\rk{z-\ko z}A^\mpi
 =-2\abs{z-\ug}^2\im\rk{z}A^\mpi,\\
 E_{12}(z)
 &=\iu\abs{z-\ug}^2A^\mpi A+\iu\ko{\rk{z-\ug}}^2\rk{\Iq-A^\mpi A}
 =\iu\abs{z-\ug}^2AA^\mpi +\iu\ko{\rk{z-\ug}}^2\rk{\Iq-AA^\mpi },\\
 E_{21}(z)
 &=-\iu\rk{z-\ug}^2\rk{\Iq-A^\mpi A}-\iu\abs{z-\ug}^2AA^\mpi 
 =-\iu\rk{z-\ug}^2\rk{\Iq-AA^\mpi }-\iu\abs{z-\ug}^2AA^\mpi,%
\intertext{and}
 E_{22}(z)
 &=-\iu\rk{z-\ug}\rk{\Iq-A^\mpi A}A+\iu\ko{\rk{z-\ug}}A\rk{\Iq-A^\mpi A}
 =\Oqq.%
\end{align*}
 Consequently, in view of \eqref{T149_2.3}, we conclude
\beql{T149_2.13}\begin{split}
 E(z)
 &=
 \bMat\Oqq&\iu\abs{z-\ug}^2AA^\mpi\\
 -\iu\abs{z-\ug}^2AA^\mpi&\Oqq\eMat
 -\diag\rk*{2\abs{z-\ug}^2\im\rk{z}A^\mpi,\Oqq}\\
 &\qquad+\bMat\Oqq&\iu\ko{\rk{z-\ug}}^2\rk{\Iq-AA^\mpi}\\
 -\iu\rk{z-\ug}^2\rk{\Iq-AA^\mpi}&\Oqq\eMat.
\end{split}\eeq
 From \rrem{R1634} we get
\beql{T149_2.14}
 \bMat\Oqq&\iu\abs{z-\ug}^2AA^\mpi\\
 -\iu\abs{z-\ug}^2AA^\mpi&\Oqq\eMat
 =\abs{z-\ug}^2\ek*{\diag\rk{AA^\mpi,\Iq}}^\ad\rk{-\Jimq}\ek*{\diag\rk{AA^\mpi,\Iq}}
\eeq
 and
\begin{multline}\label{T149_2.15}
 \bMat\Oqq&\iu\ko{\rk{z-\ug}}^2\rk{\Iq-AA^\mpi}\\
 -\iu\rk{z-\ug}^2\rk{\Iq-AA^\mpi}&\Oqq\eMat\\
 =\ek*{\diag\rk*{\rk{z-\ug}^2\rk{\Iq-AA^\mpi},\Iq}}^\ad\rk{-\Jimq}\ek*{\diag\rk*{\rk{z-\ug}^2\rk{\Iq-AA^\mpi},\Iq}}.
\end{multline}
 Using \eqref{T149_2.2} and \eqref{T149_2.13}--\eqref{T149_2.15}, we obtain \eqref{N101N}.
\end{proof}

\bpropl{141_R0815}
 Let \(\ug\in\R\). 
 Further, let \(A\) and \(B\) be \tH{} complex \tqqa{matrices} such that \(\nul{A} \subseteq \nul{B}\). 
 For each \(z \in \C\), then
\begin{multline*}
 \ek*{ \diag(A,A^\mpi )\mHTiua{B}{z}}^\ad \rk{-\Jimq}\ek*{ \diag(A,A^\mpi )\mHTiu{B}(z)}  \\
 =\ek*{ \diag \rk*{ \rk{z-\ug}B, B^\mpi  } }^\ad  \rk{-\Jimq}\ek*{ \diag \rk*{ \rk{z-\ug}B, B^\mpi  } }+ 2\im\rk{z}\diag \rk{ \Oqq, B}
\end{multline*}
 and
\begin{multline*}
 \ek*{ \diag \rk*{ \rk{z-\ug}A,A^\mpi }\mHTiua{B}{z}}^\ad \rk{-\Jimq}\ek*{ \diag \rk*{ \rk{z-\ug}A,A^\mpi}\mHTiua{B}{z}}  \\
 =\abs{z-\ug}^2\ek*{ \diag (B, B^\mpi  ) }^\ad  \rk{-\Jimq}\ek*{ \diag (B, B^\mpi  )}. 
\end{multline*}
\eprop
\begin{proof}
 Because of the assumption that the matrices \(A\) and \(B\) are \tH{}, we see from \rremss{L1631}{R1133} that
\begin{align} \label{141_R0815.2}
 A^\mpi&=\rk{A^\mpi}^\ad,&
 B^\mpi&=\rk{B^\mpi}^\ad,&
 AA^\mpi&=A^\mpi A&
&\text{and}&
 BB^\mpi&=B^\mpi B.
\end{align}
 Since \(\nul{A}\subseteq\nul{B}\) is supposed, \rremp{L1631}{L1631.c} brings
\beql{141_R0815.3}
 BA^\mpi A
 =B.
\eeq
 The application of the assumption \(A^\ad=A\) and \(B^\ad=B\), \eqref{141_R0815.2}, and \eqref{141_R0815.3} yields
\begin{gather}
 AA^\mpi B
 =A^\ad\rk{A^\mpi}^\ad B^\ad
 =\rk{BA^\mpi A}^\ad
 =B^\ad
 =B,\label{141_R0815.4}\\
 \begin{aligned}\label{141_R0815.6}
 BAA^\mpi
 &=BA^\mpi A
 =B,&
&\text{and}&
 A^\mpi AB
 &=AA^\mpi B
 =B.
 \end{aligned}
\end{gather}
 From \eqref{141_R0815.6} and \eqref{141_R0815.2} it follows
\beql{141_R0815.8}
 B^\mpi A^\mpi AB
 =B^\mpi B
 =BB^\mpi.
\eeq
 Taking into account \eqref{VWaA}, we infer
\beql{141_R0815.9}
 \diag\rk{A,A^\mpi}\cdot\mHTiua{B}{z}
 =\bMat\Oqq&-AB\\\rk{z-\ug}A^\mpi B^\mpi&\rk{z-\ug}A^\mpi\eMat.
\eeq
 Using \eqref{141_R0815.6}, the assumption that \(A\) and \(B\) are \tH{}, and \eqref{141_R0815.2}, we get
\beql{141_R0815.10}
 \ek*{\diag\rk{A,A^\mpi}\cdot\mHTiua{B}{z}}^\ad
 =\bMat\Oqq&\ko{\rk{z-\ug}}B^\mpi A^\mpi\\-BA&\ko{\rk{z-\ug}}A^\mpi\eMat.
\eeq
 In view of \eqref{141_R0815.4} and \eqref{141_R0815.6}, we conclude
\beql{141_R0815.11}
 \iu\ko{\rk{z-\ug}}A^\mpi AB-\iu\rk{z-\ug}BAA^\mpi
 =-\iu\rk{z-\ko z}B
 =2\im\rk{z}B.
\eeq
 Because of \rrem{R1634} and \rlemp{L0827}{L0827.b}, we get
\beql{141_R0815.12}\begin{split}
 &\bMat\Oqq&\iu\ko{\rk{z-\ug}}BB^\mpi\\-\iu\rk{z-\ug}BB^\mpi&\Oqq\eMat\\
 &=\ek*{\diag\rk*{\rk{z-\ug}BB^\mpi,\Iq}}^\ad\rk{-\Jimq}\cdot\diag\rk*{\rk{z-\ug}BB^\mpi,\Iq}\\
 &=\ek*{\diag\rk*{\rk{z-\ug}B,B^\mpi}}^\ad\rk{-\Jimq}\cdot\diag\rk*{\rk{z-\ug}B,B^\mpi}.
\end{split}\eeq
Combining \eqref{141_R0815.10}, \eqref{141_R0815.9}, \eqref{141_R0815.6}, \eqref{141_R0815.8}, \eqref{141_R0815.4}, \eqref{141_R0815.11}, and \eqref{141_R0815.12}, we obtain
\beql{141_R0815.13}\begin{split}
 &\ek*{\diag\rk{A,A^\mpi}\cdot\mHTiua{B}{z}}^\ad\rk{-\Jimq}\ek*{\diag\rk{A,A^\mpi}\cdot\mHTiua{B}{z}}\\
 &=\bMat\Oqq&\ko{\rk{z-\ug}}B^\mpi A^\mpi\\-BA&\ko{\rk{z-\ug}}A^\mpi\eMat
 \bMat\Oqq&\iu\Iq\\-\iu\Iq&\Oqq\eMat
 \bMat\Oqq&-AB\\\rk{z-\ug}A^\mpi B^\mpi&\rk{z-\ug}A^\mpi\eMat\\
 &=\bMat-\iu\ko{\rk{z-\ug}}B^\mpi A^\mpi&\Oqq\\-\iu\ko{\rk{z-\ug}}A^\mpi&-\iu BA\eMat
 \bMat\Oqq&-AB\\\rk{z-\ug}A^\mpi B^\mpi&\rk{z-\ug}A^\mpi\eMat\\
 &=
 \begin{pmat}[{|}]
  \Oqq&\iu\ko{\rk{z-\ug}}B^\mpi A^\mpi AB\cr\-
  -\iu\rk{z-\ug}BAA^\mpi B^\mpi&\iu\ko{\rk{z-\ug}}A^\mpi AB-\iu\rk{z-\ug}BAA^\mpi\cr
 \end{pmat}\\
 &=\bMat\Oqq&\iu\ko{\rk{z-\ug}}BB^\mpi\\-\iu\rk{z-\ug}BB^\mpi&\Oqq\eMat+2\im\rk{z}\cdot\diag\rk{\Oqq,B}\\
 &=\ek*{\diag\rk*{\rk{z-\ug}B,B^\mpi}}^\ad\rk{-\Jimq}\ek*{\diag\rk*{\rk{z-\ug}B,B^\mpi}}+2\im\rk{z}\cdot\diag\rk{\Oqq,B}.
\end{split}\eeq
 Using \eqref{141_R0815.9}, we get
\beql{141_R0815.14}\begin{split}
 \diag\rk*{\rk{z-\ug}A,A^\mpi}\cdot\mHTiua{B}{z}
 &=\diag\rk*{\rk{z-\ug}\Iq,\Iq}\cdot\diag\rk{A,A^\mpi}\cdot\mHTiua{B}{z}\\
 &=\bMat\Oqq&-\rk{z-\ug}AB\\\rk{z-\ug}A^\mpi B^\mpi&\rk{z-\ug}A^\mpi\eMat.
\end{split}\eeq
 From \rremp{L1631}{L1631.a} and the assumption that \(A\) and \(B\) are \tH{}, we obtain \(\rk{A^\mpi}^\ad=A^\mpi\) and \(\rk{B^\mpi}^\ad=B^\mpi\).
 Thus, \eqref{141_R0815.14} implies
\beql{141_R0815.15}
 \ek*{\diag\rk*{\rk{z-\ug}A,A^\mpi}\mHTiua{B}{z}}^\ad
 =\bMat\Oqq&\ko{\rk{z-\ug}}B^\mpi A^\mpi\\-\ko{\rk{z-\ug}}BA&\ko{\rk{z-\ug}}A^\mpi\eMat.
\eeq
 By virtue of \eqref{141_R0815.3} and \eqref{141_R0815.4}, we have
\beql{141_R0815.16}
 \iu\abs{z-\ug}^2A^\mpi AB-\iu\abs{z-\ug}^2BAA^\mpi
 =\iu\abs{z-\ug}^2B-\iu\abs{z-\ug}^2B
 =\Oqq.
\eeq
 Taking into account \rrem{R1634} and \rlemp{L0827}{L0827.a}, it follows
\beql{141_R0815.17}\begin{split}
 \bMat\Oqq&\iu BB^\mpi\\-\iu BB^\mpi&\Oqq\eMat
 &=\ek*{\diag\rk{BB^\mpi,\Iq}}^\ad\rk{-\Jimq}\ek*{\diag\rk{BB^\mpi,\Iq}}\\
 &=\ek*{\diag\rk{B,B^\mpi}}^\ad\rk{-\Jimq}\ek*{\diag\rk{B,B^\mpi}}.
\end{split}\eeq
 Applying \eqref{141_R0815.15}, \eqref{141_R0815.14}, \eqref{141_R0815.6}, \eqref{141_R0815.16}, \eqref{141_R0815.2}, and \eqref{141_R0815.17}, we get
\[\begin{split}
 &\ek*{\diag\rk*{\rk{z-\ug}A,A^\mpi}\cdot\mHTiua{B}{z}}^\ad\rk{-\Jimq}\ek*{\diag\rk*{\rk{z-\ug}A,A^\mpi}\cdot\mHTiua{B}{z}}\\
 &=\bMat\Oqq&\ko{\rk{z-\ug}}B^\mpi A^\mpi\\-\ko{\rk{z-\ug}}BA&\ko{\rk{z-\ug}}A^\mpi\eMat
 \bMat\Oqq&\iu\Iq\\-\iu\Iq&\Oqq\eMat
 \bMat\Oqq&-\rk{z-\ug}AB\\\rk{z-\ug}A^\mpi B^\mpi&\rk{z-\ug}A^\mpi\eMat\\
 &=
 \begin{pmat}[{|}]
  \Oqq&\iu\abs{z-\ug}^2B^\mpi A^\mpi AB\cr\-
  -\iu\abs{z-\ug}^2BAA^\mpi B^\mpi&\iu\abs{z-\ug}^2A^\mpi AB-\iu\abs{z-\ug}^2BAA^\mpi\cr
 \end{pmat}\\
 &=\bMat\Oqq&\iu\abs{z-\ug}^2B^\mpi B\\-\iu\abs{z-\ug}^2BB^\mpi&\Oqq\eMat
 =\bMat\Oqq&\iu\abs{z-\ug}^2BB^\mpi\\-\iu\abs{z-\ug}^2BB^\mpi&\Oqq\eMat\\
 &=\abs{z-\ug}^2\bMat\Oqq&\iu BB^\mpi\\-\iu BB^\mpi&\Oqq\eMat
 =\abs{z-\ug}^2\ek*{\diag\rk{B,B^\mpi}}^\ad\rk{-\Jimq}\ek*{\diag\rk{B,B^\mpi}}.\qedhere
\end{split}\]
\end{proof}

\bcorol{C0932}
 Let \(\ug\in\R\) and let \(B\in\CHq\).
 For each \(z\in\C\), then
\begin{multline*}
 \ek*{\mHTiua{B}{z}}^\ad\rk{-\Jimq}\ek*{\mHTiua{B}{z}}\\
 =\ek*{\diag\rk*{\rk{z-\ug}B,B^\mpi}}^\ad\rk{-\Jimq}\cdot\diag\rk*{\rk{z-\ug}B,B^\mpi}+2\im\rk{z}\cdot\diag\rk{\Oqq,B}
\end{multline*}
 and
\begin{multline*}
 \ek*{\diag\rk*{\rk{z-\ug}\Iq,\Iq}\cdot\mHTiua{B}{z}}^\ad\rk{-\Jimq}\ek*{\diag\rk*{\rk{z-\ug}\Iq,\Iq}\cdot\mHTiua{B}{z}}\\
 =\abs{z-\ug}^2\ek*{\diag\rk{B,B^\mpi}}^\ad\rk{-\Jimq}\cdot\diag\rk{B,B^\mpi}.
\end{multline*}
\ecoro
\bproof
 From \(\nul{\Iq}=\set{\Ouu{q}{1}}\subseteq\nul{B}\) the assertion follows from \rprop{141_R0815}.
\eproof

 Let \(\ug \in \R\), let \(\kappa \in \NOinf \), let \(\sjk\) be a sequence of complex \tpqa{matrices}, and let \(m \in \mn{0}{\kappa}\). 
 For all \(l \in \mn{0}{m}\), let \((s_j^\sta{l})_{j=0}^{\kappa - l}\) be the \(l\)\nobreakdash-th \(\al\)\nobreakdash-\(S\)-transform of \(\sjk\). 
 Let the sequence \(\rk{ \mHTiu{ \su{0}^{[j,\ug]}} }_{j=0}^m\) be given via \eqref{VWaA}, let  
\beql{Def_fVam}
\rmiupou{s}{m}
\defeq  \mHTiu{ \su{0}^\sta{0}}\mHTiu{ \su{0}^\sta{1}}\dotsm \mHTiu{ \su{0}^{[m-1,\ug]}}\mHTiu{ \su{0}^\sta{m}},
\eeq
and let  
\beql{BD_fVam}
\rmiupou{s}{m}
=
\begin{bmatrix}
\rmiupnwou{s}{m} & \rmiupneou{s}{m}\\
\rmiupswou{s}{m} & \rmiupseou{s}{m}
\end{bmatrix}
\eeq
 be the \tqqa{block} representation of \(\rmiupou{s}{m}\) with \tppa{block} \(\rmiupnwou{s}{m}\).
 Furthermore, let the sequence \(\rk{ \mHTu{ \su{0}^{[j,\ug]}} }_{j=0}^m\) be given via \eqref{VWaA}, let  
\beql{Def_fWam}
\rmupou{s}{m}
\defeq  \mHTu{ \su{0}^\sta{0}} \mHTu{ \su{0}^\sta{1}} \dotsm \mHTu{ \su{0}^{[m-1,\ug]}} \mHTu{ \su{0}^\sta{m}},
\eeq
 and let  
\[
\rmupou{s}{m} =
\begin{bmatrix}
\rmupnwou{s}{m} & \rmupneou{s}{m}
\\
\rmupswou{s}{m} & \rmupseou{s}{m}
\end{bmatrix}
\]
 be the \tqqa{block} representation of \(\rmupou{s}{m}\) with \tppa{block} \(\rmupnwou{s}{m}\).

\begin{rem} \label{R1646}
 Let \(\ug \in \R\), let \(\kappa \in \Ninf \), and let \(\sjk\) be a sequence of complex \tpqa{matrices}.
 For all \(m \in \mn{1}{\kappa}\) and all \(l \in \mn{0}{m-1}\), one can see then from \eqref{Def_fVam}, \eqref{Def_fWam}, and~\zitaa{MR3611479}{\crem{8.3}} that
\begin{align*}
 \rmiupou{s^\sta{l}}{m-l} &= \rmiupou{s^\sta{l}}{m-(l+1)}  \mHTiu{ \su{0}^\sta{m}},& 
 \rmiupou{s^\sta{l}}{m-l} &= \mHTiu{ \su{0}^\sta{l}} \rmiupou{t^\sta{l}}{m-(l+1)},
\intertext{and}
 \rmupou{s^\sta{l}}{m-l} &= \rmupou{s^\sta{l}}{m-(l+1)} \mHTu{ \su{0}^\sta{m}},&
 \rmupou{s^\sta{l}}{m-l} &= \mHTu{ \su{0}^\sta{l}} \rmupou{t^\sta{l}}{m-(l+1)}
\end{align*}
 hold true where \(t_j \defeq  s_j^\sta{l+1}\) for all \(j \in \mn{0}{m-(l+1)}\).
\end{rem}

\begin{lem}\label{141_438}
 Let \(\ug \in \R\), let \(m \in \NO \), and let \(\seqs{m}\) be a sequence of complex \tpqa{matrices}.
 Let \(\rmupou{s}{m}\) and \(\rmiupou{s}{m}\) be given by \eqref{Def_fWam} and \eqref{Def_fVam}, respectively.
 For each \(z \in \C\), then 
\beql{141_438_1}
 \rmupoua{s}{m}{z} \rmiupou{s}{m}(z) 
 = \rk{z-\ug}^{m+1} \cdot\diag \rk*{ \su{0}^\sta{m} \rk{ \su{0}^\sta{m} }^\mpi , \Iq }.
\eeq
\end{lem}
\begin{proof} 
 Because of \eqref{Def_fWam}, \eqref{Def_fVam}, and \rrem{MR3611471_D1}, we have
\beql{141_438_2}
 \rmupoua{s}{0}{z} \rmiupou{s}{0}(z)
 =  \mHTu{ \su{0}^\sta{0}}(z) \mHTiu{ \su{0}^\sta{0}}(z)
 = \rk{z-\ug} \diag \rk*{ \su{0}^\sta{0} \rk{ \su{0}^\sta{0} }^\mpi , \Iq  }.
\eeq
 Hence, there is an \(n \in \N\) such that \eqref{141_438_1} is fulfilled for each \(m \in \mn{0}{n-1}\). 
 Now we are going to prove that \eqref{141_438_1} is also true for \(m = n\). 
 In view of~\cite[\crem{8.5}]{MR3611479}, we have \(\ran{\su{0}^\sta{n} } \subseteq \ran{\su{0}^\sta{n-1} }\).
 Thus, \rrem{L1631} yields \(\su{0}^\sta{n-1} \rk{ \su{0}^\sta{n-1} }^\mpi  \su{0}^\sta{n} = \su{0}^\sta{n}\). 
 Using this and additionally \rrem{R1646}, \eqref{VWaA}, and \eqref{141_438_2}, we obtain then 
\[\begin{split}
 &\rmupoua{s}{n}{z} \rmiupoua{s}{n}{z}
 =\mHTu{ \su{0}^\sta{n}}(z) \rmupoua{s}{n-1}{z} \rmiupoua{s}{n-1}{z}\mHTiu{ \su{0}^\sta{n}}(z)  \\
 &=
 \begin{bmatrix}
  \rk{z-\ug} \Ip  &\su{0}^\sta{n} \\
  -\rk{z-\ug} \rk{ \su{0}^\sta{n} }^\mpi  & \Iq  - \rk{ \su{0}^\sta{n} }^\mpi  \su{0}^\sta{n}
 \end{bmatrix}\ek*{\rk{z-\ug}^n\diag \rk*{ \su{0}^\sta{n-1} \rk{ \su{0}^\sta{n-1} }^\mpi , \Iq }} \\
 &\qquad\times
 \begin{bmatrix}
  \Opp & -\su{0}^\sta{n} \\
  -\rk{z-\ug} \rk{ \su{0}^\sta{n} }^\mpi  & \rk{z-\ug} \Iq 
 \end{bmatrix}\\
 & =\rk{z-\ug}^{n+1}\diag \rk*{ \su{0}^\sta{n} \rk{ \su{0}^\sta{n} }^\mpi , \Iq }.
\end{split}\] 
 Thus, the assertion is proved inductively. 
\end{proof}

\bibliography{174arxiv}
\bibliographystyle{abbrv}

\vfill\noindent
\begin{minipage}{0.5\textwidth}
 Universit\"at Leipzig\\
Fakult\"at f\"ur Mathematik und Informatik\\
PF~10~09~20\\
D-04009~Leipzig
\end{minipage}
\begin{minipage}{0.49\textwidth}
 \begin{flushright}
  \texttt{
   fritzsche@math.uni-leipzig.de\\
   kirstein@math.uni-leipzig.de\\
   maedler@math.uni-leipzig.de\\
   to\_schr@web.de
  } 
 \end{flushright}
\end{minipage}

\end{document}